\documentclass[11pt]{preprint}
\usepackage[utf8]{inputenc}
\usepackage[T1]{fontenc}
\usepackage[a4paper, total={160mm,250mm}]{geometry}
\usepackage{amsmath}
\usepackage{amssymb}
\usepackage{hyperref}
\usepackage{amsthm}

\usepackage{xcolor}
\usepackage{bm}
\usepackage[cal=boondox]{mathalfa}

\usepackage{mathtools}
\mathtoolsset{showonlyrefs}

\usepackage{comment}

\excludecomment{comment}

\usepackage{framed}
\usepackage{mathtools}

\newcommand{\half}{\frac{1}{2}}
\newcommand{\thalf}{\tfrac{1}{2}}

\newcommand{\eqdef}{\stackrel{\mathclap{\mbox{\tiny def}}}{=}}
\newcommand{\convd}{\stackrel{\mathclap{\mbox{\tiny d}}}{\rightarrow}}
\newcommand{\convw}{\stackrel{\mathclap{\mbox{\tiny w}}}{\rightarrow}}

\newcommand{\dd}{\mathrm{d}}

\newcommand{\ic}{\mathbf{1}}

\newcommand{\law}{\mathrm{law}}
\newcommand{\supp}{\mathrm{supp}}

\newtheorem{theorem}{Theorem}[section]
\newtheorem{lemma}[theorem]{Lemma}
\newtheorem{proposition}[theorem]{Proposition}
\newtheorem{corollary}[theorem]{Corollary}

\theoremstyle{definition}
\newtheorem{definition}[theorem]{Definition}
\newtheorem{assumption}[theorem]{Assumption}

\theoremstyle{remark}
\newtheorem{remark}[theorem]{Remark}

\newcommand{\bbC}{\mathbb{C}}

\newcommand{\bbE}{\mathbb{E}}

\newcommand{\bbN}{\mathbb{N}}

\newcommand{\bbP}{\mathbb{P}}

\newcommand{\bbR}{\mathbb{R}}

\newcommand{\bbZ}{\mathbb{Z}}

\newcommand{\call}{\mathcal{l}}

\newcommand{\calA}{\mathcal{A}}

\newcommand{\calC}{\mathcal{C}}

\newcommand{\calE}{\mathcal{E}}
\newcommand{\calF}{\mathcal{F}}

\newcommand{\calH}{\mathcal{H}}

\newcommand{\calL}{\mathcal{L}}

\newcommand{\calN}{\mathcal{N}}

\newcommand{\calS}{\mathcal{S}}

\newcommand{\calW}{\mathcal{W}}
\newcommand{\calX}{\mathcal{X}}

\newcommand{\fHat}{\hat{f}}

\newcommand{\kHat}{\hat{k}}

\newcommand{\DHat}{\hat{D}}

\newcommand{\FHat}{\hat{F}}

\newcommand{\RHat}{\hat{R}}

\newcommand{\VHat}{\hat{V}}

\newcommand{\rhoHat}{\hat{\rho}}

\newcommand{\phiHat}{\hat{\phi}}

\newcommand{\psiHat}{\hat{\psi}}

\newcommand{\etaTil}{\tilde{\eta}}

\usepackage[english]{babel}
\usepackage{amsmath,amssymb,latexsym,xcolor}

\newcommand{\assign}{:=}
\newcommand{\cdummy}{\cdot}

\newcommand{\mathd}{\mathrm{d}}

\colorlet{darkgreen}{green!60!black}

\newcommand{\vfock}{L^2(V)}
\newcommand{\efock}{L^2(V^\varepsilon)}
\def\vfocki#1{\Gamma_{#1} L^2(V)}

\def\vsob#1#2{\calH^{#1}_{#2}(V)}
\def\esob#1#2{\calH^{#1}_{#2}(V^\varepsilon)}
\def\vsobi#1#2#3{\Gamma_{#3}\calH^{#1}_{#2}(V)}

\author{Harry Giles, Lukas Gr{\"a}fner}
\title{Construction of the $1d$ Self-repelling Brownian Polymer}
\institute{University of Warwick, CV4 7AL, UK \\ 
\hspace{-0.33cm}\email{harry.giles@warwick.ac.uk, \\ lukas.grafner@warwick.ac.uk}}

\begin{document}
\maketitle

\begin{abstract}

We consider the self-repelling Brownian polymer, introduced in \cite{AmitPariPeli83_AsymptoticBehavior}, which is formally defined as the solution of a singular SDE. The singularity comes from the drift term, which is given by the negative gradient of the local time. We construct a solution of the equation in $d = 1$, give a dynamic characterisation of its law, and show that it is the limiting distribution for a natural family of approximations. In addition, we show that the solution is superdiffusive.

As part of the construction, we consider a singular SPDE which is solved by the (recentered) local time. Using the method of energy solutions, we show that this SPDE is well-posed, and prove that this property can be transferred to the original process. Our results hold for a larger class of drift terms, in which the gradient of the local time is a special case.

\end{abstract}

\bigskip\noindent
{\it Key words and phrases.} Self-repelling motion, local time, singular SDEs, distributional drift, superdiffusivity, Stochastic Partial Differential Equations, energy solutions.
\bigskip

\section{Introduction and main results}
The self-repelling Brownian polymer (SRBP) is an $\bbR^d$-valued process that is repelled by its own local time. Formally, it is given by the solution of the following SDE:
$$ \dd X_t = \dd B_t - \beta^2 \nabla L_t(X_t) \dd t, \qquad X_0 = 0 $$
where $\beta^2 > 0$ is a coupling constant and $L_t(x)$ is the local time density of $X$ at the site $x \in \bbR^d$. Expanding $L_t(x) = \int_0^t \delta(x-X_s) \dd s$, where $\delta$ is the Dirac delta at zero, the equation becomes
\begin{equation}\label{eq:13}
\dd X_t = \dd B_t - \beta^2 \Big(\int_0^t \nabla \delta(X_t- X_s) \dd s \Big) \dd t, \qquad X_0 = 0
\end{equation}
The drift term in \eqref{eq:13} is a highly singular object, in which the gradient of the Dirac delta is evaluated along the history of the path, and as such, the classical solution theory of SDEs cannot be applied. The goal of this paper is to give a meaning to this equation in the case $d = 1$.

The SRBP has its origins in the physics literature \cite{AmitPariPeli83_AsymptoticBehavior} with the introduction of the {\it ``true'' self-avoiding walk} (TSAW), which was presented as a model of polymer statistics\footnote{Named so as to distinguish itself from the more famous {\it self-avoiding walk} \cite{MadrasSlade96_SelfAvoidingWalk}.}. The TSAW evolves as a discrete time nearest neighbour walk $(X_n)_{n = 0}^\infty$ on $\bbZ^d$, and one can think of it as a discrete version  of the SRBP. Upon defining $\ell_n(x)$ to be the local time, $\ell_n(x) = \sum_{i = 0}^n \ic\{X_i = x\}$, a neighbouring site is chosen according to the Boltzmann distribution
$$ \bbP(X_{n+1} = x | X_0, \ldots, X_n ) = \frac{e^{- \beta^2 (\ell_n(x) - \ell_n(X_n))}}{\sum_{x \sim X_n} e^{- \beta^2 (\ell_n(x) - \ell_n(X_n))}} \, .$$
In fact, the SRBP, as written in \eqref{eq:13}, is also given in this work, albeit non-rigorously (see \cite[Equation (4.1)]{AmitPariPeli83_AsymptoticBehavior}). Shortly after, the SRBP appeared in the probability literature \cite{NorrRogeWill87_SelfavoidingRandom}.

A number of results concerning the SRBP have previously been established, with a particular interest in the large scale behaviour of the process. Whether the behaviour is diffusive or superdiffusive, is known to be dimension dependent. A non-rigorous scaling argument (see the appendix of \cite{TothValko12_SuperdiffusiveBounds}) leads to the following conjecture
\begin{align*}
\bbE[|X_t|^2] \sim
\begin{cases}
t^{4/3} & d = 1 \\
t (\log t)^{1/2} & d = 2 \\
t & d \geqslant 3
\end{cases}
\end{align*}
The case $d \geqslant 3$ was treated in \cite{HorvTothVeto12_DiffusiveLimits}, and the case $d = 2$ has been more recently studied in \cite{CannizzarGiles25_InvariancePrinciple} in the so-called weak-coupling scaling. As for the case $d = 1$, the best results are only bounds \cite{TarrTothValk12_DiffusivityBounds}, which show $t^{5/4} \lesssim \bbE[|X_t|^2] \lesssim t^{3/2}$, in the Tauberian sense\footnote{The result holds at the level of the Laplace transform. See Theorem \ref{thm:super} of the present work.}. In the $d = 1$ case, under the correct superdiffusive scaling, one expects to see convergence to the ``true'' self-repelling motion \cite{TothWerner98_TrueSelfrepelling}. For more details on the history of the process, see \cite{HorvTothVeto12_DiffusiveLimits,TarrTothValk12_DiffusivityBounds,CannizzarGiles25_InvariancePrinciple} and the references therein.

In all previous works, the model that is considered is not \eqref{eq:13}, but the equation in which the Dirac delta is replaced by a smooth function. This is undesirable as it destroys the Markov property (in the spatial variable) for the local time $L_t(x)$ of the process $X$. The question of \textit{locality} is significant, as it would allow for the use of Ray-Knight type theorems to prove convergence to the ``true'' self-repelling motion of \cite{TothWerner98_TrueSelfrepelling}. This was first achieved in \cite{Toth95_TrueSelfAvoiding}, in which the TSAW is considered (with bond repulsion) and the Markov property is shown to hold. This idea was further explored in the recent work \cite{KosyginPeterso25_ConvergenceRescaled}, and we expect their results to generalise to the present case.

As mentioned above, the difficulty in making sense of \eqref{eq:13} lies in the fact that $\nabla \delta$ is a distribution. Singular SDEs with distributional drift constitute a very active and fruitful area of research \cite{DelarueDiel16,CatellierGubinelli2016,FlandoliIssoglioRusso17,ZhangZhao17,CannizzaroChouk18,
CannizzaroHaunschmidSibitzToninelli21,HarangPerkowski2021,Kremp2023,
ChatzigeorgiouEtAl25_GaussianFreefield,HaoZhang23,Armstrong2024,
GraefnerPerkowski24,DareiotisGerencerLeLing24,ButkovskyMytnik24,GaleatiGerencser2025}. Compared to the above works, an additional layer of difficulty comes from the path-dependency of the drift term in \eqref{eq:13}, as this destroys the Markov property. Nonetheless, there is some literature in this direction. In \cite{FlandoliRussoWolf04,OhashiRussoTeixeira20}, they consider the case of a function-valued path-dependent drift superpositioned with a distributional drift that is not path-dependent. We also refer to \cite{HuangWang23,WangYuanZhao25}, who consider a function-valued drift that jointly depends on the past and the law of the process. Neither case applies for \eqref{eq:13}. 

A possible approach to solving \eqref{eq:13} would be to decouple the SDE and the self-interaction. First, for a fixed Brownian path, one could consider solving the following ODE in a pathwise sense, for a large class of drift terms $b$:
\begin{equation}\label{consistentSDE}
\mathd X_t=b(t,X_t)\mathd t+\mathd B_t
\end{equation}
At this point, one could use a fixed-point argument to impose the condition $b=\nabla L_t$. Pathwise methods, as in \cite{CatellierGubinelli2016,GaleatiGubinelli2022,GaleatiGerencser2025}, may be applied to \eqref{consistentSDE} in the case of Brownian noise provided that the drift term has non-negative regularity, $\mathcal{C}^{\alpha}$ for some $\alpha \ge 0$, where $\mathcal{C}^{\alpha}$ is the space of $\alpha$-Hölder continuous functions on $\mathbb{R}$. However, in our case, it is reasonable to assume that the regularity of the local time, $b=\nabla L_t$, is not better than that of Brownian motion, $L_t\in \mathcal{C}^{1/2-}$ (see Corollary 1.8 from §1 in Chapter VI of \cite{RevouzYor1991}). Therefore, to the best of our understanding, these methods cannot be applied. 

We conclude that, to the best of our knowledge, the SRBP with the Dirac interaction, as written in \eqref{eq:13} lies outside of the scope of existing solution theories for singular SDEs and thus deserves a different analysis.

\subsection{The model and main results concerning the polymer}\label{sec:model-main-results}
In this work, we give a meaning to the SRBP, not just in the case of Dirac interaction, but for a larger class of singular interaction functions. Namely, writing $\calS(\bbR)$ for the space of Schwartz functions, with the dual space $\calS'(\bbR)$ of tempered distributions, we consider interaction functions $V \in \calS'(\bbR)$ that are both positive definite and uniformly locally integrable in the following sense.
\begin{assumption}\label{ass:posdef}
The Fourier transform exists as a non-negative function, $\VHat \in L^1_{\text{loc}}(\bbR), \VHat \geqslant 0$. Moreover, $\VHat$ satisfies:
\begin{equation}\label{eq:63}
\sup_{z \in \bbZ} \int_{z}^{z+1} \VHat(p) \dd p < \infty
\end{equation}
\end{assumption}

\begin{remark}
The assumption of positive definiteness is standard and also taken in previous works \cite{HorvTothVeto12_DiffusiveLimits,TarrTothValk12_DiffusivityBounds,CannizzarGiles25_InvariancePrinciple}, as it is related to the existence of a stationary measure for the so-called ``environment process'' (see Lemma \ref{lem:2}). The uniform local integrability condition, on the other hand, is the minimal assumption that we can place under which our method will work (see Remark \ref{rmk:2}).

The canonical choice is to take $V = \delta$, in which case $\VHat(p) = 1$ and Assumption \ref{ass:posdef} is clearly satisfied. This is the most physically interesting example, originating from the physics literature \cite{AmitPariPeli83_AsymptoticBehavior}, as previously mentioned. For a more exotic example, one may take $V$ to be any positive definite measure of finite mass. 
\end{remark}

For $V \in \calS'(\bbR)$ satisfying Assumption \ref{ass:posdef}, $\omega : \bbR \rightarrow \bbR$, and $\beta \geqslant 0$, we say that the the \textit{self-repelling Brownian polymer with interaction $V$, coupling constant $\beta \geqslant 0$, in the environment $\omega$}, is the solution of the following SDE
\begin{equation}\label{eq:3}
\dd X_t = \dd B_t - \beta \omega(X_t) \dd t - \beta^2 \Big( \int_0^t \nabla V (X_t-X_s) \dd s\Big) \dd t, \qquad X_0 = 0
\end{equation}
The element $V \in \calS'(\bbR)$ governs the nature and strength of the self-interaction, and an additional drift is provided by the field $\omega$. In this work, we shall consider a setting in which the environment $\omega$ is drawn at random and independently of the Brownian motion $(B_t)_{t \geqslant 0}$ driving the SDE. In this regard, one may think of $(X_t)_{t \geqslant 0}$ as being an SRBP in a random environment.

The reason to introduce the environment $\omega$ is technical. In particular, we assume $\law(\omega) \ll \pi$ is absolutely continuous with respect to a certain reference measure $\pi$ on $\calS'(\bbR)$. The reference measure $\pi = \pi(V)$ that we consider is such that $(\omega(x) : x \in \bbR)$ is a centred Gaussian process with covariance $\bbE[\omega(x) \omega(y)] = V(x - y)$; see Section \ref{sec:notat-wien-space} for more details. In particular, one sees the need for the positive definite assumption on $V$. This is a natural choice, as it puts the environment process into a time-stationary regime (see the comment just before Lemma \ref{lem:2}). We emphasise that our results do not cover the case $\omega \equiv 0$. In all recent works on the SRBP, \cite{TarrTothValk12_DiffusivityBounds,CannizzarGiles25_InvariancePrinciple,HorvTothVeto12_DiffusiveLimits}, it is assumed $\law(\omega) = \pi$.

We write $\text{SRBP}(V, \beta, \omega)$ for the law of the solution $(X_t)_{t \geqslant 0}$ on the path space $C(\bbR_{\geqslant 0}, \bbR)$ for a deterministic field $\omega$. For a distribution $\pi_0 \ll \pi$, we define the annealed law $\text{SRBP}(V, \beta, \pi_0)$ on the path space $C(\bbR_{\geqslant 0}, \bbR)$ according to
$$ \text{SRBP}(V, \beta, \pi_0) := \int \text{SRBP}(V, \beta, \omega) \pi_0( \dd \omega) $$
We emphasise that $\text{SRBP}(V, \beta, \pi_0)$ is a priori only well defined when $\pi_0$ and $V$ are sufficiently smooth such that the classical theory of SDEs yields a strong solution to \eqref{eq:3}, almost surely with respect to the environment measure $\pi_0$.  We refer to this as {\it the smooth setting}, and it is described in more detail in Section \ref{sec:mollified-srbp}.

We wish to consider a more general case, $V \in \calS'(\bbR)$ satisfying no condition of regularity beyond Assumption \ref{ass:posdef}. To give meaning to \eqref{eq:3}, a natural strategy is to mollify the distribution. Choosing $\rho \in \calS(\bbR)$ to be a positive definite mollifier of unit mass, $\rhoHat \geqslant 0$, $\int \rho(x) \dd x = 1$, one may consider $V^\varepsilon := \rho_\varepsilon * \rho_\varepsilon * V$, in which $\rho_\varepsilon$ is the ``sharpened'' mollifier given according to $\rho_\varepsilon(x) = \varepsilon^{-1} \rho(\varepsilon^{-1} x)$ for some $\varepsilon > 0$. The reason for applying the mollifier twice is because if $\omega \sim \pi(V)$, then it holds that $\omega^\varepsilon := \rho_\varepsilon * \omega$ is distributed according to $\pi(V^\varepsilon)$. In general, we write $\pi^\varepsilon_0$ for the law of $\omega^\varepsilon$ when $\omega \sim \pi_0$. Since $V^\varepsilon \in C^\infty(\bbR)$, and $\omega^\varepsilon \in C^\infty(\bbR)$ almost surely with respect to $\pi^\varepsilon_0$, this puts us into the smooth setting under which $\text{SRBP}(V^\varepsilon, \beta, \pi^\varepsilon_0)$ is a priori well defined. One may hope that as the mollification is removed, $\varepsilon \rightarrow 0$, the process $(X^\varepsilon_t)_{t \geqslant 0}$ converges in distribution to some limiting process, $(X_t)_{t \geqslant 0}$. Indeed, this is the content of one of our main theorems (for the full version, see Theorem \ref{thm:Xconvergence}):
\begin{theorem}\label{thm:main}
Let $V \in \calS'(\bbR)$ satisfy Assumption \ref{ass:posdef}, and let $\pi_0 \ll \pi(V)$ be such that $\dd \pi_0 / \dd \pi \in L^2(V)$. Then, the annealed law $\text{SRBP}(V^\varepsilon, \beta, \pi^\varepsilon_0)$ converges as $\varepsilon \rightarrow 0$ to a limiting distribution that is independent of the choice of mollifier $\rho$. We call this limiting distribution $\text{SRBP}(V, \beta, \pi_0)$.
\end{theorem}

It is natural to ask about properties of the limiting distribution $\text{SRBP}(V, \beta, \pi_0)$. Using the techniques of \cite{TarrTothValk12_DiffusivityBounds}, we show that the limiting measure $\text{SRBP}(V, \beta, \pi_0)$ in the stationary case $\pi = \pi_0$ satisfies superdiffusive bounds as previously shown in the smooth case. The result depends on the infrared behaviour of $V$ and under the canonical choice $V = \delta$, we have the aforementioned bound, $t^{5/4} \lesssim \bbE[|X_t|^2] \lesssim t^{3/2}$:
\begin{theorem}\label{thm:super}
Let $V\in \calS'(\bbR)$ satisfy Assumption \ref{ass:posdef} and suppose that there exists $\alpha \in (-1, 1)$ such that
\begin{equation}\label{eq:42}
0 < \liminf_{p \rightarrow 0} |p|^{-\alpha} \VHat(p) \leqslant \limsup_{p \rightarrow 0} |p|^{-\alpha} \VHat(p) < \infty
\end{equation}
Let $X \sim \text{SRBP}(V, \beta, \pi)$, then in the Tauberian sense, it holds that as $t \rightarrow \infty$,
$$ t^{(5 - 2\alpha + \alpha^2)/4} \lesssim  \bbE[ X_t^2 ] \lesssim t^{(3-\alpha)/2} $$
This is to say that the conclusion holds at the level of the Laplace transform: for $\DHat(\lambda) := \int_0^\infty e^{-\lambda t} \bbE[ X_t^2 ] \dd t$, it holds that as $\lambda \rightarrow 0$
\begin{equation*}
\lambda^{-(9 - 2\alpha + \alpha^2)/4} \lesssim  \DHat(\lambda) \lesssim \lambda^{-(5-\alpha)/2}
\end{equation*}
\end{theorem}

One further contribution of this paper is a characterisation of the limiting law, $\text{SRBP}(V, \beta, \pi_0)$. Before we describe this characterisation, we must introduce the environment process.

\subsection{The environment process and the law of the SRBP}
Although the process $(X_t)_{t \geqslant 0}$ given by \eqref{eq:13} is not a Markov process on $\bbR$, the coupled process $(X_t, L_t)$ is a Markov process, but on a larger state space, $\bbR \times \calS'(\bbR)$. This allows us to use Markov process theory. It turns out to be more convenient to put the coupled process together into a single $\calS'(\bbR)$-valued Markov process, known as {\it the environment seen by the particle}, or simply {\it the environment process}. Introducing this process is a standard tool, originating in the theory of random walks in random environments. Its definition, in the case of general interaction function $V \in \calS'(\bbR)$, is given according to
\begin{equation}\label{eq:51}
\eta_t(x) := \omega(x+X_t) + \beta \Big(\int_0^t \nabla V(x+X_t-X_s) \dd s\Big)
\end{equation}

The process $\eta$ is only a priori defined in the smooth setting. In fact, in the smooth setting, one can show that $(\eta_t)_{t \geqslant 0}$ is supported on the space of smooth fields, $\calW$, whose definition is given in \eqref{eq:7}. The SRBP can be written in terms of both the Brownian motion and the environment in the following way:
\begin{equation}\label{eq:1}
X_t = B_t - \beta \int_0^t f(\eta_s) \dd s
\end{equation}
where $f \colon \calW\to\bbR$ is given by $f(\omega) = \omega(0)$. In particular, we can study $X$ via $\eta$, and much of the analysis of this paper is done at the level of $\eta$.

An application of Itô's formula (in the smooth setting) yields that $(\eta_t)_{t \geqslant 0}$ is a solution of the following SPDE:
\begin{equation}\label{eq:spde}
\dd \eta_t(x) = \thalf \Delta \eta_t(x) \dd t + \nabla \eta_t(x) \dd B_t - \beta \Big( \nabla \eta_t(x) \eta_t(0) - \nabla V(x)\Big) \dd t, \quad \eta_0=\omega
\end{equation}
To give a meaning to $X$ in the singular case, $V \in \calS'(\bbR)$, we begin by giving a meaning to \eqref{eq:spde}. Then we will prove Theorem \ref{thm:main} by showing that the identity \eqref{eq:1} carries through.

We do not expect to find a solution $(\eta_t)_{t \geqslant 0}$ of \eqref{eq:spde} taking place in a space of smooth fields. Indeed, under the canonical choice $V = \delta$, the equation depends additively on the singular term $\nabla\delta$. Granting that we ought to look for an $\calS'(\bbR)$-valued solution, the right hand side is a priori meaningless: it involves the product $\nabla \omega(x) \omega(0)$, which is not well-defined for $\omega \in \calS'(\bbR)$. In other words, not only is there no classical way to construct a solution for \eqref{eq:spde}, but given an $\calS'(\bbR)$-valued process $(\eta_t)_{t \geqslant 0}$, one cannot even compute the RHS of the equation (even in a weak sense) and thereby check whether it is satisfied.

We approach the problem from the perspective of \textit{Energy Solutions}. Energy Solutions for singular SPDEs were first introduced in \cite{GoncalvesJara14_NonlinearFluctuations} and later refined in \cite{GubinelliJara13}, however uniqueness of such solutions remained open until the work \cite{GubinelPerkows18_EnergySolutions}. In \cite{GubinelPerkows20_InfinitesimalGenerator} a more general approach towards uniqueness, based on the analysis of the generator of the underlying equation, was introduced and consequently applied in \cite{GubinelliTurra20,LuoZhu21}. This approach was then further developed and now allows to establish weak well-posedness of energy solutions for certain scaling-critical or -supercritical SPDEs/SDEs \cite{Graefner21,GraefnerPerkowski24,Graefner24,GrafPerkPopa24_EnergySolutions,Graefner25}.

Rescaling \eqref{eq:spde} diffusively in the case $V=\delta_0$, reveals that it is subcritical in the sense of pathwise approaches to singular SPDEs such as \textit{regularity structures} \cite{Hairer2014} and \textit{paracontrolled calculus} \cite{Gubinelli2015Paracontrolled}, making it, in principle, amenable to such techniques. However, due to the unusual structure of the singular non-linearity $\nabla \eta \eta (0)$ and the presence of the rough transport noise term $\nabla\eta\mathd B$, the solvability in this sense does not seem to be a trivial consequence of existing works. Moreover, let us point out three advantages of energy solutions compared to these pathwise methods. First, the pathwise techniques are limited to subcritical equations, whereas energy solutions can treat some scaling-critical or even -supercritical examples, as cited above. This becomes relevant for the SRBP in $d\geqslant 2$ (see Remark \ref{rmk:criticality}). Second, the approach via energy solutions is probabilistic which makes it more natural to derive properties of the law of the process. Finally, the characterisation of solutions is easier to implement from the point of view of discrete systems, see \cite{GoncalvesJara14_NonlinearFluctuations} and followup works, and would be a natural choice if one wished to show convergence, say, of the ``true'' self-avoiding walk to the SRBP under the so-called weak-coupling scaling.

In this paper, we introduce a notion of energy solutions in Definition \ref{def:energysolution}. In short, we require $(\eta_t)_{t \geqslant 0}$ to solve the martingale problem associated with the SPDE \eqref{eq:spde} in the case of only linear and quadratic functionals of the environment. To make sense of the RHS of \eqref{eq:spde}, we also impose an integrability condition for additive functionals of the path, known as the Itô trick. See Section \ref{sec:definition-ito-trick} for further explanation. We show that this notion is well posed, and that it can be obtained as a limit of the mollified problem:

\begin{theorem}\label{thm:env}
Under the assumption of Theorem \ref{thm:main}, let $\text{ENV}(V^\varepsilon, \beta, \pi^\varepsilon_0)$ be the law on the path space $C(\bbR_{\geqslant 0}, \calS'(\bbR))$ of the environment process, $(\eta^\varepsilon_t)_{t \geqslant 0}$ under the choice of interaction $V^\varepsilon$ and environment distribution $\pi^\varepsilon_0$. Then $\text{ENV}(V^\varepsilon, \beta, \pi^\varepsilon_0)$ converges as $\varepsilon \rightarrow 0$ to a limiting law that is independent of the choice of $\rho$. The limit, which we call $\text{ENV}(V, \beta, \pi_0)$, satisfies the following properties.
\begin{itemize}
\item It holds that $\eta \sim \text{ENV}(V, \beta, \pi_0)$ is a Markov process.
\item The choice $\pi_0 = \pi$ is time-stationary: if $\eta \sim \text{ENV}(V, \beta, \pi)$, then $\eta_t \sim \pi$ for all $t \geqslant 0$.
\item The measure $\text{ENV}(V, \beta, \pi_0)$ is the unique distribution satisfying the notion of energy solution, as given in Definition \ref{def:energysolution}.
\end{itemize}
\end{theorem}

The benefit of this characterisation is that one only needs to check the first two moments of a process $(\eta_t)_{t \geqslant 0}$ and also an integrability condition to verify whether its law is given by $\text{ENV}(V, \beta, \pi_0)$. For the statement of uniqueness in Theorem \ref{thm:env}, our proof relies on a new PDE-version of the semigroup-based arguments from the scaling-critical setting of \cite{GrafPerkPopa24_EnergySolutions} which was developed in \cite{Graefner25}. Let us also point out that the equation studied in the present work differs in structure from the above examples that were previously treated with the energy solutions approach. Indeed, \eqref{eq:spde} contains a transport-noise term in contrast to the ``usual'' additive noise which results in a much more degenerate regularising part $\calL_0$ of the generator (see \eqref{eq:6}). Also, the singular term $\nabla\eta\, \eta(0)$ is structurally different from multiplication-type singularities that one typically considers in singular SPDEs.

\medskip

Our characterisation of the law $\text{SRBP}(V, \beta, \pi_0)$ is given in terms of $\text{ENV}(V, \beta, \pi_0)$. More precisely, we say that a $C(\bbR_{\geqslant 0}, \bbR)$-valued random variable $X$ is an \textit{energy solution} to the SDE \eqref{eq:3} if $X_0=0$ and the $\calS'(\bbR)$-valued process $\eta$ defined according to \eqref{eq:51} is an energy solution to \eqref{eq:spde}.  We obtain (see Theorem \ref{wellpsoednessenergysolutionsX} for the full version):
\begin{theorem}\label{thm:main2}
Under the assumptions of Theorem \ref{thm:main}, further assume that for some $\alpha \in (-1, \infty),\alpha^\ast\in (-\infty,\infty)$ such that $\alpha^\ast<2+\alpha$, it holds that
\begin{equation*}
0 < \liminf_{p \rightarrow 0} |p|^{-\alpha^\ast} \VHat(p) \leqslant \limsup_{p \rightarrow 0} |p|^{-\alpha} \VHat(p) < \infty.
\end{equation*}
Then it holds that $\text{SRBP}(V, \beta, \pi_0)$ is the unique energy solution to the SDE \eqref{eq:3}.
\end{theorem}
The proof of this result relies on the idea that the maps \eqref{eq:1} and \eqref{eq:51} are essentially inverses of each other so that the uniqueness of $\text{ENV}(V, \beta, \pi_0)$ translates to the uniqueness of $\text{SRBP}(V, \beta, \pi_0)$. Making this precise is non-trivial due to the singularity of our setting and the corresponding arguments are new, to the best of our knowledge.

\paragraph{Structure of the paper.} In Section \ref{sec:background}, we give some background on Wiener space analysis, and recap the theory of the SRBP in the smooth setting. In Section \ref{sec:constr-limit-gener} we define the operator $\calL$ that is a natural candidate for the generator of the limiting environment $(\eta_t)_{t \geqslant 0}$, and we give some a priori estimates. Following this, we give the notion of energy solution for $\calL$, and an additional notion of solution which we refer to as the cylinder function martingale problem. We finish this section by showing that energy solutions also solve the cylinder function martingale problem. In Section \ref{sec:existenceenergy}, we show that subsequential limits of the mollified process are energy solutions, thereby settling the problem of existence. In Section \ref{sec:uniq-energy-solut}, we prove uniqueness, from which it follows that the mollified process converges fully (not just along subsequences). Finally, in Section \ref{sec:limit-srbp}, we translate our results on the environment $\eta$ to the process $X$: we give a notion of energy solution for $\text{SRBP}(V, \beta, \pi_0)$; we show that this notion is unique, and that that the mollified SRBP $X^\varepsilon$ converges to it. We finish this section by giving a superdiffusivity result for $\text{SRBP}(\beta, V, \pi)$.

\section{Background}
\label{sec:background}
In this section, we begin by recalling some facts about Wiener space analysis, then we recapitulate results for the smooth setting.

\subsection{Notation and Wiener space analysis}\label{sec:notat-wien-space}
For an element $V \in \calS'(\bbR)$ we define the Fourier transform according to $\VHat(p) := \int_{\bbR} e^{-\iota p x} V(x) \dd x$. In general, this defines a distribution $\VHat \in \calS'(\bbR)$, but we shall only consider $V \in \calS'(\bbR)$ for which $\VHat \in L^1_{\text{loc}}(\bbR)$.

\paragraph{The Gaussian measure $\pi$.} We let $\calF = \sigma(\omega(\phi) : \phi \in \calS)$ be the cylindrical $\sigma$-algebra on $\calS'(\bbR)$. For $V \in \calS'(\bbR)$ satisfying Assumption \ref{ass:posdef}, let $\pi = \pi(V)$ be the unique measure on $(\calS'(\bbR), \calF)$ such that $\{ \omega(\phi) : \phi \in \calS \}$ is a mean zero Gaussian process with the following covariance structure
\begin{equation}\label{eq:5}
\int\omega(\phi)\omega(\psi)\pi(\dd \omega)= \int \VHat(p) \phiHat(p) \overline{\psiHat(p)} \dd p, \,\qquad \phi, \psi \in \calS
\end{equation}
The existence of such a measure is given by the Bochner-Minlos theorem, see \cite{Hida80_BrownianMotion}. For $p \geqslant 1$ we write $L^p(V)$ for the space $L^p(\pi(V), \calF)$.

\paragraph{Cylinder functions.} We say that a function $f : \calS'(\bbR) \rightarrow \bbR$ is a cylinder function if it is of the form
\begin{equation}\label{eq:16}
f(\omega) = f(\omega(\phi_1),\ldots,\omega(\phi_k))
\end{equation}
for test functions $\phi_i \in \calS(\bbR)$, and measurable function $f : \bbR^k \rightarrow \bbR$. Note that we abuse notation by writing $f$ for either the function on $\calS'(\bbR)$ or the function on $\bbR^k$. We let $\calC$ be the vector space of cylinder functions in which $f \in \calS(\bbR^k)$ is a Schwartz function, and $\calC_p$ for the case when $f \in C^\infty(\bbR^k)$ is of at most polynomial growth, and all of the derivatives are also of at most polynomial growth. It is clear that $\calC \subset \calC_p \subset L^p(V)$ for all $p \geqslant 1$, and furthermore $\calC \subset L^\infty(V)$.

Similarly, we say that $f : \bbR \times \calS'(\bbR) \rightarrow \bbR$ is a time-dependent cylinder function if it is of the form $f(t, \omega) = f(t, \omega(\phi_1),\ldots,\omega(\phi_k))$, for $f : \bbR^{k+1} \rightarrow \bbR$, and we write $t\calC$ (resp. $t\calC_p$) for the case in which $f \in \calS(\bbR^{k+1})$ (resp. $f \in C^\infty(\bbR^{k+1})$ is of at most polynomial growth, and all of the derivatives are also of at most polynomial growth).

\paragraph{Index notation.} We write $1{:}n$ for the set $\{1, \ldots, n\}$. For $(p_1, \ldots, p_n) \in \bbR^n$, and for a subset $A \subset 1{:}n$, we define $ p_A := (p_i : i \in A)$ and $p_{[A]} = \sum_{i \in A} p_i$.

\paragraph{Chaos decomposition.} We consider the chaos decomposition $L^2(V) = \oplus_{n=0}^\infty H_n$. We use the convention that an infinite direct sum of Hilbert spaces is the closure of finite linear combinations. Roughly speaking, $H_n \subset L^2(V)$ is the subspace of degree $n$ polynomials in the Gaussian random variables after having orthogonalised out all lower degree polynomials. We give a brief summary here; see \cite{Janson97_GaussianHilbert} for the full details.

Elements of $H_n$ can be represented via their ``kernels''. That is, $H_n \cong K_n$ where $K_n$ is an appropriate space of symmetric functions, via the map
\begin{equation}\label{eq:44}
K_n \rightarrow H_n, \qquad k \mapsto \int_{\bbR^n} k(x_{1:n}) \, {:} \omega(x_1) ... \omega(x_n) {:} \, \dd x_{1:n}
\end{equation}
in which we have written ${:} \xi_1 \ldots \xi_n {:}$ for the Wick product of Gaussian random variables $\xi_1, \ldots, \xi_n$. In general, $K_n$ will contain elements $k \in K_n$ that are not symmetric functions, but distributions. We consider a further space $\vfocki{n}$, which we refer to as the Fock space of degree $n$, which is obtained by taking the Fourier transform $k \mapsto \kHat$. The space $\vfocki{n}$ is a Hilbert space of symmetric functions, and we describe it now in a little more detail.

Let $\calE_n$ be the vector space of measurable functions $\psi : \bbR^n \rightarrow \bbC$ that are symmetric, and satisfy $\psi(-p) = \overline{\psi(p)}$. We define the measure space $(\bbR^n, \mu_n)$ according to $\mu_n(\dd p_{1:n}) = n! \big(\prod_{i=1}^n \VHat(p_i)\big) \dd p_{1:n}$. Then it holds that
$$ \vfocki{n} = \{\psi \in \calE_n : \lVert \psi \rVert_{L^2(\mu_n)} < \infty\} $$
In particular, the inner product for $\psi,\psi^\prime \in \vfocki{n}$ is given by
$$ \langle\psi,\psi^\prime\rangle = n! \int \big(\prod_{i=1}^n \VHat(p_i)\big) \psi(p_{1:n}) \overline{\psi^\prime(p_{1:n})} \dd p_{1:n} $$

The isomorphism described above, in which $H_n \cong K_n \cong \vfocki{n}$, ensures that $L^2(V) \cong \oplus_{n=0}^\infty \vfocki{n}$. That is to say, we may identify any element $f \in L^2(V)$ with its chaos decomposition $(\psi_n)_{n \geqslant 0}$. Note that if $f \in \calS$ is a cylinder function, then the kernels $\psi_n$ belong to $\calS(\bbR^n, \bbC)$.

\paragraph{Definition of $\calL_0$ and $\calN$.} We define two important self-adjoint operators by specifying their action on $\calE_n$ for arbitrary $n \in \bbN$. Let $\calN : \calE_n \rightarrow \calE_n$ be multiplication by $n$, and define $\calL_0 : \calE_n \rightarrow \calE_n$ according to
\begin{equation}\label{eq:6}
\calL_0 \psi(p_{1:n}) = -\thalf |p_{[1:n]}|^2 \psi(p_{1:n})
\end{equation}
For completeness, let $\calN, \calL_0$ annihilate $\calE_0$. They are both multiplication operators, and so it is clear that they commute. Note that $\calL_0$ is given by the square of the sum, not the sum of the squares. In other words, $\calL_0$ is {\it not} the second quantized Laplacian.

The operator $\calL_0$ is the generator of the Stochastic Linear Transport Equation\footnote{Also known as the {\it diffusion in random scenery.}}:
\begin{equation}\label{eq:14}
\dd \eta_t(x) = \thalf \Delta \eta_t(x) \dd t + \nabla \eta_t(x) \dd B_t
\end{equation}
The solution of equation \eqref{eq:14} is simply given by the transport $\eta_t(x) = \eta_0(x + B_t)$. The transport nature of this equation can also be emphasised by writing the RHS in Stratonivich form, in which case it becomes $\nabla \eta_t(x) \circ \dd B_t$.

We defined these operators on the Fock space, but their action on cylinder functions can also be given explicitly. Indeed, in the case of $\calL_0$, applying Itô's formula on \eqref{eq:14} to a nice enough cylinder function $f$ of the form \eqref{eq:16}, yields
\begin{equation}\label{eq:47}
\calL_0 f = \thalf \sum_{i = 1}^k \partial_i f(\omega(\phi)) \omega(\Delta\phi_i) + \sum_{i,j = 1}^k \partial_{ij}^2 f(\omega(\phi)) \omega(\nabla\phi_i)\omega(\nabla\phi_j)
\end{equation}
As for $\calN$, we have the identity $\calN f = \delta(D f)$, where $D$ and $\delta$ are the Malliavin derivative and divergence respectively \cite[Theorem 15.145]{Janson97_GaussianHilbert}. In particular, this gives the explicit expression
$$ \calN f = \sum_{i = 1}^k \partial_if(\omega(\phi)) \omega(\phi_i) + \sum_{ij = 1}^k \partial^2_{ij} f(\omega(\phi)) \int \VHat(p) \phiHat_i(p) \phiHat_j(p) \dd p$$

From these expressions we conclude that $\calN, \calL_0$ are closed on $\calC_p$, and in fact $\calN$ is closed on the smaller space $\calC$. Furthermore, the action of $\calL_0$ does not depend on the choice of $V$, whereas the action of $\calN$ does.

\paragraph{Definition of the Sobolev spaces $\vsob{s}{\alpha}$.} We define the spaces $\vsob{s}{\alpha}$ for $s, \alpha \in \bbR$ which measure the regularity of kernels in terms of the operators $\calN, \calL_0$. We define, for $n \in \bbN$,
\begin{align*}
\vsobi{s}{\alpha}{n} &:= \{ \psi \in \calE_n : \lVert \psi \rVert_{\vsobi{s}{\alpha}{n}} < \infty\} \\
\lVert \psi \rVert_{\vsobi{s}{\alpha}{n}} &:= \lVert (1 - \calL_0)^{s/2} (1 + \calN)^{\alpha/2} \psi \rVert_{\vfocki{n}}
\end{align*}
Finally, we take $\vsob{s}{\alpha} = \oplus_{n = 0}^\infty \vsobi{s}{\alpha}{n}$. We have the duality relation $(\vsob{s}{\alpha})^* = \vsob{-s}{-\alpha}$. It holds that $\calC \subset \vsob{s}{\alpha}$ is dense. Similarly, for all $T > 0$, it holds that $t\calC \subset L^2_T\vsob{s}{\alpha}$ is dense.

\paragraph{The operator $\nabla_0$.} Let $\nabla_0 : \vfock \rightarrow \vfock$ be the unbounded, skew-symmetric operator whose action on $\vfocki{n}$ is given according to $\nabla_0 \psi(p_{1:n}) := \iota p_{[1:n]} \psi(p_{1:n})$. The action on cylinder functions $f$ of the form \eqref{eq:16} is given by $\nabla_0 f = \sum_{i = 1}^n \partial_i f(\omega(\phi)) \omega(\nabla\phi)$, so that it is plain to see that $\nabla_0$ is closed on $\calC_p$. Moreover, the following integration by parts formula holds: for any $f : \bbN \rightarrow \bbN$,
\begin{equation}\label{eq:10}
\lVert f(\calN) \nabla_0 u \rVert_{\vfock} = 2 \lVert f(\calN) (-\calL_0)^{1/2} u \rVert_{\vfock}
\end{equation}

\subsection{The smooth SRBP}\label{sec:mollified-srbp}
When the interaction function is smooth, $V \in C^\infty(\bbR)$, the SRBP is classically well defined. In this subsection, we give a summary of ideas; further details can be found in \cite{TarrTothValk12_DiffusivityBounds}.

Let $\calW$ be the Fréchet space of fields
\begin{equation}\label{eq:7}
\calW := \Big\{ \omega \in C^\infty(\bbR, \bbR) : \lVert \omega \rVert_{k,m}  < \infty, \, \forall k,m \in \bbN \Big\}
\end{equation}
where the seminorms are defined as $\lVert \omega \rVert_{k,m} := \sup_{x} (1+|x|)^{-1/m} | \partial^k \omega(x) |$. For $\omega\in\calW$, $V\in C^\infty(\bbR)$ satisfying Assumption \ref{ass:posdef}, and $\beta \in \bbR$, we define $(X_t)_{t\geqslant 0}$ to be the strong solution of the SDE \eqref{eq:3}. Let $(\eta_t)_{t\geqslant 0}$ be the $\calW$-valued environment process given according to \eqref{eq:51}.

\begin{lemma}
The environment process $(\eta_t)_{t \geqslant 0}$ solves the SPDE \eqref{eq:spde}; in particular, it is a $\calW$-valued Markov process. The action of the generator $\calL$ on cylinder functions $f \in \calC$ of the form $f = f(\eta(\phi_1), \ldots, \eta(\phi_k))$ is given according to $\calL f = \calL_0 f + \calA f$, where $\calL_0 f$ is the generator of the Stochastic Linear Transport equation \eqref{eq:47}, and $\calA$, due to the self-interaction, is given according to
$$ \calA f = \sum_{i = 1}^k \partial_i f(\omega(\phi_i)) \big(\omega(\nabla\phi) \omega(0) - V(\nabla\phi_i) \big) $$
\end{lemma}
In the above, $\calA f$ belongs to neither $\calC$ nor even $\calC_p$, because it involves the term $\omega(0)$.

\begin{remark}\label{rmk:1}
In this smooth case, it can be shown, using Kolmogorov continuity, that $\pi(V)$ is concentrated on $\calW$. We abuse notation by writing $\pi(V)$ for the measure on $\calW$, as well as the measure on $\calS'(\bbR)$. 
\end{remark}

The reason we are interested in the measure $\pi(V)$ is because it is a stationary measure for the environment process $\eta$. This is summarised in the following lemma, along with other known facts about $\eta$. The proof is mostly an application of Itô's formula, and we refer the reader to \cite{HorvTothVeto12_DiffusiveLimits} for more details.

\begin{lemma}\label{lem:2}
The measure $\pi = \pi(V)$ is invariant for the SPDE \eqref{eq:spde}. The generator $\calL$, viewed as an unbounded operator on $\vfock$, can be decomposed as $\calL = \calL_0 + \calA_+ + \calA_-$, where the self-adjoint operator $\calL_0$ is given in \eqref{eq:6}, and for $n\geqslant 1$, the operators $\calA_\pm$ map $\vfocki{n}$ to $\vfocki{n\pm1}$. For $\psi\in\vfocki{n}$, they are given according to
\begin{equation}\label{eq:2}
\begin{split}
\calA_+ \psi(p_{1:n+1}) &= \frac{\iota \beta}{n+1} \sum_{i=1}^{n+1} p_{[1:(n+1)\setminus i]} \, \psi(p_{1:(n+1)\setminus i})\,, \\
\calA_- \psi(p_{1:n-1}) &= \iota \beta n p_{[1:n-1]} \int \VHat(q) \psi(q,p_{1:n-1}) \dd q\,,
\end{split}
\end{equation}
along with $\calA_+|_{\vfocki{0}} = \calA_-|_{\vfocki{0}\oplus\vfocki{1}} = 0$. It holds that $(\calA_\pm)^* = - \calA_\mp$, so that $\calA \eqdef \calA_+ + \calA_-$ is skew self-adjoint.

For cylinder function $u \in t\calC$, the associated Dynkin martingale
\begin{equation}\label{eq:8}
M_t(u) = u(t, \eta_t) - u(0, \eta_0) - \int_0^t ( \partial_t + \calL )u(s, \eta_s) \dd s
\end{equation}
may be given explicitly according to $M_t(u) = \int_0^t \nabla_0 u(s, \eta_s) \dd B_s$. In particular, the quadratic variation is given by $\langle M(u) \rangle_t = \int_0^t |\nabla_0 u(s, \eta_s)|^2 \dd s$.
\end{lemma}

We finish this section by proving the following integrability property. It is the starting point for verifying the condition known as the {\it Itô} trick, that is introduced in Section \ref{sec:definition-ito-trick}. The strategy of proof is standard, see for example \cite{KomoLandOlla12_FluctuationsMarkov,GubinelliJara13}.

\begin{proposition}\label{prp:1}
For $V \in C^\infty(\bbR)$ satisfying Assumption \ref{ass:posdef}, and environment distributions $\pi_0 \ll \pi(V)$, let $\bbP_0$ be the law of the environment process $(\eta_t)_{t \geqslant 0}$ with interaction function $V$, coupling constant $\beta^2 \geqslant 0$, and environment distribution $\pi_0$; in other words, $\bbP_0 = \text{ENV}(V, \beta, \pi_0)$. For all $p \geqslant 1$, there is a constant $C = C(p) > 0$ such uniformly among all such $V$, $\pi_0$, $T \geqslant 0$, and time-dependent cylinder functions $u \in t\calC$, it holds
\begin{equation}\label{eq:15}
\bbE_0\Big[ \sup_{t \in [0,T]} \Big\lvert \int_0^t \calL_0u(s, \eta_s) \dd s \Big\rvert^p \Big] \leqslant C \lVert \dd \pi_0 / \dd \pi \rVert_{\vfock} \lVert \nabla_0 u \rVert^p_{L^2_TL^{2p}(V)}
\end{equation}
\end{proposition}

\begin{proof}
Write $\bbP = \text{ENV}(V, \beta, \pi)$. We begin with Cauchy-Schwarz:
$$ \bbE_0\Big[ \sup_{t \in [0,T]} \Big\lvert \int_0^t \calL_0u(s, \eta_s) \dd s \Big\rvert^p \Big] \leqslant \lVert \dd \pi_0 / \dd \pi \rVert_{\vfock} \Big( \bbE\Big[ \sup_{t \in [0,T]} \Big\lvert \int_0^t \calL_0(s, \eta_s) \dd s \Big\rvert^{2p}  \Big] \Big)^{1/2} $$
With this step we have reduced the proof to the stationary setting, and it remains to show that
\begin{equation}\label{eq:9}
\bbE \Big[ \sup_{t \in [0,T]} \Big\lvert \int_0^t \calL_0u(s, \eta_s) \dd s \Big\rvert^{2p} \Big] \leqslant \lVert \nabla_0 u \rVert^{2p}_{L^2_TL^{2p}(V)}
\end{equation}

Let $(M_t(u))_{t \geqslant 0}$ be the Dynkin martingale as given in \eqref{eq:8}. We reverse the process $\zeta_t = (t, \eta_t)$ to obtain $(\zeta^*_t)_{t \in [0, T]}$, where $\zeta^*_t = (T - t, \eta^*_t)$, and $\eta^*_t = \eta_{T - t}$. We also consider the Dynkin martingale associated to $\zeta^*_t$,
$$ M^*_t(u) := u(T-t, \eta^*_t) - u(T, \eta^*_0) - \int_0^t ( - \partial_t + \calL^* ) u(T-s, \eta^*_s) \dd s $$
where $\calL^* = \calL_0 - \calA$ is the adjoint of $\calL$. We note that $(M^*_t(u))_{t \in [0, T]}$ is a martingale with respect to the filtration generated by $(\eta^*_t)$.

We combine the martingales in such a way so as to cancel out both the boundary terms and the terms that correspond to the anti-symmetric part of the generator:
$$ \int_0^t \calL_0u(s, \eta_s) \dd s = \frac{1}{2} \big( M_t(u) + M^*_T(u) - M^*_{T-t}(u) \big) $$
We control each term on the RHS in the above separately, noting that by doing so we pick up a $p$-dependent constant. Each of these terms is controlled in the same manner, we only detail one example. Burkholder-David-Gundy is used to give
\begin{align*}
\bbE\Big[ \sup_{t \in [0, T]} | M_t(u) |^{2p} \Big] \lesssim \bbE\Big[ \langle M(u) \rangle_T^p \Big] = \Big\lVert \int_0^t (\nabla_0 u(s, \eta_s))^2 \dd s \Big\rVert_{L^p_\bbP}^p \leqslant \Big(\int_0^t \lVert (\nabla_0 u(s, \eta_s))^2 \rVert_{L^p_\bbP} \dd s \Big)^p
\end{align*}
where in the above we used Minkowski's inequality. Due to time-stationarity under $\bbP$, the RHS equals $\lVert \nabla_0 u \rVert^{2p}_{L^2_TL^{2p}(V)}$, and this concludes the proof of \eqref{eq:15}.
\end{proof}

\section{The limiting generator $\calL$ and notions of solution}\label{sec:constr-limit-gener}
In Section \ref{sec:mollified-srbp} we introduced the smooth SRBP in which the interaction function $V$ was assumed to be smooth. We derived the generator $\calL$ for the corresponding environment process, $(\eta_t)_{t \geqslant 0}$. In the case of rough interaction, $V \in \calS'(\bbR)$, one can still make sense of the operator $\calL$, but the immediate connection with any stochastic process $(\eta_t)_{t \geqslant 0}$ is lost.
\begin{definition}
Let $V \in \calS'(\bbR)$ satisfy Assumption \ref{ass:posdef}, and define the unbounded operator $\calL : \vfock \rightarrow \vfock$ according to the action
$$\calL f := \calL_0 f + \calA f, \qquad \calA f := \calA_+ f + \calA_- f$$
where $\calA_\pm$ are given in \eqref{eq:2}.
\end{definition}
We remark that it is not a priori clear which elements $f \in \vfock$ return an element $\calL f \in \vfock$, that is to say, it is not clear what is the domain of $\calL$. In general we do not expect $\calL f \in \vfock$ even for $f \in \calC$. However, $\calL f$ is well defined as an element of $\calE_{n-1} \oplus \calE_n \oplus \calE_{n+1}$ provided that $f \in \calE_n \cap \calS(\bbR^n, \bbC)$. Whether such $\calL f \in \vfock$ turns out to be unimportant for our arguments: we only require a weaker $\vsob{-1}{0}$ integrability, as explored in Section \ref{sec:grad-sect-cond}.

The goal of this section is to give a good notion of solution associated to the operator $\calL$. That is to say, given a distribution $\bbP_0$ on the path space $C(\bbR_+, \calS'(\bbR))$, we wish to present sufficient conditions under which the corresponding stochastic process $(\eta_t)_{t \geqslant 0}$ has generator $\calL$. Our notion of solution is that of {\it energy solution}, the history of which we have given already in the introduction. The requirements to be an energy solution can be thought of as a form of martingale problem for the operator $\calL$, along with an integrability condition for additive functionals, known as the Itô trick. In Sections \ref{sec:existenceenergy} and \ref{sec:uniq-energy-solut} we show that our definition is well posed, in that we have existence and uniqueness.

We begin this section with an a priori estimate, known as the Graded Sector Condition, followed by the definition of the Itô trick and the definition of an energy solution. Then we give a second notion of solution, the {\it cylinder function martingale problem}, which will be useful for proving uniqueness. Finally, in Theorem \ref{thm:etoc}, we show that an energy solution is also a solution to the cylinder function martingale problem, provided that the Itô trick holds with sufficiently high integrability.

\subsection{The Graded Sector Condition}\label{sec:grad-sect-cond}
We begin this subsection by introducing a central estimate, which we refer to as the {\it inner integral bound}. The proof is an elementary consequence of uniform local integrability, condition \eqref{eq:63}, and can be found in the appendix, Section \ref{sec:proof-prop-iib}.
\begin{proposition}[Inner integral bound]\label{prp:iib}
Let $V \in \calS'(\bbR)$ satisfy Assumption \ref{ass:posdef} and define for $\lambda, s > 0$,
\begin{equation}\label{eq:46}
J^s(\lambda) := \sup_{p \in \bbR} \Big\{ \beta^2 \int_\bbR \frac{\VHat(q)}{(\lambda + \thalf |p+q|^2)^s} \dd q \Big\}
\end{equation}
Then it holds that $J^s(\lambda) < \infty$ and $\lim_{\lambda\rightarrow\infty} J^s(\lambda) = 0$ for all $s > 1/2$.
\end{proposition}

The reason for the name ``inner integral bound'' is as follows. In the analysis, we shall come across iterated integrals, in which one of the variables is ``free'', and can be taken on the inside, producing an inner integral. That inner integral we shall always be able to control with the quantity $J^s(\lambda)$. See for example \eqref{eq:40}, in which the free variable is $p_{n+1}$, and the quantity $J^1(1)$ is used as an upper bound.

\begin{remark}\label{rmk:2}
The assumption of uniform local integrability, Equation \eqref{eq:63}, is minimal in the following sense. Consider $V \in \calS'(\bbR)$ for which $\VHat \in L^1_{\text{loc}}(\bbR)$, but \eqref{eq:63} does not hold, i.e.\ there exists a sequence $(z_n)_{n = 1}^\infty \subset \bbZ$ such that $ \int_{z_n}^{z_n+1} \VHat(q) \dd q \rightarrow \infty$. Then for any $\lambda, s > 0$, it holds
$$ \sup_{p \in \bbR} \Big\{ \int_\bbR \frac{\VHat(q)}{(\lambda + \thalf |p+q|^2)^s} \dd q  \Big\} \geqslant \int_\bbR \frac{\VHat(q)}{(\lambda + \thalf |q-z_n|^2)^s} \dd q \geqslant \frac{1}{(\lambda+1)^s} \int_{z_n}^{z_{n+1}} \VHat(q) \dd q $$
In particular, one has $J^s(\lambda) = \infty$.
\end{remark}

One crucial regularity property of our operator $\calL$ is that it is a continuous map $\calL : \vsob{1}{1} \rightarrow \vsob{-1}{0}$. The idea being that you lose ``two derivatives'' in the operator $(1-\calL_0)^{1/2}$ and ``one derivative'' in $(1+\calN)^{1/2}$. We refer to this estimate as {\it the Graded Sector Condition}. A version of this estimate is used in previous works on the SRBP \cite{HorvTothVeto12_DiffusiveLimits, CannizzarGiles25_InvariancePrinciple} and in many other works involving different models, e.g.\ the Stochastic Burgers Equation \cite{CannGubiToni24_GaussianFluctuations}, however not in the context of solving singular SDEs/SPDEs, to the best of our knowledge.

\begin{lemma}[Graded Sector Condition]\label{lem:1}
Let $V \in \calS'(\bbR)$ satisfy Assumption \ref{ass:posdef}. Then for each $\alpha \in \bbR$, the operator $\calA_\pm$ defines a bounded linear operator $\calA_{\pm} : \vsob{1}{\alpha} \rightarrow \vsob{-1}{\alpha-1}$ with operator norm
$$ \lVert \calA_{\pm} \rVert_{L(\vsob{1}{\alpha}, \vsob{-1}{\alpha-1})}^2 \leqslant J^1(1)(2^{1-\alpha} + 2^\alpha) $$
where $J^1(1)$ is as defined in \eqref{eq:46}, and $J^1(1) < \infty$ by Proposition \ref{prp:iib}.
\end{lemma}

\begin{proof}
We begin with $\calA_+$. By density, it is enough to consider $f \in \calC$, and by orthogonality we may assume that $f = \psi$ for some $\psi \in \calE_n \cap \calS(\bbR^n)$. In which case,
$$ \lvert \calA_+ \psi  \rvert^2 \leqslant \frac{\beta^2}{n+1} \sum_{i=1}^{n+1} |p_{[1:(n+1)\setminus i]}|^2 \, |\psi(p_{1:(n+1)\setminus i})|. $$
Using this bound, and the symmetry of $\mu_{n+1}(\dd p_{1:(n+1)})$, we obtain
\begin{equation}\label{eq:40}
\lVert \calA_+ \psi \rVert_{\vsob{-1}{\alpha-1}}^2 \leqslant (n+2)^{\alpha-1} \beta^2 \int \frac{|p_{[1:n]}|^2 |\psi(p_{1:n})|^2}{1 + \thalf|p_{[1:(n+1)]}|^2} \mu_{n+1}(\dd p_{1:(n+1)})
\end{equation}
Using $\mu_{n+1}(\dd p_{1:(n+1)}) = (n+1) \VHat(p) \dd p_{n+1} \, \mu_n(\dd p_{1:n})$, we obtain
\begin{align*}
\beta^2 &\int \frac{|p_{[1:n]}|^2 |\psi(p_{1:n})|^2}{1 + \thalf|p_{[1:(n+1)]}|^2} \mu_{n+1}(\dd p_{1:(n+1)}) \\
 &\leqslant (n+1) \int |p_{[1:n]}|^2 |\psi(p_{1:n})|^2 \big( \beta^2 \int \frac{\VHat(p)}{1 + \thalf|p_{[1:n]} + p_{n+1}|^2} \dd p_{n+1} \big) \, \mu_n(\dd p_{1:n}) \\
 &\leqslant (n+1) \int |p_{[1:n]}|^2 |\psi(p_{1:n})|^2 J^1(1) \, \mu_n(\dd p_{1:n})
\end{align*}
Plugging this into \eqref{eq:40} gives
$$ \lVert \calA_+ \psi \rVert_{\vsob{-1}{\alpha-1}}^2 \leqslant 2 J^1(1) \Big( \frac{n+2}{n+1} \Big)^{\alpha-1} \lVert \psi \rVert_{\vsob{1}{\alpha}}^2 $$
which implies $\lVert \calA_+ \rVert_{L(\vsob{1}{\alpha},\vsob{-1}{\alpha-1})}^2 \leqslant 2^\alpha J^1(1)$. By duality, it holds that $\lVert \calA_- \rVert_{L(\vsob{1}{\alpha},\vsob{-1}{\alpha-1})}^2 \leqslant 2^{1-\alpha} J^1(1)$.
\end{proof}

\subsection{The Itô trick and the notion of energy solution}\label{sec:definition-ito-trick}
If one naively writes down the martingale problem associated to $\calL$, then it will involve additive functionals $\int f(s, \eta_s) \dd s$ that are not well-defined. Indeed, $f(s, \omega)$ shall depend on the value $\omega(0)$, which is nonsense because we know only that $\omega \in \calS'(\bbR)$. The key insight is that we can give meaning to these terms, provided that we make an additional integrability assumption, known as the Itô trick, as in \cite{GoncalvesJara14_NonlinearFluctuations,GubinelliJara13,GubinelPerkows20_InfinitesimalGenerator}.
\begin{definition}[Itô trick]\label{def:itoTrick}
Let $V \in \calS'(\bbR)$ satisfy Assumption \ref{ass:posdef} and let $\eta=(\eta_t)_{t\geqslant 0}$ be a $C(\bbR_{\geqslant 0}, \calS'(\bbR))$-valued random variable. For $p \geqslant 1$, we say that $\eta$ satisfies the Itô trick (with respect to $V$ and $p$) if for each horizon $T \geqslant 0$, there exists a constant $C > 0$ such that 
\begin{equation}\label{eq:17}
\bbE\Big[ \sup_{0 \leqslant t \leqslant T} \Big\lvert \int_0^t f(s, \eta_s) \dd s \Big\rvert^p \Big] \leqslant C \lVert \sqrt{2p-1}^{\mathcal{N}} f\rVert^p_{L^2_T\vsob{-1}{0}} \qquad \forall f \in t\calC
\end{equation}
\end{definition}

\begin{remark}
In the smooth setting, Proposition \ref{prp:1} gives us a way to control additive functions of the form $\int \calL_0 u(s, \eta_s) \dd s$. If one applies this proposition in the case of $u = (1-\calL_0)^{-1} f$, then one obtains the control given on the RHS of \eqref{eq:17}. See Proposition \ref{prp:tight}.
\end{remark}

As mentioned, this assumption allows for us to give meaning to the additive functionals for nice enough $f$, and we continue to write $\int_0^t f(s, \eta_s) \dd s$, even though the process will no longer be, in general, of finite variation. In the following definition the exponential factor $\sqrt{2p-1}^\calN$ appearing in \eqref{eq:17} is irrelevant, because we consider $p = 1$.
\begin{definition}\label{def:i}
Assuming that Definition \ref{def:itoTrick} is satisfied in the case $p = 1$, for each $T \geqslant 0$, we may define $I_T : L^2_T\vsob{-1}{0} \rightarrow L^1_\bbP C_T$ to be the continuous extension of the map $t\calC \ni f \mapsto ( \int_0^t f(s, \eta_s) \dd s )_{t \in [0, T]}$, the existence of which is due to the bound \eqref{eq:17}. For any $f \in L^2_{T, \text{loc}}\vsob{-1}{0}$, we write $\int_0^\cdummy f(s, \eta_s) \dd s$ for the unique $C(\bbR_{\geqslant 0}, \bbR)$-valued random variable that coincides with $I_T(f)$ on the time interval $[0, T]$, for each $T \geqslant 0$.
\end{definition}

Definition \ref{def:i} is used to give meaning to additive functionals $\int_0^\cdummy \calA h(s, \eta_s) \dd s$ that would appear in the martingale problem for the operator $\calL$. For a single test function $\phi \in \calS(\bbR)$, we define the linear functional $\ell_\phi \in \calC_p$ to be the cylinder function given according to
\begin{equation}\label{eq:38}
\ell_{\phi}(\eta) := \eta(\phi)
\end{equation}
Indeed, if $\eta=(\eta_t)_{t\geqslant 0}$ is a $C(\bbR_{\geqslant 0}, \calS'(\bbR))$ valued random variable satisfying the Itô trick with respect to $V$ and $p = 1$, then Definition \ref{def:i} defines a process $\int_0^\cdummy \calA \ell_\phi(s, \eta_s) \dd s$ because of the following estimate:

\begin{proposition}\label{prp:2}
Let $V \in \calS'(\bbR)$ satisfy Assumption \ref{ass:posdef}. For $\phi \in \calS(\bbR)$ it holds that
$$ \lVert \calA \ell_\phi \rVert_{\vsob{-1}{0}}^2 \, \vee \, \lVert \calL \ell_\phi \rVert_{\vsob{-1}{0}}^2 \leqslant J^1(1) \int \VHat(p) |p|^2 | \phiHat(p)|^2 \dd p $$
In particular, for all $T \geqslant 0$, it holds $\calA \ell_\phi, \calL \ell_\phi \in L^2_{T}\vsob{-1}{0}$.
\end{proposition}

\begin{comment}
If you want the general $p$ version, then $I : L^2_T\sqrt{2p-1}^{-\mathcal{N}}\vsob{-1}{0} \rightarrow L^p_\bbP C_T$. But it needs work because you don't have ``density'' of $t\calC$, in fact $t\calC$ doesn't even belong to $L^2_T\sqrt{2p-1}^{-\mathcal{N}}\vsob{-1}{0}$.
\end{comment}

\begin{proof}
Due to the Lemma \ref{lem:1} in the case $\alpha = 1$, it holds that
\begin{align*}
\lVert \calA \ell_\phi \rVert_{\vsob{-1}{0}}^2 , \lVert \calL \ell_\phi \rVert_{\vsob{-1}{0}}^2 \lesssim J^1(1) \lVert \ell_\phi \rVert_{\vsob{1}{1}}^2 \lesssim J^1(1) \int \VHat(p) |p|^2 | \phiHat(p)|^2 \dd p
\end{align*}
The RHS is finite due to the decay of $\phiHat$.
\end{proof}

In addition, we give the following definition from \cite{GubinelPerkows20_InfinitesimalGenerator}, which ensures that a process $(\eta_t)_{t \geqslant 0}$ can be estimated in terms of the measure $\vfock$.
\begin{definition}[Incompressibility]\label{def:incomp}
Let $\eta=(\eta_t)_{t\geqslant 0}$ be a $C(\bbR_+, \calS'(\bbR))$ valued random variable. We say that $\eta$ is {\it incompressible} (with respect to $V$) if for all $T\geqslant 0$ and $f\in\mathcal{C}$ it holds
$$ \bbE[|f(\eta_t)|] \lesssim_T \|f\|_{\vfock}, \quad \forall t\in[0,T] $$
\end{definition}

We now define what it means to be an energy solution.
\begin{definition}[Energy solution]\label{def:energysolution}
Let $V \in \calS'(\bbR)$ satisfy Assumption \ref{ass:posdef} and let $\eta=(\eta_t)_{t\geqslant 0}$ be a $C(\bbR_{\geqslant 0}, \calS'(\bbR))$ valued random variable. We say that $\eta$ is an {\it energy solution} (with respect to V and $\calL$), if
\begin{enumerate}
\item There exists $p_0 \geqslant 1$ such that for all $p\in[1,p_0)$, the Itô trick, Definition \ref{def:itoTrick}, is satisfied with respect to $V$ and $p$.
\item The incompressibility condition is satisfied with respect to $V$, Definition \ref{def:incomp}.
\item For all $\phi \in \calS$ it holds that the following process is a martingale adapted to the natural filtration generated by $\eta$, in which $(\int_0^t \calA \ell_{\phi}( \eta_s) \dd s)_{t \geqslant 0}$ is defined thanks to Proposition \ref{prp:2},
\begin{equation}\label{eq:32}
M_t(\ell_\phi) = \eta_t(\phi) - \eta_0(\phi) - \int_0^t \calL_0 \ell_{\phi}(\eta_s) \dd s - \int_0^t \calA \ell_{\phi}(\eta_s) \dd s
\end{equation}
\item The quadratic variation of $M(\ell_\phi)$ is given by
\begin{equation}\label{eq:qvenergy}
\langle M(\ell_\phi) \rangle_t = \int_0^t|\eta_s(\nabla \phi)|^2 \dd s
\end{equation}
\end{enumerate}
\end{definition}
\begin{remark}
Definition \ref{def:energysolution} is taken from \cite{GubinelPerkows20_InfinitesimalGenerator}, where the term \textit{solution to the cylinder function martingale problem} is used instead of \textit{energy solution}.
\end{remark}
We will also make use of the following equivalent formulation.
\begin{proposition}
An equivalent notion of energy solution can be obtained by replacing the fourth condition with the following:
\begin{enumerate}
\item[4.] For all $\phi_1, \phi_2 \in \calS$, with the functional $h(\eta) := \eta(\phi_1) \eta(\phi_2)$ it holds that the following process is a martingale adapted to the natural filtration generated by $\eta$
$$ M_t(h) = h(\eta_t) - h(\eta_0) - \int_0^t \calL h(\eta_s) \dd s $$
\end{enumerate}
\end{proposition}

\begin{proof}
This is a standard argument, see for example \cite[Lemma 6.5]{CannizzarGiles25_InvariancePrinciple}.
\end{proof}

\subsection{Cylinder function martingale problem}\label{sec:cylind-funct-mart}
Roughly speaking, an energy solution is a process which solves the martingale problem for time-independent observables $h \in \calC_p$ that are either linear or quadratic. We now give a second notion of solution that involves a stronger condition, in which the martingale problem is satisfied for a larger class: for all time-dependent cylinder functions, $h \in t\calC$. The main result of this subsection, Theorem \ref{thm:etoc}, shows that in fact an energy solution already satisfies this seemingly stronger notion of solution.

The notion of cylinder function martingale problem involves additive functionals $\int \calA h \dd s$ for $h \in t \calC$. As in the case of energy solutions, these are to be interpreted in the sense of Definition \ref{def:i}, for which we require the following analogue of Proposition \ref{prp:2}:
\begin{proposition}
For $h \in t \calC$ it holds that $\calA h, \calL h \in L^2_{T}\vsob{-1}{0}$.
\end{proposition}

\begin{proof}
As in the proof of Proposition \ref{prp:2}, we apply Lemma \ref{lem:1}, and note that $J^1(1) < \infty$. Therefore the result follows provided that $\lVert h \rVert_{L^2_T\vsob{1}{1}} < \infty$. This holds because $\calN, \calL_0$ are closed on $t \calC_p$.
\end{proof}

\begin{definition}[Cylinder function martingale problem]\label{def:mcyl}
Let $V$ satisfy Assumption \ref{ass:posdef}. Let $\eta=(\eta_t)_{t \geqslant 0}$ be a $C(\bbR_+, \calS'(\bbR))$ valued random variable. We say that $\eta$ satisfies the {\it cylinder function martingale problem} associated to $\calL$ if
\begin{enumerate}
\item The Itô trick, Definition \ref{def:itoTrick}, is satisfied with respect to $V$ and $p = 1$.
\item The incompressibility condition is satisfied with respect to $V$, Definition \ref{def:incomp}
\item For all cylinder functions $h \in t\calC$ it holds that the following process is a martingale adapted to the natural filtration generated by $\eta$, in which $(\int_0^t \calA h(s, \eta_s) \dd s)_{t \geqslant 0}$ is defined thanks to Proposition \ref{prp:2},
\begin{equation}\label{eq:33}
M_t(h) = h(t, \eta_t) - h(0, \eta_0) - \int_0^t (\partial_t+\calL_0) h(s, \eta_s) \dd s - \int_0^t \calA h(s, \eta_s) \dd s
\end{equation}
\end{enumerate}
\end{definition}

\begin{theorem}\label{thm:etoc}
Let $V \in \calS'(\bbR)$ satisfy Assumption \ref{ass:posdef}. If $\eta$ is an energy solution, Definition \ref{def:energysolution}, in which $p_0 > 4$, then it holds that $\eta$ is a solution to the cylinder function martingale problem, Definition \ref{def:mcyl}.
\end{theorem}

The proof follows a strategy similar to that which can be found in \cite{GrafPerkPopa24_EnergySolutions,GubinelPerkows18_EnergySolutions,GubinelPerkows20_InfinitesimalGenerator}. The idea is as follows. All there is to prove is that the RHS of \eqref{eq:33} is a martingale. Writing $h(t, \eta) = h(t, \omega(\phi_1), \ldots, \omega(\phi_k))$, in which $\phi_i \in \calS(\bbR)$, it holds that the RHS of \eqref{eq:33} is equal to
\begin{equation}\label{eq:48}
\begin{split}
h(t, \eta_t) &- h(0, \eta_0) - \int_0^t \partial_t h (s, \eta_s) \dd s \\
&- \int_0^t \Big( \thalf \sum_{i = 1}^k \partial_i h(s, \eta_s) \eta_s(\Delta\phi_i) + \sum_{i,j = 1}^k \partial_{ij}^2 h(s, \eta_s) \eta_s(\nabla\phi_i)\eta_s(\nabla\phi_j) \Big) \dd s \\
&- \int_0^t \sum_{i=1}^k \partial_i h(s, \eta_s) \mathcal{A} \ell_{\phi_i}(\eta_s) \dd s
\end{split}
\end{equation}
Indeed, this is a simple consequence of the fact that $\calA h(s, \eta) = \sum_{i=1}^k \partial_i h(s, \eta) \mathcal{A} \ell_{\phi_i}(\eta_s)$ as elements of $L^2_T\vsob{-1}{0}$.

If we could apply Itô's formula to the process $h(\eta_t) = h(t, \eta_t(\phi_1), \ldots, \eta_t(\phi_k))$, in which the ``semi-martingale'' $(\eta_t(\phi))_{t \geqslant 0}$ is given according to \eqref{eq:32}, then this would exactly give us \eqref{eq:48}. However, $(\eta_t(\phi))_{t \geqslant 0}$ is not a semi-martingale, because the additive functional term $\int_0^t \calA \ell_{\phi}( \eta_s) \dd s$ in \eqref{eq:32} is not really an integral in time, in particular, it is not a finite variation process. Therefore, an approximation needs to be made.

In this proof it shall be convenient to use an abstract approximation of $\mathcal{A}$ given according to
\begin{equation}\label{eq:54}
\mathcal{A}^m_{\pm} := 1_{- \mathcal{L}_0 \leqslant m} \mathcal{A}_{\pm}, \qquad \mathcal{A}^m = \mathcal{A}^m_+ + \mathcal{A}^m_-
\end{equation}
for fixed $m \in \bbN$. In addition, we write $(C^\alpha_T, \lVert \cdot \rVert_{C^\alpha_T})$ for the space of H\"{o}lder continuous functions $f : [0, T] \rightarrow \bbR$ with exponent $\alpha \in (0, 1]$.

\begin{proof}[Proof of Theorem \ref{thm:etoc}]
Define the semi-martingale
\begin{equation*}
\eta^m_t(\phi_i) := \eta_0(\phi_i) + \int_0^t \eta_s(\Delta \phi_i) \dd s + \int_0^t \mathcal{A}^m \ell_{\phi_i}(\eta_s) \dd s + M_t(\phi_i)
\end{equation*}
where $M_t(\phi_i)$ is the martingale in \ref{eq:32}. Applying Itô's formula to $h(t, \eta^m_t)$, we get that the following process is a martingale:
\begin{equation}\label{approxmartingaleeimpliesc}
\begin{split}
h(t, \eta_t^m) &- h(0, \eta_0^m) - \int_0^t \partial_t h (s, \eta^m_s) \dd s \\
&- \int_0^t \Big( \thalf \sum_{i = 1}^k \partial_i h(s, \eta^m_s) \eta_s(\Delta\phi_i) + \sum_{i,j = 1}^k \partial_{ij}^2 h(s, \eta^m_s) \eta_s(\nabla\phi_i)\eta_s(\nabla\phi_j) \Big) \dd s \\
&- \int_0^t \sum_{i=1}^k \partial_i h(s, \eta^m_s) \mathcal{A}^m \ell_{\phi_i}(\eta_s) \dd s
\end{split}
\end{equation}
Note that $\eta$ and $\eta^m$ are both present in \eqref{approxmartingaleeimpliesc}.

Thanks to \eqref{eq:48}, it remains to show that we can pass to the limit $m \rightarrow \infty$ in $L^1(\bbP)$ and obtain \eqref{eq:48}. This is sufficient because the martingale property is maintained in the limit. We refer to all of the terms in \eqref{approxmartingaleeimpliesc} except the last as the ``easy terms'', for their convergence is straightforward, as we now demonstrate. The last term, the ``hard term'', will require a finer argument, which will depend on Lemma \ref{Lemmaforeimpliesc}.

{\it The ``easy terms'' of \eqref{approxmartingaleeimpliesc}.} We begin by writing, for some $\phi = \phi_i$,
\begin{equation}\label{eq:62}
\eta^m_s(\phi) - \eta_s(\phi) = \int_0^s (\calA^m - \mathcal{A}) \ell_{\phi}(\eta_r) \dd r
\end{equation}
Applying the Itô trick for $p \leqslant p_0$, Definition \ref{def:itoTrick}, we obtain that for some constant $C>0$ independent of $m$, it holds
$$  \bbE\Big[\sup_{0 \leqslant s \leqslant t} \Big|\int_0^s (\calA^m - \mathcal{A}) \ell_{\phi}(\eta_r) \dd r \Big|^p \Big] \leqslant C \lVert (\calA^m_+ - \mathcal{A}_+) \ell_{\phi} \rVert_{\vsob{-1}{0}}^p $$
Note that the factor $\sqrt{2p-1}^\calN$ is of no concern, because all terms are in chaos no larger than $n = 2$, and so this factor is absorbed into the constant $C$. Moreover, there is no contribution from $(\calA^m_- - \mathcal{A}_-) \ell_{\phi}$. We claim that $\lVert (\calA^m_+ - \mathcal{A}_+) \ell_{\phi} \rVert_{\vsob{-1}{0}} \rightarrow 0$ as $m \rightarrow \infty$ for any $\phi \in \calS(\bbR)$. Indeed, in a similar fashion to \eqref{eq:40}, we compute
$$ \lVert (\calA^m - \mathcal{A}) \ell_{\phi} \rVert_{\vsob{-1}{0}}^2 \leqslant \beta^2  \int_{\thalf |p + q|^2 \geqslant m} \frac{|p \phiHat(p)|^2}{1 + \thalf |p+q|^2} \mu_2(\dd p, \dd q) $$
Taking the integral in $q$ on the inside, and writing $M = \sup_{z \in \bbZ} \int_z^{z+1} \VHat(q) \dd q$, we obtain
$$ \lVert (\calA^m - \mathcal{A}) \ell_{\phi} \rVert_{\vsob{-1}{0}}^2 \leqslant M \beta^2 \int \VHat(p) |p \phiHat(p)|^2 \dd p \sum_{j = m}^\infty \frac{1}{1 + \thalf j^2}  $$
It holds that $\int \VHat(p) |p \phiHat(p)|^2 \dd p < \infty$ because $|p \phiHat(p)|^2$ decays faster than any polynomial. In particular, we have the claim. Combining, we obtain that for all $p \leqslant p_0$, it holds
\begin{equation}\label{eq:49}
\bbE[\sup_{0 \leqslant s \leqslant t} |\eta^m_s(\phi) - \eta_s(\phi)|^p]  \rightarrow 0 \quad \text{ as } m \rightarrow \infty
\end{equation}
In particular, using a simple union bound, \eqref{eq:49} holds not only in the case where $(\eta_t(\phi))_{t \geqslant 0}$ is a $\bbR$-valued process, but also in the case where $\phi = (\phi_1, \ldots, \phi_k)$ for some $\phi_i \in \calS(\bbR)$, in which case $(\eta_t(\phi))_{t \geqslant 0} = (\eta_t(\phi_1), \ldots, \eta_t(\phi_k))_{t \geqslant 0}$ is a $\bbR^k$-valued process.

To handle the first of the ``easy terms'', we must show $\bbE[|h(t, \eta^m_t) - h(t, \eta_t)|] \rightarrow 0$ as $m \rightarrow \infty$, but this follows directly from \eqref{eq:49}, simply because $h(t, \eta_t)$ is Lipschitz in the variable $\eta_t(\phi) \in \bbR^k$.
The remaining ``easy terms'' are handled as follows. We take just one of them as an example. Consider the term, with $f(\omega) = \omega(\nabla\phi_i)\omega(\nabla\phi_j)$,
\begin{equation}\label{eq:61}
\Big|\int_0^t \big( \partial_{ij}^2 h(s, \eta^m_s) - \partial_{ij}^2 h(s, \eta_s) \big) f(\eta_s) \dd s  \Big| \lesssim \sup_{0 \leqslant s \leqslant t} |\eta^m_s(\phi) - \eta_s(\phi)| \, \int_0^t |f(\eta_s)| \dd s
\end{equation}
where in the last line we used that $\partial_{ij}^2 h$ is also Lipschitz. Taking expectations and using H\"{o}lder with conjugate indices, $p, q > 1$, $p^{-1} + q^{-1} = 1$, we obtain
\begin{equation}\label{eq:50}
\begin{split}
\bbE\Big[ \Big|\int_0^t \big( \partial_{ij}^2 h(s, \eta^m_s) &- \partial_{ij}^2 h(s, \eta_s) \big) f(\eta_s) \dd s  \Big| \Big] \\
&\lesssim \bbE[ \sup_{0 \leqslant s \leqslant t} |\eta^m_s(\phi) - \eta_s(\phi)|^p ]^{1/p} \bbE\Big[ \Big(\int_0^t |f(\eta_s)| \dd s \Big)^q \Big]^{1/q}
\end{split}
\end{equation}
The first factor goes to zero as $m \rightarrow \infty$ by \eqref{eq:49}, provided that we choose $p \leqslant p_0$, and the second factor is finite by incompressibility, and it does not depend on $m$.

{\it The ``hard term'' of \eqref{approxmartingaleeimpliesc}.} We begin with the simple bound
\begin{align}
\bbE&\Big[ \Big| \int_0^t \sum_{i=1}^k \partial_i h(s, \eta^m_s) \mathcal{A}^m \ell_{\phi_i}(\eta_s) \dd s - \int_0^t \sum_{i=1}^k \partial_i h(s, \eta_s) \mathcal{A} \ell_{\phi_i}(\eta_s) \dd s \Big| \Big] \nonumber \\
&\leqslant \sum_{i=1}^k \bbE\Big[ \Big| \int_0^t \big( \partial_i h(s, \eta^m_s) - \partial_i h(s, \eta_s) \big) \mathcal{A}^m \ell_{\phi_i}(\eta_s) \dd s \Big| \Big] \label{eq:52} \\
&\, + \sum_{i=1}^k \bbE\Big[ \Big| \int_0^t  \big( \partial_i h(s, \eta_s) \mathcal{A}^m \ell_{\phi_i}(\eta_s) - \partial_i h(s, \eta_s) \mathcal{A} \ell_{\phi_i}(\eta_s) \big) \dd s \Big| \Big] \label{eq:53}
\end{align}
We need to show that both \eqref{eq:52} and \eqref{eq:53} go to zero as $m \rightarrow \infty$.

The term \eqref{eq:53} is easily controlled via the Itô trick. Indeed, it is sufficient to show that $\partial_i h \mathcal{A}^m \ell_{\phi_i} \rightarrow \partial_i h \mathcal{A} \ell_{\phi_i}$ in $L^2_T\vsob{-1}{0}$. We already showed $\mathcal{A}^m \ell_{\phi_i} \rightarrow \mathcal{A} \ell_{\phi_i}$, and therefore the result follows by continuity of the multiplication map $f \mapsto hf$ in $L^2_T\vsob{-1}{0}$, in which the factor $h \in \calC$ is a fixed cylinder function.

As for the term \eqref{eq:52}, if we were to proceed in the same way as we did for the ``easy terms'' then we run intro problems: in this case, the second term in \eqref{eq:50} would be $m$-dependent, and can only be controlled with a bound that diverges as $m \rightarrow \infty$. The key observation is to note that we have sufficient H\"{o}lder regularity so that we can replace the estimate \eqref{eq:61} with Young's inequality.

For $j \in 1{:}k$, write $A^m_t(\phi_j) := \int_0^t \mathcal{A}^m \ell_{\phi_j} (\eta_r) \dd r$, and $A_t(\phi_j) := \int_0^t \mathcal{A} \ell_{\phi_j} (\eta_r) \dd r$, then using the Lipschitz property of $\partial_ih$, we obtain via Young's inequality that for $\alpha, \beta \in (0, 1]$ such that $\alpha + \beta > 1$, it holds:
\begin{align*}
\Big| \int_0^t \big( \partial_i h(s, \eta^m_s) - \partial_i h(s, \eta_s) \big) \mathcal{A}^m \ell_{\phi_i}(\eta_s) \dd s \Big| &\leqslant \| \eta(\phi) - \eta^m(\phi) \|_{C^{\alpha}_T} \| A^m_t(\phi_i) \|_{C^{\beta}_T}  \\
&\leqslant \Big(\sum_{i = 1}^k \| A^m_t(\phi_j) - A_t(\phi_j) \|_{C^{\alpha}_T} \Big) \| A^m_t(\phi_i) \|_{C^{\beta}_T}
\end{align*}
where in the last line we used \eqref{eq:62}. We need to show that this term goes to zero in expectation as $m \rightarrow \infty$. Provided that $p_0 > 4$, we may choose $\alpha, \beta \in (0, 3/4 - 1/p_0)$ such that $\alpha + \beta > 1$. At this point the result follows from applying Cauchy-Schwartz and invoking Lemma \ref{Lemmaforeimpliesc}.
\end{proof}

\begin{lemma}\label{Lemmaforeimpliesc}
Let $\eta$ be an energy solution, Definition \ref{def:energysolution}, in which it is assumed that the Itô trick holds for all $p \in [1, p_0)$. Let $\calA^m$ be the approximation given in \eqref{eq:54} and for fixed $\phi \in \calS(\bbR)$, define $(A^m_t)_{t \geqslant 0}$ to be the finite-variation process, $A^m_t := \int_0^t \mathcal{A}^m \ell_\phi (\eta_r) \dd r$, and also define $(A_t)_{t \geqslant 0}$ to be the process $A_t := \int_0^t \mathcal{A} \ell_\phi (\eta_r) \dd r$, whose existence follows from Proposition \ref{prp:2}. Then, for all $\alpha \in (0, 3/4 - 1/p_0)$, $T \geqslant 0$ it holds
\begin{align}
\sup_m \bbE [\| A^m \|_{C^\alpha_T}^{p_0}], \: \bbE [\| A \|_{C^\alpha_T}^{p_0}] &< \infty \label{eq:56}\\
\lim_{m \rightarrow \infty} \bbE [\| A - A^m \|_{C^\alpha_T}^{p_0}] &= 0 \label{eq:58}
\end{align}
\end{lemma}

\begin{proof}
First, we make use of Proposition \ref{prp:iib}, to see that for any $q > 1/2$, it holds that
\begin{equation}\label{eq:55}
\sup_m \| \mathcal{A}_+^m \ell_\phi \|_{\vsob{-q}{0}}, \| \mathcal{A}_+ \ell_\phi \|_{\vsob{-q}{0}} < \infty
\end{equation}
Indeed, this is just a computation (c.f. \eqref{eq:40}) in which we use the quantity $J^q(1)$, as defined in \eqref{eq:46}, to estimate an inner-integral:
\begin{align*}
\| \mathcal{A}_+^m \ell_\phi \|_{\vsob{-q}{0}}^2 \leqslant \beta^2  \int \frac{|p \phiHat(p)|^2}{(1 + \thalf |p+q|^2)^q} \mu_2(\dd p, \dd q)
\leqslant 2 J^q(1) \int \VHat(p) |p \phiHat(p)|^2 \dd p
\end{align*}
The bound also holds for $\mathcal{A}_+$ in place of $\mathcal{A}_+^m$, and the RHS is uniform in $m \in \bbN$, and finite because of  Proposition \ref{prp:iib} and the decay of $\phiHat(p)$.

To estimate H\"{o}lder norms of $A^m$ and $A$, we use two main estimates. First, consider $0 \leqslant s \leqslant t \leqslant T$, and apply the usual Jensen's inequality, then incompressibility, followed by hypercontractivity to obtain that for all $p \geqslant 1$
\begin{align}
\bbE [ | A^m_t - A^m_s |^p ] &= \bbE \big[ \big| \int_s^t \mathcal{A}^m\ell_\phi (\eta_r) \dd r \big|^p \big] \leqslant (t-s)^{p-1} \int_s^t \bbE \big[|\mathcal{A}^m\ell_\phi (\eta_r)|^p] \dd r \nonumber\\
&\leqslant (t-s)^p \| \mathcal{A}^m\ell_\phi \|_{L^{2p}(V)}^{p} \lesssim (t - s)^{p} \| \mathcal{A}^m\ell_\phi \|_{L^2(V)}^{p} \nonumber\\
&\leqslant (t - s)^{p} m^{pq/2} \| \mathcal{A}^m\ell_\phi \|_{\vsob{-q}{0}}^{p} \label{eq:60}
\end{align}
Provided that we consider $q > 1/2$, then \eqref{eq:55} implies that the RHS is finite.

Second, for $n \leqslant m$, applying the Itô trick, we obtain that for all $p \in [1, p_0)$, it holds
\begin{align}
\bbE [ | A^m_t - A^m_s - (A^n_t - A^n_s)|^p ] &= \bbE \Big[ \Big| \int_s^t (\mathcal{A}^m - \calA^n)\ell_\phi (\eta_r) \dd r \Big|^p \Big] \nonumber\\
                                  &\lesssim (t - s)^{p / 2} \| \ic_{- \mathcal{L}_0 > n} \mathcal{A}_+ \ell_\phi \|_{\vsob{-1}{0}}^p \nonumber\\
                                  &\leqslant (t - s)^{p / 2} n^{- p(1-q)/2} \| \mathcal{A}_+\ell_\phi \|_{\vsob{-q}{0}}^p \label{eq:59}
\end{align}
Again, the RHS is finite for $q > 1/2$. This estimate also holds for the case in which $A^m$ is replaced with $A$.

We use these two estimates to prove the claim that for each $p \in [1, p_0]$, there is a constant $C > 0$, such that uniformly in $m \in \bbN$ and $0 \leqslant s \leqslant t \leqslant T$, it holds
\begin{equation}\label{eq:57}
\bbE [ | A^m_t - A^m_s |^p ] \leqslant C (t-s)^{p(1-q/2)}
\end{equation}
Indeed, if $|t - s|/T \leqslant m^{-1}$, then the estimate follows from \eqref{eq:60}. Otherwise, choose $1 \leqslant n < m$ such that $(n+1)^{-1} < |t - s|/T \leqslant n^{-1}$, and consider
$$ \bbE [ | A^m_t - A^m_s |^p ] \lesssim_p \bbE [ | A^n_t - A^n_s |^p ] + \bbE [ | A^m_t - A^m_s - (A^n_t - A^n_s)|^p ] $$
The desired control follows by using \eqref{eq:60} on the first term, in the case $m = n$, and \eqref{eq:59} on the second term. The claim is proven.

Now we are ready to prove \eqref{eq:56}. Starting with \eqref{eq:57} in the case $p = p_0$, and applying Kolmogorov's criterion \cite[Theorem 23.7]{Kallenberg21_FoundationsModern}, noting that we can choose $q > 1/2$ arbitrarily close to $q = 1/2$, we conclude that
$$ \sup_m \bbE [\| A^m \|_{C^{\alpha}_T}^{p_0}] < \infty, \quad \text{ for all } \alpha \in (0, 3/4 - 1/p_0) $$
To obtain the result for $A$ in place of $A^m$, the argument is similar. Choose $n \in \bbN$ such that $(n+1)^{-1} < |t - s|/T \leqslant n^{-1}$, and we bound
$$ \bbE [ | A_t - A_s |^p ] \lesssim_p \bbE [ | A_t - A_s - (A^n_t - A^n_s)|^p ] + \bbE [ | A^n_t - A^n_s |^p ] $$
The former term we control with \eqref{eq:59}, which as we mentioned holds in the case $m = \infty$, and the latter we apply \eqref{eq:57} in the case $m = n$.

Finally we prove \eqref{eq:58}. It follows from \eqref{eq:59} in the case $m = \infty$ that the conclusion holds in the regime $\alpha \in (0, 1/2-1/p_0)$. To obtain the full regime $\alpha \in (0, 3/4 - 1/p_0)$, we use the interpolation estimate:
$$ \|A - A^m\|_{C^{\theta\gamma_1+(1-\theta)\gamma_2}_T}\lesssim \|A - A^m\|_{C^{\gamma_1}_T}^\theta \|A - A^m\|_{C^{\gamma_2}_T}^{1-\theta} $$
We have already proven that for any $\gamma_1 \in (0, 3/4 - 1/p_0)$, we have $\bbE[\|A - A^m\|_{C^{\gamma_1}_T}^{p_0}] < \infty$ uniformly among $m \in \bbN$. For each fixed $\alpha \in (0, 3/4 - 1/p_0)$ it suffices to make an appropriate choice $\gamma_1 \in (\alpha, 3/4 - 1/p_0)$, $\gamma_2 \in (0,1/2-1/p_0)$, and $\theta \in (0, 1)$ such that $\alpha = \theta\gamma_1+(1-\theta)\gamma_2$. The claim then follows by raising the interpolation estimate to the power of $p_0$, and applying H\"{o}lder's inequality with exponents $(p^{-1}, q^{-1}) = (\theta, 1-\theta)$.
\end{proof}

\section{Existence of energy solutions}
\label{sec:existenceenergy}
In this section we show that there exists a process $\eta$ that satisfies the notion of energy solution as given in Definition \ref{def:energysolution}, using established techniques from energy solution theory \cite{GubinelliJara13,GubinelPerkows20_InfinitesimalGenerator}. We construct the solution by considering a natural family of approximations $(\eta^\varepsilon : \varepsilon \in (0, 1))$. First we show that the approximations are tight, and then that any subsequential limit point satisfies Definition \ref{def:energysolution}.

\subsection{The natural approximation and tightness}
We restate the approximation detailed before the statement of Theorem \ref{thm:main}. Let $\rho \in \calS(\bbR)$ be a positive definite mollifier of unit mass, $\rhoHat \geqslant 0$, $\int \rho(x) \dd x = 1$, and take the rescaling $\rho_\varepsilon(x) = \varepsilon^{-1} \rho(x/\varepsilon)$, for $\varepsilon > 0$. For $V \in \calS'(\bbR)$ satisfying Assumption \ref{ass:posdef} define $V^\varepsilon = \rho_\varepsilon * \rho_\varepsilon * V$ to be the approximation obtained by applying the rescaled mollifier twice. Consider $\pi_0 \ll \pi$ be a measure on $(\calS', \calF)$ for which $\lVert \dd \pi_0 / \dd \pi \rVert_{\vfock} < \infty$ and write $\pi^\varepsilon_0$ (resp. $\pi^\varepsilon$) for the law of $\omega^\varepsilon := \rho_\varepsilon * \omega$ under which $\omega \sim \pi_0$ (resp. $\pi(V)$). In particular, it holds
\begin{equation}\label{eq:21}
\lVert \dd \pi_0^\varepsilon / \dd \pi^\varepsilon \rVert_{\efock} \leqslant \lVert \dd \pi_0 / \dd \pi \rVert_{\vfock}, \quad \forall \varepsilon > 0
\end{equation}

We shall write $\eta^\varepsilon, \calL^\varepsilon$ for the environment, generator etc.\@ from the smooth setting of Section \ref{sec:mollified-srbp} corresponding to this choice $V = V^\varepsilon$ and with initial state sampled according to $\eta^\varepsilon_0 \sim \pi^\varepsilon_0$. We shall refer to such objects as the {\it natural approximation}.

\begin{proposition}\label{prp:tight}
Let $V \in \calS'(\bbR)$ satisfy Assumption \ref{ass:posdef} and let $\pi_0 \ll \pi$ be a measure on $(\calS', \calF)$ for which $\lVert \dd \pi_0 / \dd \pi \rVert_{\vfock} < \infty$. Let $\rho \in \calS(\bbR)$ be a fixed mollifier, and for $\varepsilon \in (0,1)$, let $\eta^\varepsilon$ be the natural approximation described above. Then it holds that $(\eta^\varepsilon : \varepsilon \in (0, 1))$ is tight as a family of measures on $C([0,T], \calS'(\bbR))$.

Moreover, writing $\eta^0$ for any subsequential limit point of the family, and writing $V^0 = V$, we have the estimates:
\begin{itemize}
\item For all $p \geqslant 1, T \geqslant 0$ there exists a constant $C(p, T)$ such that for all such $V, \beta, \pi_0, \rho$, and $\varepsilon \in [0, 1)$, the process $(\eta^\varepsilon)_{t \geqslant 0}$ satisfies the Itô trick with respect to $V^\varepsilon$ and with constant $C = C(p, T) \lVert \dd \pi_0 / \dd \pi \rVert_{\vfock}$.
\item For all $p \geqslant 1, T \geqslant 0$, and for all such $V, \beta, \pi_0, \rho$ there exists a constant $C = C(p, T, V, \beta, \pi_0, \rho)$ such that uniformly among $\varepsilon \in [0, 1)$, it holds
\begin{equation}\label{eq:11}
\bbE \big[|\eta^\varepsilon_t(\phi)-\eta^\varepsilon_s(\phi)|^p\big] \leqslant C (t-s)^{p/2}, \quad \forall 0 \leqslant s \leqslant t \leqslant T
\end{equation}
\end{itemize}
\end{proposition}

\begin{proof} The proof is structured as follows. First we prove the estimates in the case of $\varepsilon \in (0, 1)$. This allows us to deduce tightness for the sequence. At which point we finish by proving the estimates also in the case $\varepsilon = 0$.

We shall write $\bbP_0^\varepsilon$ for the law of the process $\eta^\varepsilon$ on the path space $C(\bbR_{\geqslant 0}, \bbR)$. Similarly, write $\bbP^\varepsilon$ for the time-stationary regime, in which $\pi_0^\varepsilon = \pi^\varepsilon$.

{\it Itô Trick for $\varepsilon \in (0, 1)$.} For $h \in t\calC$, define $u := (1 - \calL_0)^{-1} h$, with the operator acting on the field variable, $\omega \in \calW$, while keeping fixed $t \in [0, T]$.

Proposition \ref{prp:1}, implies that
\begin{equation}\label{eq:20}
\bbE^\varepsilon_0 \Big[ \sup_{t \in [0,T]} \Big\lvert \int_0^t \calL_0u(s, \eta^\varepsilon_s) \dd s \Big\rvert^p \Big] \leqslant C(p) \lVert \dd \pi_0 / \dd \pi \rVert_{\vfock} \lVert \nabla_0 u \rVert_{L^2_TL^{2p}(V^\varepsilon)}
\end{equation}
where we have also used the bound \eqref{eq:21}. Note that we have $\calL_0$ and $\nabla_0$ in the above, not $\calL_0^\varepsilon$ and $\nabla_0^\varepsilon$, because the action is independent of $\varepsilon \in (0, 1)$.
By hypercontractivity \cite[Theorem 5.10]{Janson97_GaussianHilbert}, it holds
\begin{equation}\label{eq:71}
\lVert \nabla_0 u \rVert_{L^2_TL^{2p}(V^\varepsilon)} \leqslant \lVert \sqrt{2p-1}^\calN \nabla_0 u \rVert_{L^2_TL^2(V^\varepsilon)} \leqslant \lVert \sqrt{2p-1}^\calN h \rVert_{L^2_T\esob{-1}{0}}
\end{equation}

Additionally, we have $ \sup_{t \in [0,T]} \Big\lvert \int_0^t u(s, \eta^\varepsilon_s) \dd s \Big\rvert \leqslant \int_0^T |u(s, \eta^\varepsilon_s)| \dd s $ and therefore by Cauchy-Schwarz, Minkowski's inequality, Jensen's inequality, and hypercontractivity:
\begin{align}
\bbE^\varepsilon_0 \Big[ \sup_{t \in [0,T]} \Big\lvert \int_0^t u(s, \eta^\varepsilon_s) \dd s \Big\rvert^p \Big] &\leqslant \lVert \dd \pi_0 / \dd \pi \rVert_{\vfock}  \big(\bbE^\varepsilon \Big[ \big( \int_0^T |u(s, \eta^\varepsilon_s)| \dd s \big)^{2p} \Big] \big)^\half \nonumber\\
&\leqslant \lVert \dd \pi_0 / \dd \pi \rVert_{\vfock} \big( \int_0^T \lVert u(s, \cdot) \rVert_{L^{2p}(V^\varepsilon)} \dd s \big)^p \nonumber\\
&\leqslant C(T) \lVert \dd \pi_0 / \dd \pi \rVert_{\vfock} \big( \int_0^T \lVert u(s, \cdot) \rVert_{L^{2p}(V^\varepsilon)}^2 \dd s \big)^{p/2} \nonumber\\
&\leqslant C(T) \lVert \dd \pi_0 / \dd \pi \rVert_{\vfock} \lVert \sqrt{2p-1}^\calN u(s, \cdot) \rVert_{L^2_T\efock}^p \nonumber\\
&\leqslant C(T) \lVert \dd \pi_0 / \dd \pi \rVert_{\vfock} \lVert \sqrt{2p-1}^\calN h(s, \cdot) \rVert_{L^2_T\esob{-1}{0}}^p \label{eq:19}
\end{align}

Combining \eqref{eq:20}, \eqref{eq:71}, and \eqref{eq:19} along with the identity $(a+b)^p \leqslant 2^p(a^p + b^p)$ yields the Itô trick with constant $C = 2^p(C(p) + C(T)) \lVert \dd \pi_0 / \dd \pi \rVert_{\vfock}$.

{\it Control of increments for $\varepsilon \in (0, 1)$.} Fix $\phi \in \calS(\bbR)$. We begin by reducing to the stationary case:
\begin{align}
\bbE^\varepsilon_0 \big[|\eta^\varepsilon_t(\phi)-\eta^\varepsilon_s(\phi)|^p\big] \leqslant \lVert \dd \pi^\varepsilon_0 / \dd \pi^\varepsilon \rVert_{\efock} \big( \bbE^\varepsilon \big[|\eta^\varepsilon_t(\phi)-\eta^\varepsilon_s(\phi)|^{2p} \big]\big)^{1/2} \nonumber\\
\leqslant \lVert \dd \pi_0 / \dd \pi \rVert_{\vfock} \big( \bbE^\varepsilon \big[|\eta^\varepsilon_{t - s}(\phi)-\eta^\varepsilon_0(\phi)|^{2p} \big]\big)^{1/2} \label{eq:70}
\end{align}
where in the last line we used \eqref{eq:21} and time-stationarity.

Let $(M^\varepsilon_t(\ell_\phi))_{t \geqslant 0}$ be the Dynkin martingale of Lemma \ref{lem:2}, in which $\ell_\phi$ is the linear functional $\ell_\phi(\eta) := \eta(\phi)$, and take the upper bound
\begin{equation}\label{eq:12}
\bbE^\varepsilon \big[|\eta^\varepsilon_t(\phi)-\eta^\varepsilon_0(\phi)|^{2p}\big] \lesssim_p \bbE^\varepsilon \Big[ \Big|\int_0^t \calL^\varepsilon\ell_\phi(\eta^\varepsilon_s) \dd s \Big|^{2p} \Big] + \bbE^\varepsilon \big[|M^\varepsilon_t(\ell_\phi)|^{2p}\big]
\end{equation}

For the first term of \eqref{eq:12}, we apply the Itô trick with $h(s, \omega) = \ic\{s \in [0, t]\} \calL^\varepsilon\ell_\phi(\omega)$. Although originally proved for cylinder functions $h$, the Itô trick extends to the present case via a density argument. This yields
$$ \bbE^\varepsilon \Big[ \Big|\int_0^t \calL^\varepsilon\ell_\phi(\eta^\varepsilon_s) \dd s \Big|^{2p} \Big] \lesssim_{p, T} t^p \lVert \calL^\varepsilon\ell_\phi \rVert_{\esob{-1}{0}}^{2p} $$
where we may omit the factor of $\sqrt{4p-1}^\calN$ because $\calL^\varepsilon\ell_\phi$ has no elements in chaos higher than $n = 2$. Applying Proposition \ref{prp:2}, we have
\begin{align*}
\lVert \calL^\varepsilon\ell_\phi \rVert_{\esob{-1}{0}}^2 \lesssim J^1_\varepsilon(1) \int \VHat^\varepsilon(p) |p|^2 | \phiHat(p)|^2 \dd p
\end{align*}
In which we write $J^1_\varepsilon(1)$ for the quantity defined in \eqref{eq:46} but in the case $V = V^\varepsilon$. Because $\VHat^\varepsilon(p) = \rhoHat(\varepsilon p)^2 \VHat(p)$, it holds that there is a constant $C(V, \beta, \rho, \phi) > 0$ such that uniformly in $\varepsilon \in (0, 1)$, it holds
\begin{equation}\label{eq:69}
J^1_\varepsilon(1),  \int \VHat^\varepsilon(p) |p|^2 | \phiHat(p)|^2 \dd p \lesssim C(V, \beta, \rho, \phi)
\end{equation}

For the second term of \eqref{eq:12} we use Burkholder-Davis-Gundy to obtain
\begin{align*}
\bbE^\varepsilon \big[ |M^\varepsilon_t(\ell_\phi)|^{2p} \big] \lesssim_p \bbE^\varepsilon\big[ \langle M^\varepsilon(\ell_\phi) \rangle^p_t \big] \leqslant t^p \lVert \nabla_0 \ell_\phi \rVert_{L^{2p}(V^\varepsilon)}^{2p}
\end{align*}
where in the last line we used the explicit form of the quadratic variation given in Lemma \ref{lem:2}, and Jensen's inequality. The uniform control in $\varepsilon \in (0, 1)$ follows from hypercontractivity:
$$\lVert \nabla_0\ell_\phi \rVert_{L^{2p}(V^\varepsilon)} \lesssim_p \lVert \nabla_0\ell_\phi \rVert_{\efock} \lesssim \lVert \ell_\phi \rVert_{\esob{1}{0}} = \int \VHat^\varepsilon(p) |p|^2 | \phiHat(p)|^2 \dd p $$
In particular, this can be controlled uniformly in $\varepsilon \in (0, 1)$ thanks to \eqref{eq:69}. Plugging all of these estimates back into \eqref{eq:70} in the case $t = t - s$ yields the estimate \eqref{eq:11}.

{\it Tightness.} Mitoma's criterion \cite{Mitoma83_TightnessProbabilities} implies that it is sufficient to fix $\phi \in \calS(\bbR)$ and show tightness of the family $\{\eta^\varepsilon(\phi) : \varepsilon \in (0,1)\}$ as a sequence of processes on $C([0,T], \bbR)$. Kolmogorov's criterion \cite[Theorem 23.7]{Kallenberg21_FoundationsModern} implies that it is enough to show that the initial state $(\eta^\varepsilon_0(\phi) : \varepsilon \in (0, 1))$ is tight as a family of measures on $\bbR$, and that for some fixed $p > 2$ we may control the increments as in \eqref{eq:11}. In which case it remains to prove tightness of the initial state. This is trivial, because $\eta_0^\varepsilon(\phi) \convd \eta(\phi)$ when $\eta \sim \pi_0$, so it remains to control the increments.

{\it Estimates in the case $\varepsilon = 0$} It follows from Kolmogorov's criterion \cite[Theorem 23.7]{Kallenberg21_FoundationsModern} that the control on the increments, as stated in \eqref{eq:11}, also holds with the same constants for any subsequential limit $\eta^0$.
As for the Itô trick, taking limits in \eqref{eq:20} along a subsequence, yields
$$ \bbE \Big[ \sup_{t \in [0,T]} \Big\lvert \int_0^t \calL_0 u(s, \eta_s) \dd s \Big\rvert^p \Big] \leqslant C(p) \lVert \dd \pi_0 / \dd \pi \rVert_{\vfock} \lVert \nabla_0 u \rVert^{p}_{L^2_TL^{2p}(V)} $$
Note that we have also used the weak convergence $\pi(V^\varepsilon) \rightarrow \pi(V)$. The Itô trick follows by repeating the same argument given above, in which the same constant is obtained.

\end{proof}

\subsection{Existence via the natural approximation}

\begin{theorem}\label{thm:exist}
Consider the hypothesis of Proposition \ref{prp:tight}. Any subsequential limit point $\eta$ of the tight family $(\eta^\varepsilon : \varepsilon \in (0, 1))$ is an energy solution in the sense of Definition \ref{def:energysolution}.
\end{theorem}

\begin{proof}
Let $\eta$ be a subsequential limit. That is, writing $\bbP_0$ for the law of $\eta$, and $\bbP^\varepsilon_0$ for the law of $\eta^\varepsilon$, it holds that $\bbP^\varepsilon_0 \rightarrow \bbP_0$ along a subsequence. We have already proven in Proposition \eqref{prp:tight} that the Itô trick holds, and therefore it remains to verify properties 2--4 of Definition \eqref{def:energysolution}.

{\it Incompressibility.} For any $F \in \calC$, it holds
\begin{align}
\bbE[|F(\eta_t)|] &= \lim_{\varepsilon\rightarrow0} \bbE^\varepsilon[|F(\eta^\varepsilon_t)|] \le \lim_{\varepsilon\rightarrow0} \lVert \dd \pi_0^\varepsilon / \dd \pi^\varepsilon \rVert_{\vfock} \big( \bbE^\varepsilon[|F(\eta^\varepsilon_t)|^2] \big)^{1/2}
\\
&\leqslant \lVert \dd \pi_0 / \dd \pi \rVert_{\vfock}  \lim_{\varepsilon\rightarrow0} \lVert F \rVert_{\efock} = \lVert \dd \pi_0 / \dd \pi \rVert_{\vfock} \lVert F \rVert_{\vfock}
\end{align}
where the last line follows from stationarity of $(\eta^\varepsilon_t)_{t \geqslant 0}$ under the initial condition $\eta^\varepsilon_0 \sim \pi^\varepsilon$.

{\it The martingale property.} For $\phi \in \calS$, $0 \leqslant s \leqslant t \leqslant T$ and for a bounded and continuous function $G : C([0, s], \calS') \rightarrow \bbR$, it holds by Lemma \ref{lem:2}
$$\bbE^\varepsilon_0 \Big[ \big( \eta_t^\varepsilon(\phi) - \eta_s^\varepsilon(\phi) - \int_s^t \calL^\varepsilon\ell_\phi (\eta_r^\varepsilon) \dd r \big) G( (\eta_r^\varepsilon)_{r \in [0, s]} )  \Big] = 0 $$
The martingale property for the limit process follows from the claim
\begin{align}
\lim_{\varepsilon\rightarrow0} \bbE^\varepsilon_0 \Big[ \eta_t^\varepsilon(\phi) G( (\eta_r^\varepsilon)_{r \in [0, s]} ) \Big]  &= \bbE \Big[ \eta_t(\phi) G( (\eta_r)_{r \in [0, s]} ) \Big], \quad \forall t \in [s, T] \label{eq:23}\\
\lim_{\varepsilon\rightarrow0} \bbE^\varepsilon_0 \Big[ \int_s^t \calL^\varepsilon\ell_\phi (\eta_r^\varepsilon) \dd r G( (\eta_r^\varepsilon)_{r \in [0, s]} ) \Big]  &= \bbE \Big[ \int_s^t \calL\ell_\phi (\eta_r) \dd r G( (\eta_r)_{r \in [0, s]} ) \Big] \label{eq:24}
\end{align}
where the RHS of \eqref{eq:24} is understood in the sense of Definition \ref{def:itoTrick}. We give only a proof of \eqref{eq:24} as it is the more challenging of the two.

The difficulty of \eqref{eq:24} is that the parameter $\varepsilon\in(0,1)$ appears ``diagonally'': attached both to the generator $\calL^\varepsilon$ and the process $\eta^\varepsilon$. Accordingly, we introduce another parameter $\kappa \geqslant \varepsilon$ and expand the difference of the LHS and RHS of \eqref{eq:24} according to
\begin{align}
&\bbE^\varepsilon_0 \Big[ \int_s^t \calL^\varepsilon \call_\phi (\eta_r^\varepsilon) \dd r G( (\eta_r^\varepsilon)_{r \in [0, s]} ) \Big] - \bbE^\varepsilon_0 \Big[ \int_s^t \calL^\kappa \call_\phi (\eta_r^\varepsilon) \dd r G( (\eta_r^\varepsilon)_{r \in [0, s]} ) \Big] \label{eq:25}\\
&+ \bbE^\varepsilon_0 \Big[ \int_s^t \calL^\kappa \call_\phi (\eta_r^\varepsilon) \dd r G( (\eta_r^\varepsilon)_{r \in [0, s]} ) \Big] - \bbE \Big[ \int_s^t \calL^\kappa \call_\phi (\eta_r) \dd r G( (\eta_r)_{r \in [0, s]} ) \Big] \label{eq:26}\\
&+ \bbE \Big[ \int_s^t \calL^\kappa \call_\phi (\eta_r) \dd r G( (\eta_r)_{r \in [0, s]} ) \Big] - \bbE \Big[ \int_s^t \calL \call_\phi (\eta_r) \dd r G( (\eta_r)_{r \in [0, s]} ) \Big] \label{eq:27}
\end{align}
To obtain \eqref{eq:24}, we show that these terms converge to zero in the regime where we first send $\varepsilon \rightarrow 0$, followed by $\kappa \rightarrow 0$. More precisely, we claim that \eqref{eq:26} converges to zero as $\varepsilon \rightarrow 0$, \eqref{eq:27} is independent of $\varepsilon$ and converges to zero as $\kappa \rightarrow 0$, and we have a uniform control of \eqref{eq:25} among $\varepsilon \in (0, \kappa)$ in terms of a vanishing function of $\kappa$.

To obtain the claim on \eqref{eq:25}, we begin with Cauchy-Schwarz and and application of the Itô trick for $\eta^\varepsilon$ with respect to $V^\varepsilon$ in the case $p = 2$:
\begin{equation}\label{eq:28}
\big| \bbE^\varepsilon_0 \Big[ \int_s^t (\calL^\varepsilon - \calL^\kappa)\ell_\phi (\eta_r^\varepsilon) \dd r G( (\eta_r^\varepsilon)_{r \in [0, s]} ) \Big] \big| \lesssim_{G, T, \pi_0} \lVert (\calL^\varepsilon - \calL^\kappa)\ell_\phi \rVert_{\esob{-1}{0}}
\end{equation}
It is crucial here that the Itô trick holds with a constant $C$ that is $\varepsilon$-independent. The RHS of \eqref{eq:28} may be computed explicitly:
\begin{equation}\label{eq:29}
(\calL^\kappa - \calL^\varepsilon)\ell_\phi = V^\varepsilon(\nabla\phi) - V^\kappa(\nabla\phi) = \int_{\bbR} \VHat(p) \widehat{\nabla\phi}(p) (\rhoHat(\varepsilon p)^2 - \rhoHat(\kappa p)^2) \dd p
\end{equation}
and the claim follows by performing a Taylor expansion of $\rhoHat$.

The claim on \eqref{eq:26} follows from the weak convergence $\eta^\varepsilon \rightarrow \eta$. The functional under consideration is not bounded, however, so to apply weak convergence we need an additional ingredient: a $p$-moment bound for some $p > 1$. This is also obtained by applying the Itô trick for $\eta^\varepsilon$ with respect to $V^\varepsilon$, along the lines of \eqref{eq:28} but for a higher moment, and with $\calL^\kappa$ in place of $(\calL^\kappa - \calL^\varepsilon)$. It remains to control $\lVert \calL^\kappa \ell_\phi \rVert_{\esob{-1}{0}}$ uniformly among $\varepsilon \in (0, 1)$. Lemma \ref{lem:1} gives us a uniform control of $\lVert \calL^\varepsilon \ell_\phi \rVert_{\esob{-1}{0}}$ (see \eqref{eq:22}) and the residual $\lVert (\calL^\varepsilon - \calL^\kappa) \ell_\phi \rVert_{\esob{-1}{0}}$ can be easily controlled, e.g.\@ through \eqref{eq:29}.

The claim on \eqref{eq:27} is similarly a consequence of the Itô trick for $\eta$ with respect to $V$, and the observation $\calL^\kappa \ell_\phi - \calL \ell_\phi = V(\nabla\phi) - V^\kappa(\nabla\phi)$. This concludes the proof of the martingale problem.

{\it The quadratic variation.} We repeat the strategy that we just used for the martingale property. We start with knowledge of the martingale property at scale $\varepsilon > 0$:
$$\bbE^\varepsilon_0 \Big[ \big( M^\varepsilon_t(\phi)^2 - \int_s^t|\eta^\varepsilon_r(\nabla \phi)|^2 \dd r \big) G( (\eta_r^\varepsilon)_{r \in [0, s]} )  \Big] = 0 $$
and then we take a limit as $\varepsilon \rightarrow 0$.

\end{proof}

\section{Uniqueness of energy solutions}\label{sec:uniq-energy-solut}
In this section we show that there is only one distribution on the path space satisfying the definition of energy solution.

\subsection{Solving Kolmogorov's Backward Equation}
The key property of $\calL$ required in this section is the bound
\begin{equation}\label{eq:34}
\calL \in L(\vsob{1}{2}, \vsob{-1}{0})
\end{equation}
In fact, Lemma \ref{lem:1} implies the stronger regularity $\calL \in L(\vsob{1}{1}, \vsob{-1}{0})$. However, in this section we assume only \eqref{eq:34} because it is all that is needed.

\begin{remark}\label{rmk:3}
The property of \eqref{eq:34} is just weaker than the Graded Sector Condition of e.g.\ \cite[Section 2.7.4]{KomoLandOlla12_FluctuationsMarkov}, which, in our language, asks for $\calL \in L(\calH^1_{2\gamma}(V), \calH^{- 1}_0(V))$ for some $\gamma \in (0, 1)$.
\end{remark}

We define $\calX(V) \assign L^2_T \vsob{1}{2} \cap H^1_T\vsob{-1}{0}$. We briefly explain why this is a natural choice. For $F \in \calX(V)$, it follows from \eqref{eq:34} that
\begin{equation}\label{eq:35}
\partial_t F, \calL F \in L^2_T\vsob{-1}{0}
\end{equation}
With this level of regularity, we can use Definition \ref{def:i} to give meaning to the additive functional $\int (\partial_t + \calL) F \dd t$. Such a process plays a crucial role in the proof of uniqueness (see Proposition \ref{martingalepropextension}).

The main theorem of this section is to show that we are able to construct solution of Kolmogorov's Backward Equation that live in the space $\calX(V)$.
\begin{theorem} \label{existenceKBE}
Let $V \in \calS'(\bbR)$ satisfy Assumption \ref{ass:posdef}, $\beta^2 > 0$, and suppose that $J^1(1) \leqslant 1/400$, where $J^s(\lambda)$ is as given in \eqref{eq:46}. Then, for all time horizons $T \geqslant 0$ and initial conditions $F_0 \in \calC$, there exists a weak solution $F \in \calX(V)$ to
\begin{equation}\label{eq:18}
\partial_t F = \calL F, \quad F (0) = F_0
\end{equation}
meaning that for all $G \in t\calC$ it holds $F(0) = F_0$ and
\begin{equation}\label{KBEweakformulation}
\int_0^T \langle  G(t), \partial_t F(t) \rangle_{\vsob{1}{0}, \vsob{-1}{0}} \dd t = \int_0^T \langle G(t), \calL F(t) \rangle_{\vsob{1}{0}, \vsob{-1}{0}} \dd t
\end{equation}
\end{theorem}

It is a consequence of \eqref{eq:35} that both sides of \eqref{KBEweakformulation} are well-defined. However, a priori it is not clear that $F(0)$ is well-defined for $F \in \calX(V)$. Before we prove Theorem \ref{existenceKBE}, we exhibit the following lemma which gives an embedding that ensures that $F(0)$ is well-defined in this case. This lemma is essentially an application of the fundamental theorem of calculus.
\begin{lemma}\label{embeddingcont}
For each time horizon $T \geqslant 0$, there is a continuous embedding
$$L^2_T \calH^1_0 (V) \cap H^1_T \calH^{- 1}_0 (V) \subset C_T \vfock$$
Moreover, for $F\in L^2_T \calH^1_0 (V) \cap H^1_T \calH^{- 1}_0 (V)$, it holds
\begin{equation}\label{differentialformulahnorm}
\| F (t) \|^2_{\vfock} = \| F (s) \|^2_{\vfock} + 2 \int_s^t \langle F (r) , \partial_tF(r) \rangle_{\calH^1_0 (V),\calH^{- 1}_0 (V)} \dd r .
\end{equation}
\end{lemma}

\begin{proof}
Our proof is based on \cite[{\textsection}5.9.2, Theorem 3]{Evans10_PartialDifferential}. Let $\rho \in C^\infty_c(\bbR)$, $\rho \geqslant 0$, be a smooth mollifier satisfying $\supp(\rho) \subset B_1(0)$ and define $\rho_\varepsilon = \varepsilon^{- 1} \rho(\varepsilon^{- 1} \cdummy)$. Let $F \in L^2_T \mathcal{H}^1_0 \cap H^1_T \mathcal{H}^{- 1}_0$ and define
\[ F^{\varepsilon} (t) = \rho_{\varepsilon} F (t) = \int_{\bbR} \rho_{\varepsilon} (t - s) F (s) \dd s, \quad t \in \bbR \]
where we extend $F (t)$ to be $0$ for $t$ outside of its original domain. Then we have $F^{\varepsilon} \in C^{\infty}_T \mathcal{H}^1_0$ and moreover $F^{\varepsilon} \in H^1_T \mathcal{H}^{-1}_0$ with weak time-derivative $\partial_tF^\varepsilon = \rho_\varepsilon \partial_tF$. It also holds $\| F^{\varepsilon} \|_{L^2_T\mathcal{H}^1_0} \lesssim \| F \|_{L^2_T \mathcal{H}^1_0}, \| F^{\varepsilon} \|_{H^1_T \mathcal{H}^{- 1}_0} \lesssim \| F \|_{H^1_T \mathcal{H}^{- 1}_0}$ and $F^{\varepsilon} \rightarrow F$ in $L^2_T \mathcal{H}^1_0$ and $H^1_T \mathcal{H}^{- 1}_0$. 
By considering $\frac{\dd}{\dd t}\| F^{\varepsilon} (t) \|^2$ we obtain for $s, t \in [0, T]$,

\begin{equation}
\| F^{\varepsilon} (t) \|^2 = \| F^{\varepsilon} (s) \|^2 + 2 \int_s^t \langle F^{\varepsilon} (r), \partial_tF^{\varepsilon} (r) \rangle \dd r . \label{maybetherelevantformula}
\end{equation}
Taking a limit as $\varepsilon \rightarrow 0$ proves \eqref{differentialformulahnorm}. Moreover, integrating \eqref{maybetherelevantformula} over $s \in [0, T]$ yields
\begin{align*}
\| F^{\varepsilon} (t) \|^2 &= T^{- 1} \int_0^T \|F^{\varepsilon} (s) \|^2 \dd s + 2 T^{- 1} \int_0^T \int_s^t \langle F^{\varepsilon} (r) , \partial_tF^{\varepsilon} (r) \rangle \dd r \dd s \\
&\leqslant \lVert F^\varepsilon \rVert_{L^2_T \vfock}^2 + \lVert F^\varepsilon \rVert_{L^2_T \vsob{1}{0}} \lVert \partial_tF^\varepsilon \rVert_{L^2_T \vsob{-1}{0}} \\
&\lesssim \lVert F^\varepsilon \rVert_{L^2_T \vsob{1}{0}}^2 + \lVert F^\varepsilon \rVert_{H^1_T \vsob{-1}{0}}^2
\end{align*}
Taking a supremum over $t \in [0, T]$ implies $\lVert F^\varepsilon \rVert_{C_T\vfock} \lesssim \lVert F^\varepsilon \rVert_{L^2_T \vsob{1}{0}} + \lVert F^\varepsilon \rVert_{H^1_T \vsob{-1}{0}}$. This proves the embedding $L^2_T \mathcal{H}^1_0 \cap H^1_T \mathcal{H}^{- 1}_0 \subset C_T \vfock$.

\end{proof}

\begin{proof}[Proof of Theorem \ref{existenceKBE}]
For $m \in \bbN$, define the truncation $\calA^m_\pm$ of $\calA_\pm$ according to
$$ \mathcal{A}^m_{\pm} := \ic_{(1 - \mathcal{L}_0), \mathcal{N} \leqslant m} \mathcal{A}_\pm \ic_{(1 - \mathcal{L}_0), \mathcal{N} \leqslant m} $$
By construction, it holds that $\mathcal{A}^m_{\pm} \in L(\vsob{\alpha}{\beta})$ for all $\alpha, \beta \in \bbR$.  As a consequence of Lemma \ref{lem:1}, it holds that uniformly among $m \in \bbN$,
\begin{equation}\label{eq:31}
\| \calA^m_+ - \calA^m_- \|_{L(\calH^1_2, \calH^{- 1}_0)} \leqslant \| \calA^m_+ \rVert_{L(\calH^1_2, \calH^{- 1}_0)} + \lVert \calA^m_- \|_{L(\calH^1_2, \calH^{- 1}_0)} \leqslant 10 \sqrt{J^1(1)}
\end{equation}
Due to the smallness assumption on $\beta^2$, it holds that the RHS is bounded by $c = 1/2$. 

Define $\mathcal{L}^m = \mathcal{L}_0 + \mathcal{A}^m_+ + \calA^m_-$. Our strategy is to begin by solving the truncated equation
\begin{equation}\label{approxabstractPDE}
F^m (t) = F_0 + \int_0^t \mathcal{L}^m  F^m (s) \dd s
\end{equation}
It is a standard argument (see the appendix, Section \ref{sec:proof-abstractpde}) that this equation exhibits a solution $F^m \in \calX(V)$. At which point, we seek a bound $\lVert F^m \rVert_{\calX(V)} \leqslant C$ that is uniform in $m \in \bbN$, at which point we can invoke weak-compactness to obtain a subsequential limit, $F^m \convw F$.

Consider, for the moment, the easier task of obtaining a uniform bound $\lVert F^m \rVert_{L^2_T \vsob{1}{0}} \leqslant C$. Combining \eqref{approxabstractPDE} and \eqref{maybetherelevantformula} yields
\begin{equation}\label{eq:30}
\begin{split}
\| F^m(t) \|^2 &= \| F_0 \|^2 + 2 \int_0^t \langle F^m (r), \calL^mF^m (r) \rangle \dd r \\
&= \| F_0 \|^2 + 2 \int_0^t \| F^m(r) \|^2 \dd r  -2\lVert F^m \rVert_{L^2_T \vsob{1}{0}}^2
\end{split}
\end{equation}
where we have used that the symmetric part of $1-\calL^m$ is given by $1-\calL_0$, independently of $m \in \bbN$. Dropping the last term on the RHS of \eqref{eq:30} and applying Gronwall implies that $\| F^m(t) \|^2 \leqslant \| F_0 \|^2 e^{2t}$, which is uniform among $m \in \bbN$. Plugging this back into \eqref{eq:30} yields a uniform bound $\lVert F^m \rVert_{L^2_T \vsob{1}{0}}^2 \lesssim_T \| F_0 \|^2$.

We can obtain a uniform bound for the stronger norm, $\lVert F^m \rVert_{L^2_T \vsob{1}{2}} \leqslant C$, by running the same argument with $(1 + \mathcal{N}) F^m$ in place of $F^m$:
\begin{align*}
\| F^m(t) \|^2_{\vsob{0}{2}} &= \| F_0 \|^2_{\vsob{0}{2}} + 2 \int_0^t \langle (1 + \mathcal{N})F^m (r), (1 + \mathcal{N}) \calL_m F^m (r) \rangle_{\vsob{-1}{-2}, \vsob{1}{2}} \dd r \\
&= \| F_0 \|^2_{\vsob{0}{2}} + 2 \int_0^t \| F^m(r) \|^2_{\vsob{0}{2}} \dd r - 2\lVert F^m \rVert_{L^2_T \vsob{1}{2}}^2 \\
&\qquad - 2\int_0^t \langle (1 + \mathcal{N})F^m (r), [(1 + \mathcal{N}), \calL_m] F^m (r) \rangle_{\vsob{-1}{-2}, \vsob{1}{2}} \dd r
\end{align*}
In comparison with \eqref{eq:30}, we have an additional term that comes from commuting $(1+\calN)$ and $\calL_m$, producing the commutator $[(1 + \mathcal{N}), \calL_m] := (1 + \mathcal{N}) \calL_m - \calL_m (1 + \mathcal{N})$. Using the observation $[(1 + \mathcal{N}), \calL_m] = \calA^m_+ - \calA^m_-$, we control this additional term simply according to
$$ |\int_0^t \langle (1 + \mathcal{N})F^m (r), [(1 + \mathcal{N}), \calL_m] F^m (r) \rangle_{\vsob{-1}{-2}, \vsob{1}{2}} \dd r | \leqslant c \lVert F^m \rVert_{L^2_T \vsob{1}{2}}^2 $$
where $c = 1/2$ comes from the RHS of \eqref{eq:31}. In total, we have derived
$$ \| F^m(t) \|^2_{\vsob{0}{2}} \leqslant \| F_0 \|^2_{\vsob{0}{2}} + 2 \int_0^t \| F^m(r) \|^2 \dd r - 2(1-c)\lVert F^m \rVert_{L^2_T \vsob{1}{0}}^2 $$
Crucially, the last term is still negative, thanks to $c \in (0, 1)$. From here we repeat the Gronwall argument to obtain the uniform bound among $m \in \bbN$ that $\lVert F^m \rVert_{L^2_T \vsob{1}{2}} \lesssim_{c, T} \| F_0 \|_{\vsob{0}{2}}$.

To obtain the uniform bound $\lVert F^m \rVert_{\calX(V)} \leqslant C$ it remains to control $\lVert F^m \rVert_{H^1_T\vsob{-1}{0}}$. This control is obtained due to the fact that $F^m$ solves equation \eqref{approxabstractPDE}, and that $\calL^m \in L(\calH^1_2 (V),\calH^{- 1}_0(V))$ with operator norm uniformly bounded among $m \in \bbN$.

The uniform bounds allow for us to apply the Banach-Alaoglu theorem, and we obtain a weakly convergent subsequence, $F^m \convw F$ in $\calX(V)$. By Lemma \ref{embeddingcont}, consideration of the linear functional $C_T \vfock \ni F \mapsto F(0)$ implies that $F(0) = F_0$.

It remains to show that $F$ fulfils \eqref{KBEweakformulation}. To that end, fix $G \in t\calC$. Equation \eqref{approxabstractPDE} implies that

\begin{eqnarray}
\langle \partial_t F^m, G \rangle_{L^2_T \mathcal{H}^{- 1}_0, L^2_T \mathcal{H}^1_0} & = & \langle \mathcal{L}^m F^m, G \rangle_{L^2_T \mathcal{H}^{- 1}_0, L^2_T \mathcal{H}^1_0} \nonumber\\
& = & \langle (\mathcal{L}^m - \mathcal{L}) F^m, G \rangle_{L^2_T \mathcal{H}^{- 1}_0, L^2_T \mathcal{H}^1_0} \nonumber\\
&  & + \langle \mathcal{L} (F^m - F), G \rangle_{L^2_T \mathcal{H}^{- 1}_0, L^2_T \mathcal{H}^1_0} + \langle \mathcal{L} F, G \rangle_{L^2_T \mathcal{H}^{- 1}_0, L^2_T \mathcal{H}^1_0} .\label{approxPDE}
\end{eqnarray}
Provided that we can show that the first two terms on the RHS vanish in the limit $m \rightarrow \infty$, then the result follows from the weak convergence $\partial_t F^m \convw \partial_t F$ in $L^2_T \mathcal{H}^{- 1}_0$.
For the first term, we have
\begin{align*}
\langle (\mathcal{L}^m - \mathcal{L})F^m, G \rangle_{L^2_T \mathcal{H}^{- 1}_0, L^2_T \mathcal{H}^1_0} & = \langle \mathcal{A} (\ic - \ic_{(1 - \mathcal{L}_0), \mathcal{N} \leqslant m}) F^m, G \rangle_{L^2_T \mathcal{H}^{- 1}_0, L^2_T \mathcal{H}^1_0}\\
& = -\langle F^m, (\ic - \ic_{(1 - \mathcal{L}_0), \mathcal{N} \leqslant m}) \mathcal{A} G \rangle_{L^2_T \mathcal{H}^1_2, L^2_T \mathcal{H}^{- 1}_{-2}}\\
& \leqslant_{\text{abs}} \| F^m \|_{L^2_T \mathcal{H}^1_2} \| (\ic - \ic_{(1 - \mathcal{L}_0), \mathcal{N} \leqslant m}) \mathcal{A} G \|_{L^2_T \mathcal{H}^{- 1}_{-2}},
\end{align*}
which vanishes as $m \rightarrow \infty$ due to the uniform control on $\| F^m \|_{\calX(V)}$ and dominated convergence. For the second term, we use the weak convergence $F^m \rightarrow F$ in $L^2_T \mathcal{H}^1_2$ to conclude
$$ \langle \mathcal{L} (F^m - F), G \rangle_{L^2_T  \mathcal{H}^{- 1}_0, L^2_T \mathcal{H}^1_0} = \langle F^m -  F, \mathcal{L}^{\ast} G \rangle_{L^2_T \mathcal{H}^1_2, L^2_T  \mathcal{H}^{- 1}_{-2}} \rightarrow 0 $$

\end{proof}

\subsection{Uniqueness of the cylinder function martingale problem}
The goal of this section is to prove Theorem \ref{cylmartprob}, which gives us uniqueness of solutions for the cylinder martingale problem. We finish this section by proving our main theorem concerning the convergence of the environment, Theorem \ref{thm:env}, which is a short application of the results from this section and from Section \ref{sec:existenceenergy}.

\begin{theorem}\label{cylmartprob}
Let $V \in \calS'(\bbR)$ satisfy Assumption \ref{ass:posdef}, let $\beta^2 > 0$. Then, given a distribution $\pi_0 \ll \pi$, there is at most one law on the path space $C(\bbR_{\geqslant 0}, \calS'(\bbR))$ for which $\eta$ is a solution to the cylinder function martingale problem, Definition \ref{def:mcyl}, with initial state distribution $\law(\eta_0) = \pi_0$. Moreover, if $\eta$ is such a solution, then it is a Markov process.
\end{theorem}

As previously mentioned, a key object in the proof of Theorem \ref{cylmartprob} will be the Dynkin martingales associated to elements $F \in \calX(V)$.

\begin{lemma}\label{martingalepropextension}
Let $\eta$ be a solution to the cylinder function martingale problem associated to $\calL$. Then for all $F \in \calX(V)$, it holds that $F(t, \eta_t) - F (0, \eta_0) - \int (\partial_t + \calL) F(s, \eta_s) \dd s$ is a martingale in the filtration generated by $\eta$.
\end{lemma}

\begin{proof}
Let $(F^n)_{n \in \bbN} \subset t\calC$ such that $F^n \rightarrow F$ in $\calX(V)$. By Lemma \ref{embeddingcont} we have that $F^n \rightarrow F$ in $C_T \vfock$, therefore, by incompressibility, it holds for all $s \in [0, T]$ that $F^n(s, \eta_s) \rightarrow F(s, \eta_s)$ in $L^1(\bbP)$. Furthermore, by \eqref{eq:34},  $(\partial_t + \calL) F^n \rightarrow (\partial_t + \calL) F$ in $L^2_T\vsob{-1}{0}$ so that by the Itô trick in the case $p = 1$, it holds $I_t (\partial_t + \calL) F^n \rightarrow I_t (\partial_t + \calL) F$ in $L^1(\bbP)$. In conclusion, we have that for all $t \in [0, T]$, it holds
\begin{equation}\label{eq:36}
F(t, \eta_t) - F (0, \eta_0) - I_t( (\partial_t + \calL) F ) = \lim_n \big\{ F^n(t, \eta_t) - F^n(0, \eta_0) - I_t( (\partial_t + \calL) F^n ) \big\}
\end{equation}
where the limit is taken in $L^1(\bbP)$.

Define $M_t$ (resp. $M^n_t$) according to the LHS (resp. RHS) of \eqref{eq:36}. By the definition of solution to the cylinder function martingale problem, it holds that $(M^n_t)_{t \in [0, T]}$ is a martingale: for all $0 \leqslant s \leqslant t \leqslant T$ and bounded and continuous $G : C([0, s], \calS')$ it holds
\begin{equation}\label{eq:37}
\bbE[ M^n_t G(\eta_r : r \in [0, s]) ] = \bbE[ M^n_s G(\eta_r : r \in [0, s]) ]
\end{equation}
The $L^1(\bbP)$ convergence of \eqref{eq:36} allows for us to pass to the limit in \eqref{eq:37} and conclude that $(M_t)_{t \in [0, T]}$ is also a martingale.
\end{proof}

One further obstacle in the proof of Theorem \ref{cylmartprob} is the smallness assumption that was needed in Theorem \ref{existenceKBE}. To overcome this assumption, we use the fact that our notion of solution behaves as expected under diffusive scaling. This is the content of the following proposition. We omit the proof, for which the only task is the rescale the three conditions appearing in Definition \ref{def:mcyl}.

\begin{proposition}\label{prp:4}
Define the rescaling map $Q_N : C(\bbR_{\geqslant 0}, \calS'(\bbR)) \rightarrow C(\bbR_{\geqslant 0}, \calS'(\bbR))$ according to $Q_N \eta(t, x) := N^{-1/2} \eta(t/N^2,x/N)$. If $\eta$ is a solution of the cylinder function martingale problem corresponding to parameters $(V, \beta, \pi_0)$, in which $\pi_0 := \law(\eta_0)$, then $Q_N\eta$ is a solution corresponding to $(V_N, \beta_N, Q_N\pi_0)$, in which $V_N(x) := N^{-1} V(N^{-1} x)$, $\beta_N = N^{-1/2} \beta$ and $Q_N\pi_0$ is defined to be the law of $N^{-1/2} \omega(N^{-1} x)$ when $\omega \sim \pi_0$.
\end{proposition}

\begin{proof}[Proof of Theorem \ref{cylmartprob}]
Suppose first that $(V, \beta)$ are such that $J^1(1) < 1/400$, as this puts us in the domain of Theorem \ref{existenceKBE}. Therefore, for $F_0 \in \calC$, there exists $F \in \calX(V)$ which solves the Kolmogorov Backward Equation. Define $\FHat \in \calX_T(V)$ according to $\FHat(t) = F(T - t)$ for $t \in [0, T]$. Then by construction, $(\partial_t + \mathcal{L} )\FHat = 0$ as an element of $\vsob{-1}{0}$ and therefore by Lemma \ref{martingalepropextension}, it holds that
\begin{equation}\label{dualityKBE}
\bbE [F_0(\eta_T)] = \bbE [F(T, \eta_0)]
\end{equation}
This implies uniqueness of one-dimensional distributions.

Uniqueness of the process follows provided we can show uniqueness of the $N$-dimensional distributions for arbitrary $N \in \bbN$. This is derived in much the same way as \cite{GubinelPerkows20_InfinitesimalGenerator}, by performing an induction in $N$. TO that end, let us write $\calX_T(V)$ in place of $\calX(V)$ to emphasise the dependence on $T$.

Assume the uniqueness holds for $N \in \bbN$, and fix $0 \leqslant t_1 < \ldots < t_{N + 1} \leqslant T$. Let $h_i \in \calC$ for $i \in 1{:}(N+1)$. Let $F \in \calX_{t_{N+1}}(V)$ be the solution of the Kolmogorov Backward Equation with $F_0 = h_{N+1}$. Then, again due to the above martingale property of Lemma \ref{martingalepropextension}, it holds
\begin{equation}\label{finitediminductive}
\bbE[h_1(\eta_{t_1}) \ldots h_{N + 1}(\eta_{t_{N + 1}})] = \bbE [h_1 (\eta_{t_1}) \ldots h_N (\eta_{t_N}) \FHat (t_N, \eta_{t_N})]
\end{equation}
Uniqueness follows by the inductive hypothesis.
again by incompressibility and denseness of $\vsob{1}{1}$ in $\vfock$.

The Markov property can be concluded from \eqref{finitediminductive} which shows that
$$ \bbE [h_{N + 1} (\eta_{t_{N + 1}}) \mid (\eta_t)_{t \in [0, t_N]} ] = \FHat (t_N, \eta_{t_N}) $$

It remains to relax the assumption of smallness. For general $(V, \beta)$ case, we invoke Proposition \ref{prp:4}. This tells us that $Q_N\eta$ solves the cylinder function martingale problem with respect to the rescaled parameters $(V_N, \beta_N, Q_N\pi_0)$, where the rescaled parameters are as given in that proposition. Let $J^s_N(\lambda)$ denote the quantity of \eqref{eq:46} in the case of $(V, \beta) = (V_N, \beta_N)$. It holds $J^1_N(1) = J^1(N^2)$.

Due to Proposition \ref{prp:iib}, it holds that $\lim_{N \rightarrow \infty} J^1_N(1) = 0$, so in particular, we may take $N \in \bbN$ large enough so that $J^1_N(1) < 1/400$. In particular, the above argument ensures that the law of $Q_N \eta$ is uniquely determined, and that $Q_N\eta$ is a Markov process. These properties are preserved upon inverting the map $Q_N$.
\end{proof}

As a consequence of the duality identity \eqref{dualityKBE}, varying $\pi_0 \ll \pi$ directly leads to the following
\begin{corollary}
The solutions in Theorem \ref{existenceKBE} are unique in terms of the initial condition $F_0$.
\end{corollary}

\begin{proof}[Proof of Theorem \ref{thm:env}]
Fix $T \geqslant  0$. Thanks to Proposition \ref{prp:tight}, we know that $(\eta^\varepsilon : \varepsilon \in (0, 1))$ is tight as a sequence of measures on the path space $C([0, T], \calS'(\bbR))$, and we know from Theorem \ref{thm:exist} that any subsequential limit $\eta$ is an energy solution. Theorem \ref{thm:etoc} tells that $\eta$ is a solution to the cylinder function martingale problem, in which case the result follows after applying Theorem \ref{cylmartprob}.

The only remaining claim to be proved is that $\pi_0 = \pi$ yields a time-stationary solution. This property is inherited from the natural approximation. Indeed,
$$ \law(\eta_t) = \lim_{\varepsilon \rightarrow 0} \law(\eta^\varepsilon_t) = \lim_{\varepsilon \rightarrow 0} \pi^\varepsilon = \pi $$
\end{proof}
\begin{remark}\label{rmk:criticality}
The majority of our analysis depends only on the finiteness $J^1(\lambda) < \infty$. For the purpose of this work, we consider only dimension $d = 1$, in which case it is convenient to impose the equivalent condition, Assumption \ref{ass:posdef}. In general dimensions, one could use our techniques to prove existence and uniqueness, working directly under the assumption $J^1(\lambda) < \infty$. This is of interest because such a condition is invariant under the diffusive scaling of Proposition \ref{prp:4}, which is to say that it is scaling critical. 
\end{remark}
\section{Construction and properties of the polymer}\label{sec:limit-srbp}
The goal of this section it to go from the SPDE describing the environment $\eta$ to the diffusion $X$. First, in Subsection \ref{subsec:constructionX}, we will first give a canonical construction of $X$, Definition \ref{defconstructionX}, and then show in Theorem \ref{thm:Xconvergence} that it is the limit in law of the natural approximation $(X^\varepsilon : \varepsilon \in (0, 1))$. Then, in Subsection \ref{sec:char-srbp}, we give a notion of energy solution $X$ to the SDE and prove in Theorem \ref{wellpsoednessenergysolutionsX} that this notion is well-posed and agree with the construction from Definition \ref{defconstructionX}. Finally, in Subsection \ref{subsec:superdiffusivity} we then show that $X$ behaves superdiffusively, using the connection to the environment process $\eta$ and a more fine-grained analysis of its generator.

\subsection{Construction of SRBP and convergence}\label{subsec:constructionX}
In the smooth setting, the polymer is given in terms of the environment and the Brownian motion according to the identity \eqref{eq:1}, which we restate here for convenience,
\begin{equation}\label{Xfrometa}
X_t = B_t - \beta \int_0^t f(\eta_s) \dd s,
\end{equation}
in which $f(\omega) = \omega(0)$. As for the singular setting, we have now constructed the environment $\eta$, and we aim to use this identity as a definition of the polymer.

To that end, suppose that $(\eta_t)_{t \geqslant 0}$ is an energy solution, as in Definition \ref{def:energysolution}. To define the additive functional in \eqref{Xfrometa} we use Definition \ref{def:i}, for which we need only check that $f \in \vsob{-1}{0}$. This follows from Assumption \ref{ass:posdef}, which ensures
\[ \|f\|_{\vsob{-1}{0}}^2 = \int_\bbR \frac{\VHat(p)}{1 + \thalf |p|^2} \dd p < \infty . \]

As for the Brownian motion $(B_t)_{t \geqslant 0}$ in \eqref{Xfrometa}, this also needs to be recovered from $\eta$. This is a classical problem when working with martingale problems: although we think of $\eta$ as being a solution to the SPDE \eqref{eq:spde}, we never explicitly reference the noise $(B_t)_{t \geqslant 0}$ that drives the equation. The noise is present only implicitly within the martingales $M(\ell_\phi)$ of \eqref{eq:32}, which ought to obey the representation $M_t(\ell_\phi) = \int_0^t \eta_s(\nabla \phi) \dd B_s$. To recover $(B_t)_{t \geqslant 0}$, we have to invert this stochastic integral.

To that end, fix a non-zero $\phi\in\mathcal{S}$, and define
\begin{equation}\label{eq:4}
B_t := \int_0^t \frac{\ic\{\eta_s(\nabla \phi)\neq 0\}}{\eta_s(\nabla\phi)} \dd M_s(\ell_\phi) .
\end{equation}
To see that this stochastic integral is well-defined, we compute the expected quadratic variation:
\begin{align*}
\bbE \Big[\int_0^t \Big|\frac{\ic\{\eta_s(\nabla \phi)\neq 0\}}{\eta_s(\nabla \phi)}\Big|^2 \dd \langle M(\ell_\phi)\rangle_s\Big] &= \bbE \Big[\int_0^t \ic\{\eta_s(\nabla \phi)\neq 0\} \dd s\Big] \\
&= \int_0^t \bbP(\eta_s(\nabla \phi)\neq 0)\dd s = t
\end{align*}
In the first line, we used \eqref{eq:qvenergy}. In the last line we used the absolute continuity $\law (\eta_s) \ll \pi$, which is derived from the incompressibility condition.

To see that $B$ is a Brownian motion, we use Lévy's characterisation. We have $t - \langle B\rangle_t = \int_0^t \ic\{\eta_s(\nabla \phi) = 0\} \dd s$, which is non-negative. By the argument above, it is expectation zero, and therefore zero almost surely. This implies that $B$ is indeed a Brownian motion.

Moreover, it holds that $B$ is independent of our choice of $\phi$. Indeed, consider elements $\phi_1, \phi_2 \in \calS$. Applying \eqref{eq:qvenergy} to $\phi = \phi_1+\phi_2$ we obtain the quadratic covariation
\[\langle M(\ell_{\phi_1}),M(\ell_{\phi_2})\rangle_t=\int_0^t \eta_s(\nabla \phi_1)\eta_s(\nabla \phi_2)\dd s\]
Denoting the corresponding Brownian motions by $B^{\phi_1},B^{\phi_2}$, one has
\begin{align*}
\bbE \langle B^{\phi_1}-B^{\phi_2}\rangle_t &= \bbE\langle B^{\phi_1}\rangle_t-2 \bbE\langle B^{\phi_1},B^{\phi_2}\rangle_t+ \bbE\langle B^{\phi_2}\rangle_t \\
&=2t-2 \bbE\big[ \int_0^t \frac{\ic_{\eta_s(\nabla \phi_1)\neq 0}}{\eta_s(\nabla \phi_1)} \frac{\ic_{\eta_s(\nabla \phi_2)\neq 0}}{\eta_s(\nabla \phi_2)} \dd \langle M(\ell_{\phi_1}),M(\ell_{\phi_2})\rangle_s \big] \\ &= 0
\end{align*}
\begin{definition}\label{defconstructionX}
Let $(\eta_t)_{t \geqslant 0}$ be an energy solution with coupling constant $\beta^2 > 0$ and initial state distribution $\pi_0 := \law(\eta_0)$. That is to say, $\eta \sim \text{ENV}(V, \beta, \pi_0)$. We write $X=\mathcal{X}^V(\eta)$ for the process $X$ given according to \eqref{Xfrometa}, in which the additive functional and the Brownian motion are derived from $\eta$ in the above way, and we write $\text{SRBP}(V, \beta, \pi_0)$ for the law of $X$ on the path space $C(\bbR_{\geqslant 0}, \bbR)$.
\end{definition}

We now prove Theorem \ref{thm:main}, which we restate here in a slightly stronger form so as to include the convergence of the random environment.
\begin{theorem}\label{thm:Xconvergence}
Let $V \in \calS'(\bbR)$ satisfy Assumption \ref{ass:posdef}, let $\pi_0 \ll \pi(V)$ be such that $\dd \pi_0 / \dd \pi \in L^2(V)$, and let $\eta \sim ENV(V, \beta, \pi_0)$ be the environment process as given by Theorem \ref{thm:env}. Recall the definition of the natural approximation $(\omega^\varepsilon, X^\varepsilon)$ from Section \ref{sec:model-main-results} corresponding to the mollified interaction $V^\varepsilon$. Then, as distributions on $\mathcal{S}'(\bbR) \times C(\bbR_+,\bbR)$, it holds
\begin{equation}\label{eq:convergencejointlaw}
(\omega^\varepsilon,X^\varepsilon) \convd (\eta_0,\mathcal{X}^V(\eta))
\end{equation}
In particular, the limit is independent of the choice of mollifier $\rho \in \calS(\bbR)$ used in the natural approximation.
\end{theorem}

We give a rough idea of the proof. We have proven in Theorem \ref{thm:env} that $\eta^\varepsilon \rightarrow \eta$. The result for the polymer can be derived from this fact, essentially because the polymer is a function of the environment, as in \eqref{Xfrometa}. If the polymer were a continuous bounded function of the environment, then we would be done. Unfortunately, the relationship is more complicated, so we must make an approximation argument.

The additive functional in \eqref{Xfrometa} is in fact the easiest term, and can be approximated by the continuous bounded function $\int_0^\cdummy \eta_t(\rho_{1/m}) \dd t$, where we continue to write $\rho_\kappa$ for the ``sharpened'' mollifier, $\rho_\kappa(x) = \kappa^{-1} \rho(\kappa^{-1} x)$. . The Brownian motion, on the other hand, is more delicate: it's definition, \eqref{eq:4}, involves the reciprocal function $x \mapsto x^{-1}$; a stochastic integral against the martingale $(M(\ell_\phi)_t)_{t\geqslant 0}$; and the martingale itself features a term $\int_0^t \calA \ell_\phi(\eta_s) \dd s$ which is not a continuous function of $(\eta_t)_{t\in[0,T]}$, but is given by Definition \ref{def:i}. We have to approximate each of these dependencies one-by-one to obtain a bounded continuous function $F^m : C([0, T], \mathcal{S}'(\bbR)) \rightarrow  C([0, T], \bbR)$ for which $X \approx F^m(\eta)$.

We begin with Lemma \ref{lem:abstractF}, which details the requirements of our approximation $(F^m : m \in \bbN)$. Then, all that remains is to spell out our specific choice of $F^m$, in which each discontinuous dependency has been replaced by a continuous approximation. Throughout the rest of this section, we write $C_T := C([0, T], \bbR)$.

\begin{lemma}\label{lem:abstractF}
Assume the hypothesis of Theorem \ref{thm:Xconvergence}. It holds that the natural approximation $(X^\varepsilon : \varepsilon \in (0, 1))$ is tight. Assume further that for each $T \geqslant 0$, there exists a family $\{ F^m \}_{m \in \bbN}$ of continuous functions $F^m : C([0, T], \mathcal{S}'(\bbR)) \rightarrow C_T$, such that
\begin{equation}\label{approxfamily}
\lim_{m \rightarrow \infty} \big( \limsup_{\varepsilon \rightarrow 0} \bbE[\| X^{\varepsilon} - F^m(\eta^{\varepsilon}) \|_{C_T}] + \bbE[\| X - F^m(\eta) \|_{C_T}] \big) = 0
\end{equation}
Then the conclusion of Theorem \ref{thm:Xconvergence} holds.
\end{lemma}

\begin{proof}
Tightness follows a standard route. Proposition \ref{prp:tight} tells us that the Itô trick holds uniformly in $\varepsilon \in (0, 1)$, in which case for $p \geqslant 1$, and $T\geqslant 0$, we have a constant $C = C(p, T)$ such that uniformly in $\varepsilon > 0$, $s, t \in [0, T]$, it holds
\[ \bbE [| X^{\varepsilon}_t - X^{\varepsilon}_s |^p] \lesssim_p \bbE \Big[ \Big| \int_s^t \eta^{\varepsilon}_r (0) \dd r \Big|^p \Big] +\bbE [| B^{\varepsilon}_t - B^{\varepsilon}_s |^p] \leqslant C | t - s |^{p / 2}, \]
for any $p \geqslant 2$.

Next we turn to proving the conclusion of Theorem \ref{thm:Xconvergence} conditional on the existence of $\{ F^m \}_{m \in \bbN}$. Fix $T>0$ and a continuous bounded function $f : C([0, T],\bbR) \rightarrow \bbR$. Consider the decomposition for fixed $m \in \bbN$
\begin{equation}\label{epsilonthird}
\begin{split}
| \bbE[f(X^{\varepsilon})] - \bbE[f(X)] | &\leqslant \bbE[| f(X^{\varepsilon}) - f(F^m(\eta^{\varepsilon})) |] + | \bbE[f(F^m(\eta))] - \bbE[f(F^m(\eta^{\varepsilon}))] | \\
&\qquad + \bbE[| f(F^m(\eta)) - f(X) |].
\end{split}
\end{equation}
We wish to show $\limsup_{\varepsilon \rightarrow 0}| \bbE[f(X^{\varepsilon})] - \bbE[f(X)] | = 0$. To obtain this, we shall take limsup as $\varepsilon \rightarrow 0$ in \eqref{epsilonthird}. The middle term on the RHS of \eqref{epsilonthird} vanishes in this limit due to the convergence $\eta^\varepsilon \convd \eta$. The first and last terms on the RHS do not vanish in this limit, but they vanish upon further taking $m \rightarrow \infty$. In detail, we claim
\begin{align}
\lim_{m \rightarrow \infty} \big( \limsup_{\varepsilon \rightarrow 0} \bbE[| f(X^{\varepsilon}) - f(F^m(\eta^{\varepsilon})) |] + \bbE[| f(X) - f(F^m(\eta)) |]\big) = 0 \label{eq:39}
\end{align}

This claim follows from the property \eqref{approxfamily}. We prove the claim only for the first term, as the second term is similar. Fix $\delta > 0$. Tightness gives a compact subset $K \subset C_T$ such that $\sup_{\varepsilon \in (0, 1)} \bbP (X^\varepsilon \notin K) < \delta$. In particular, we have uniformly among $\varepsilon \in (0, 1)$ that
\begin{align*}
\bbE[| f(X^{\varepsilon}) - f (F^m(\eta^{\varepsilon})) |] \leqslant \delta\| f \|_{\infty} + \bbE[ | f(X^{\varepsilon}) - f (F^m(\eta^\varepsilon)) | \ic_{\{ X^{\varepsilon} \in K \}} ]
\end{align*}

It holds that $f : C_T \rightarrow \bbR$ is uniformly continuous upon being restricted to the compact subset $K \subset C_T$. In fact, we have the slightly stronger property which allows points $y \in C_T$ that are slightly outside of $K$. For any $\delta' > 0$, there exists $\kappa = \kappa(\delta', K)$ such that $|f(y_1) - f(y_2)| < \delta'$ for any $y_1, y_2 \in C_T$ satisfying $\| y_1 - y_2 \|_{C_T} < \kappa$, provided that $y_i \in K$ for at least one of $i \in \{1, 2\}$.

Using this property with $\delta' = \delta$, we continue our estimate to find
\begin{align*}
\bbE[| f(X^{\varepsilon}) - f (F^m(\eta^{\varepsilon})) |] &\leqslant \delta (\| f \|_{\infty} + 1) + 2 \| f \|_{\infty} \bbP(\|X^{\varepsilon} - F^m(\eta^{\varepsilon})\|_{C_T} > \kappa(\delta) ) \\
&\leqslant \delta (\| f \|_{\infty} + 1) + 2 \| f \|_{\infty} \kappa(\delta)^{-1} \bbE[\|X^{\varepsilon} - F^m(\eta^{\varepsilon})\|_{C_T}]
\end{align*}
In particular, \eqref{eq:39} follows from invoking \eqref{approxfamily}, and the fact that $\delta > 0$ is arbitrary.
\end{proof}

\begin{proof}[Proof of Theorem \ref{thm:Xconvergence}]
Thanks to Lemma \ref{lem:abstractF}, it suffices to construct a family $F^m$ satisfying \eqref{approxfamily}. Both terms in \eqref{approxfamily} are handled with the same argument. To allow us to present both cases with the same notation, we write $\eta^0, X^0, V^0$, in place of $\eta, X, V$. For all $\varepsilon \geqslant 0$, it holds that $X^\varepsilon$ is given according to \eqref{Xfrometa}: in the case $\varepsilon > 0$, this is an identity, whereas in the case $\varepsilon = 0$, it is a definition.

The additive functional term in \eqref{Xfrometa} can easily be approximated by a continuous function of the environment $(\eta^\varepsilon_t)_{t \in [0, T]}$. Indeed, consider $F^m_0 : C([0, T], \mathcal{S}'(\bbR)) \rightarrow C_T$ given by $F^m_0(\eta) = \int_0^{\cdummy} \eta_s(\rho_{1 / m}) \dd s$. The corresponding error in this approximation is controlled by the Itô trick, which holds for $\varepsilon > 0$ thanks to Proposition \ref{prp:tight} and for $\varepsilon = 0$ thanks to the definition of energy solution. Applying the Itô trick yields
\begin{align*}
\bbE \big[ \big\Vert \int_0^{\cdummy} \eta^\varepsilon_s(0) \dd s - \int_0^{\cdummy} \eta^\varepsilon_s(\rho_{1 / m})  \dd s  \big\rVert_{C_T} \big] &\lesssim_T \| \ell_\delta - \ell_{\rho_{1 / m}} \|_{\esob{-1}{0}}  \\
&\lesssim \int_{\bbR} \frac{\VHat(k)\rhoHat^2(\varepsilon k) (1 - \rhoHat(k / m))^2}{1 + |k|^2} \dd k
\end{align*}
In particular, thanks to Assumption \ref{ass:posdef}, it holds
\begin{equation}\label{eq:66}
\lim_{m \rightarrow \infty} \Big( \limsup_{\varepsilon \rightarrow 0} \bbE \Big[ \Big\Vert \int_0^{\cdummy} \eta^\varepsilon_s(0) \dd s - F^m_0(\eta^\varepsilon)  \Big\rVert_{C_T} \Big] + \bbE \Big[ \Big\Vert \int_0^{\cdummy} \eta_s(0) \dd s - F^m_0(\eta)  \Big\rVert_{C_T} \Big]\Big) = 0
\end{equation}

Now we approximate the Brownian motion in \eqref{Xfrometa} as a continuous function of $(\eta^\varepsilon_t)_{t \in [0, T]}$. We do not assume that the natural approximation is constructed on the same probability space, uniformly among $\varepsilon > 0$, in which case the Brownian motion appearing in \eqref{Xfrometa} is $\varepsilon$-dependent, and we write $(B^\varepsilon_t)_{t \geqslant 0}$ to emphasise this. Fix $\phi \in \calS \setminus\{0\}$, then for any $\varepsilon > 0$, it is a consequence of Lemma \ref{lem:2} that
\begin{equation}\label{eq:65}
B^\varepsilon_t = \int_0^t \frac{\ic\{\eta_s^\varepsilon(\nabla \phi)\neq 0\}}{\eta_s^\varepsilon(\nabla\phi)} \dd M_s^\varepsilon
\end{equation}
in which $(M^\varepsilon_t)_{t\geqslant 0}$ is the Dynkin martingale
\begin{equation}\label{eq:64}
M^\varepsilon_t = \eta^\varepsilon_t(\phi) - \eta^\varepsilon_0(\phi) - \thalf \int_0^t \eta^\varepsilon(\Delta \phi) \dd s - \int_0^t \calA^\varepsilon \ell_\phi(\eta^\varepsilon_s) \dd s
\end{equation}
Upon setting $B^0 := B$ for the Brownian motion as defined in \eqref{eq:4}, and $M^0 = M(\ell_\phi)$ for the Dynkin martingale given to us by the the notion of energy solution, we have that \eqref{eq:65} and \eqref{eq:64} also hold in the case $\varepsilon = 0$.

As described after the statement of the theorem, \eqref{eq:65} has within it three different instances of discontinuous dependence upon the environment $\eta^\varepsilon$. We take each one of these in turn, performing an appropriate substitution in each case, until we end up with a continuous function $F^m_2 : C([0, T], \mathcal{S}'(\bbR)) \rightarrow C_T$ so that $F^m_2(\eta^\varepsilon)$ will serve as a good approximation of $(B^\varepsilon_t)_{t\geqslant 0}$.

We begin by replacing the reciprocal function $x \mapsto x^{-1}$. For $m \in \bbN$, let $p^m \in \mathcal{S} (\bbR)$ be such that $p^m(x) \uparrow x^{- 1}$ on $(0, 1]$, $p^m (x) \downarrow x^{- 1}$ on $[- 1, 0)$ and $p^m \rightarrow (x \mapsto x^{- 1})$ locally uniformly in $\bbR \backslash [- 1, 1]$. For $\varepsilon \geqslant 0$, define $B^{\varepsilon, m}_t := \int_0^{\cdot} p^m (\eta_s^\varepsilon(\nabla \phi)) \dd M^\varepsilon_s$. We compute the approximation error for $\varepsilon \geqslant 0$: writing $\bbE_{\pi^\varepsilon}$ for the measure in which $\eta^{\varepsilon}$ is in the stationary regime, $\eta^\varepsilon_0 \sim \pi(V^\varepsilon)$, it holds
\begin{align*}
\bbE [ \lVert B^{\varepsilon, m} - B^{\varepsilon} \rVert_{C_T}^2] &\lesssim \int_0^T \bbE \Big[ \Big| p^m(\eta_s^{\varepsilon}(\nabla \phi)) - \frac{\ic\{\eta_s^\varepsilon(\nabla \phi) \neq 0\}}{\eta_s^\varepsilon(\nabla\phi)} \Big|^2 \eta_s^{\varepsilon} (\nabla \phi)^2 \Big] \dd s \\
&\lesssim \Big\lVert \frac{\dd \pi^\varepsilon_0}{\dd \pi^\varepsilon} \Big\rVert_{\efock} \int_0^T \bbE_{\pi^\varepsilon} \Big[ \Big|p^m(\eta_s^{\varepsilon}(\nabla \phi)) - \frac{\ic\{\eta_s^\varepsilon(\nabla \phi) \neq 0\}}{\eta_s^\varepsilon(\nabla\phi)} \Big|^4 \eta_s^{\varepsilon} (\nabla \phi)^4 \Big]^{1/2} \dd s \\
&\leqslant T \Big\lVert \frac{\dd \pi^\varepsilon_0}{\dd \pi^\varepsilon} \Big\rVert_{\vfock} \bbE[ |Z p^m(Z) - 1|^4 : Z \sim N(0, \sigma_\varepsilon^2) ]^{1/2}
\end{align*}
in which we define $\sigma^2_\varepsilon := \lVert \ell_{\nabla\phi} \rVert_{\efock}^2$. We observe that
$$ \limsup_{\varepsilon \rightarrow 0} \bbE[ |Z p^m(Z) - 1|^4 : Z \sim N(0, \sigma_\varepsilon^2) ] = \bbE[ |Z p^m(Z) - 1|^4 : Z \sim N(0, \sigma_0^2) ] \overset{m \rightarrow \infty}{\longrightarrow} 0 $$
Combining the above two displays, along with the uniform bound \eqref{eq:21}, we conclude that
\begin{equation}\label{eq:67}
\lim_{m \rightarrow \infty} \big( \limsup_{\varepsilon \rightarrow 0} \bbE[\lVert B^{\varepsilon, m} - B^{\varepsilon} \rVert_{C_T} ] + \bbE[\lVert B^{0, m} - B^{0} \rVert_{C_T} ]\big) = 0
\end{equation}

Next, we replace the stochastic integral appearing in the definition of $B^{\varepsilon, m}$ by a Riemann sum. For $n \in \bbN$, $s \in [0, T]$, define $[s]^n = i T / n$ for the value $i \in 1{:}n$ that satisfies $i T / n \leqslant s < (i +1) T / n$. Define the process
\begin{equation*}
\begin{split}
B^{\varepsilon, m, n}_t &:= \int_0^t p^m (\eta_{[s]^n}^{\varepsilon}(\nabla \phi)) \dd M^\varepsilon_s \\
&= \sum_{i = 1}^n p^m (\eta_{t \wedge i T / n}^{\varepsilon} (\nabla \phi)) (M^{\varepsilon}_{t \wedge (i + 1) T / n} - M^{\varepsilon}_{t \wedge i T / n})
\end{split}
\end{equation*}
We compute the following approximation error for $\varepsilon \geqslant 0$: using Cauchy-Schwarz twice, and the control $(p^m(x) - p^m(y))^4 \lesssim \| p^m \|_{C^1}^4 (x-y)^4$, we obtain
\begin{align}
\bbE [\lVert B^{\varepsilon, m} - B^{\varepsilon, m, n}\rVert_{C_T}^2] &\leqslant \int_0^T \bbE[(p^m(\eta_s^{\varepsilon}(\nabla \phi)) - p^m(\eta_{[s]^n}^{\varepsilon}(\nabla \phi)))^2 \eta_s^{\varepsilon}(\nabla \phi)^2] \dd s \nonumber \\
&\leqslant \int_0^T \bbE[\eta_s^{\varepsilon}(\nabla \phi)^4]^{1/2} \bbE[(p^m(\eta_s^{\varepsilon}(\nabla \phi)) - p^m(\eta_{[s]^n}^{\varepsilon}(\nabla \phi)))^4]^{1/2} \dd s \nonumber \\
&\lesssim T \| p^m \|_{C^1}^2 \Big\lVert\frac{\dd \pi^\varepsilon_0}{\dd \pi^\varepsilon} \Big\rVert_{\efock}^{1/2} \sigma_\varepsilon^2 \sup_{s \in [0, T]} \bbE[|\eta_s^{\varepsilon}(\nabla \phi) - \eta_{[s]^n}^{\varepsilon}(\nabla \phi)|^4]^{1/2} \label{eq:68}
\end{align}
For the last term, we use Proposition \ref{prp:tight} in the case $p = 4$: there exists a constant $C = C(T, V, \beta, \pi_0, \rho)$ such that uniformly among $\varepsilon \in [0, 1)$ it holds
$$ \bbE[|\eta_s^{\varepsilon}(\nabla \phi) - \eta_{[s]^n}^{\varepsilon}(\nabla \phi)|^4] \leqslant C |s - [s]^n|^2 $$
Moreover, we have the following uniform control for $\varepsilon \in [0, 1)$:
$$ \lVert \dd \pi^\varepsilon_0/\dd \pi^\varepsilon\rVert_{\efock}, \sigma_\varepsilon^2, \lVert\ell_{\nabla\phi}\rVert_{\efock} \lesssim_{V, \phi, \pi_0} 1 $$
Plugging these last two displays into \eqref{eq:68}, and using $|s - [s]^n|\lesssim_T n^{-1}$, we conclude that there is a constant $C = C(T, V, \beta, \pi_0, \rho)$ such that
\begin{align}
\limsup_{\varepsilon \rightarrow 0} \big( \bbE[\lVert B^{\varepsilon, m} - B^{\varepsilon, m, n} \rVert_{C_T} ] + \bbE[\lVert B^{0, m} - B^{0, m, n} \rVert_{C_T}] \big) \leqslant \, C n^{-1/2} \| p^m \|_{C^1} \overset{m \rightarrow \infty}{\longrightarrow} 0 \label{eq:75}
\end{align}
provided that we choose $n = n(m)$ going to infinity fast enough with $m \in \bbN$.

Finally, we replace the martingale $M^\varepsilon$ appearing in the Riemann sum, with an approximation that is a continuous function of $(\eta^\varepsilon)_{t \in [0, T]}$. The problem being that in the case $\varepsilon = 0$, the final term of \eqref{eq:64} is given by Definition \ref{def:i}, and it is not a continuous function of $(\eta^0_t)_{t \in [0, T]}$. To that end, define the continuous function $M^N : C([0, T], \mathcal{S}'(\bbR)) \rightarrow C_T$ for $N \in \bbN$ according to
$$ M^N(\eta) := \eta_\cdummy(\phi) - \eta_0(\phi) - \thalf \int_0^\cdummy \eta(\Delta \phi) \dd s - \int_0^\cdummy g^N(\eta_s) \dd s $$
in which $g^N(\eta) := \eta(\nabla\phi) \eta(\rho_{1/N}) - V(\nabla \phi)$. Define the continuous function $F^{m, n, N} : C([0, T], \calS'(\bbR)) \rightarrow C_T$ according to
$$ F^{m, n, N}(\eta) = \sum_{i = 1}^n p^m (\eta_{\cdummy \wedge i T / n}^{\varepsilon} (\nabla \phi)) (M^{N}_{\cdummy \wedge (i + 1) T / n}(\eta) - M^{N}_{\cdummy \wedge i T / n}(\eta)) $$
In other words, $F^{m, n, N}(\eta^\varepsilon)$ coincides with the Riemann sum $B^{\varepsilon, m, n}$, only with $M^N(\eta^\varepsilon)$ in place of the Dynkin martingale $M^\varepsilon$. We compute the approximation error for $\varepsilon \geqslant 0$:
\begin{align*}
\bbE [ \lVert B^{\varepsilon, m, n} - F^{m, n, N}(\eta^\varepsilon)\rVert_{C_T} ] & \lesssim n \| p^m \|_{L^\infty(\bbR)} \bbE[ \lVert M^\varepsilon - M^N(\eta^\varepsilon) \rVert_{C_T} ]\\
& = n \| p^m \|_{L^\infty(\bbR)} \bbE[ \lVert \int_0^\cdummy (\mathcal{A}^{\varepsilon} \ell_{\phi}(\eta^\varepsilon_s) - g^N(\eta^\varepsilon_s)) \dd s \rVert_{C_T} ] \\
& \lesssim_T n \| p^m \|_{L^\infty(\bbR)} \| \mathcal{A}^{\varepsilon} \ell_{\phi} - g^N \|_{\esob{-1}{0}}
\end{align*}
where in the last line we used the Itô trick for $p = 1$, which holds uniformly in $\varepsilon \geqslant 0$ thanks to Proposition \ref{prp:tight}.

We decompose $g^N = g^N_0 + g^N_2$ into its chaos components:
$$ g^N_2(\eta^\varepsilon) = \eta^\varepsilon(\nabla\phi) \eta^\varepsilon(\rho_{1/N}) - \langle\ell_{\nabla \phi}, \ell_{\rho_{1/N}}\rangle_{\efock},\quad g^N_0 = \langle\ell_{\nabla \phi}, \ell_{\rho_{1/N}}\rangle_{\efock} - V(\nabla \phi) $$
Note that while $g^N$ does not depend on $\varepsilon \geqslant 0$, the decomposition is $\varepsilon$-dependent. By orthogonality:
\begin{equation}\label{eq:72}
\| \mathcal{A}^{\varepsilon} \ell_{\phi} - g^N \|_{\esob{-1}{0}}^2 = \| \mathcal{A}^{\varepsilon} \ell_{\phi} - g^N_2 \|_{\esob{-1}{0}}^2 + |g^N_0|^2
\end{equation}

For the first term in \eqref{eq:72}, we have
\begin{align*}
\| \mathcal{A}^{\varepsilon} \ell_{\phi} - g^N_2 \|_{\esob{-1}{0}}^2 &\lesssim \int \int \VHat^\varepsilon(p) \VHat^\varepsilon(q) \frac{ |p \phiHat(p)|^2 |\rhoHat(N^{-1} q) - 1|^2 }{1 + \thalf |p+q|^2} \dd p \dd q \\
&\lesssim_\rho \int \int \VHat(p) \VHat(q) \frac{ |p \phiHat(p)|^2 |\rhoHat(N^{-1} q) - 1|^2 }{1 + \thalf |p+q|^2} \dd p \dd q
\end{align*}
where we use $\VHat^\varepsilon(p) = \rhoHat(\varepsilon p)^2 \VHat(p)$. We split this integral into regions $R_1 = |p+q| < N^{1/2}$ and $R_2 = |p+q| \geqslant N^{1/2}$. The contribution from $R_1$ is bounded above by
\begin{align*}
\lVert\rhoHat\rVert_{C^1}^2 \int_{\bbR}  \VHat(p) |p \phiHat(p)|^2 &\big( \int_{|q| \leqslant |p| + N^{1/2}} \VHat(q) N^{-2} |q|^2 \dd q \big)  \dd p \\
&\lesssim_{\rho, V}  N^{-2} \int_{\bbR} \VHat(p) |p \phiHat(p)|^2  (|p| + N^{1/2})^3  \dd p \lesssim_{V, \phi} N^{-1/2}
\end{align*} 
The contribution from $R_2$ is bounded above by
\begin{align*}
\lVert\rhoHat\rVert_{L^\infty(\bbR)}^2\int \VHat(p) |p \phiHat(p)|^2 \big( \int_{|p+q| \geqslant N^{1/2}}  \frac{\VHat(q)}{1 + \thalf |p+q|^2} \dd q \big) \dd p  \lesssim_{\rho, V, \phi} N^{-1/2}
\end{align*}
For the second term in \eqref{eq:72}, we have
$$ |g^N_0|^2 \lesssim | \langle \ell_{\nabla\phi}, \ell_{\rho_{1/N}} - \ell_{\delta} \rangle_{\efock}|^2 + |(V^\varepsilon-V)(\nabla\phi)|^2 $$
The second term goes to zero as $\varepsilon \rightarrow 0$, and equals zero in the case $\varepsilon = 0$. Moreover, we also have
\begin{align*}
|\langle \ell_{\nabla\phi}, \ell_{\rho_{1/N}} - \ell_{\delta} \rangle_{\efock}| &\lesssim \lVert\rhoHat\rVert_{L^\infty(\bbR)}^2 \int \VHat(p) |p\phiHat(p)| |\rhoHat(N^{-1} p) - 1| \dd p  \\
&\lesssim_{\rho, V, \phi} N^{-1}
\end{align*}
Plugging this in, we conclude that there is a constant $C = C(T, \rho, V, \phi)$ such that
\begin{align}
\limsup_{\varepsilon \rightarrow 0} \bbE [ \lVert B^{\varepsilon, m, n} - F^{m, n, N}(\eta^\varepsilon)\rVert_{C_T} ] + \bbE [ \lVert B^{0, m, n} - F^{m, n, N}(\eta^0)\rVert_{C_T} ] &\leqslant C N^{-1/2} n \| p^m \|_{L^\infty(\bbR)} \nonumber \\
&\overset{m \rightarrow \infty}{\longrightarrow} 0 \label{eq:76}
\end{align}
provided that we also choose $N = N(m)$ going to infinity fast enough with $m \in \bbN$.

To approximate both the additive functional and the Brownian motion, we consider $F^m$ defined according to $F^m = F^m_0 + F^{m, n(m), N(m)}$. It remains to verify \eqref{approxfamily}. Consider the decomposition
\begin{align*}
X^\varepsilon - F^m(\eta^\varepsilon) = & \big( \int_0^t \eta^{\varepsilon}_s (0) \dd s - F_0^\varepsilon(\eta^\varepsilon) \big) + (B^{\varepsilon} - B^{\varepsilon, m}) \\
&\, + (B^{\varepsilon, m} - B^{\varepsilon, m, n}) + (B^{\varepsilon, m, n} - F^{m, n, N}(\eta^\varepsilon))
\end{align*}
The result follows from the triangle inequality, and invoking \eqref{eq:66}, \eqref{eq:67}, \eqref{eq:75}, and \eqref{eq:76}.
\end{proof}

\subsection{Characterisation of the SRBP}\label{sec:char-srbp}
\begin{definition}\label{defenergysolutionsX}
Let $V \in \calS'(\bbR)$ satisfy Assumption \ref{ass:posdef} and let $\pi_0 \ll \pi(V)$ be such that $\dd \pi_0 / \dd \pi \in L^2(V)$. Let $(\omega,X)$ be an $\mathcal{S}'(\bbR)\times C(\bbR_+,\bbR)$-valued random variable such that $\omega\sim\pi_0$ and $X_0=0$. The pair $(\omega,X)$ is called an \textit{energy solution} to the SDE \eqref{eq:3} if the continuous $\mathcal{S}'(\bbR)$-valued process
\[\eta=\mathcal{E}^V(\omega,X),\]
defined, in the weak sense, by
\begin{equation}\label{eq:etaFromX}
\eta_t=\omega(X_t+\cdot)+\beta\int_0^t\nabla V(X_t-X_s+\cdot)\dd s,
\end{equation}
is an energy solution with respect to $V$ and $\mathcal{L}$ in the sense of Definition \ref{def:energysolution}.
\end{definition}
The following result characterises the process $X=(\eta_0,\mathcal{X}^V(\eta))$ from Subsection \ref{subsec:constructionX} as the unique energy solution to \eqref{eq:3}.
\begin{theorem}\label{wellpsoednessenergysolutionsX}
Let $V \in \calS'(\bbR)$ satisfy Assumption \ref{ass:posdef} and suppose that there exist $\alpha \in (-1, \infty),\alpha^\ast\in (-\infty,\infty)$ such that $\alpha^\ast<2+\alpha$ and
\begin{equation}\label{largescalebehaviourwellposed}
0 < \liminf_{p \rightarrow 0} |p|^{-\alpha^\ast} \VHat(p) \leqslant \limsup_{p \rightarrow 0} |p|^{-\alpha} \VHat(p) < \infty .
\end{equation}
Let $\pi_0 \ll \pi(V)$ be such that $\dd \pi_0 / \dd \pi \in L^2(V)$. Then, there exists a unique (in law) energy solution $(\omega,X)$ to \eqref{eq:3} such that $\omega\sim \pi_0$. In fact, it holds
\[ (\omega,X)\sim (\eta_0,\mathcal{X}^V(\eta)), \]
where $\eta\sim\text{ENV}(V,\beta,\pi_0)$.
\end{theorem}
The proof of Theorem \ref{wellpsoednessenergysolutionsX} relies on the following ideas: Existence of an energy solution $(\omega,X)$ follows from the fact that we have
\begin{equation}
\label{eq:74}
\eta^{\varepsilon}=\eta^{\varepsilon}_0(X^\varepsilon+\cdot)+\beta\int_0^\cdot\nabla V^\varepsilon(X^\varepsilon_\cdot-X^\varepsilon_s+\cdot)\dd s,
\end{equation}
as well as, in law, $\eta^\varepsilon\rightarrow\eta,X^\varepsilon\rightarrow X=\mathcal{X}^V(\eta)$, so that we can at least formally pass to the limit in the above equation. Thus, we have that $X$ is an energy solution. As can be seen from the proof of Theorem \ref{wellpsoednessenergysolutionsX} below, it is straightforward to make this argument precise.

For uniqueness, we will have to work more but the underlying idea is quickly laid out as follows. Suppose that $(X,\omega)$ is an energy solution and define $\eta=\mathcal{E}^V(\omega,X)$ according to \eqref{eq:etaFromX}. Testing \eqref{eq:etaFromX} against the function $f(x)=x$ and integrating the $\nabla V$-term by parts, we formally have
\[ \eta_t(f)=-X_t\hat{\omega}(0)+\omega(f)-\beta\hat{V}(0)t, \]
and thus
\begin{equation}\label{eq:22}
X_t=\frac{1}{\hat{\eta_0}(0)} \Big(\eta_0(f)-\beta\hat{V}(0)t-\eta_t(f) \Big)
\end{equation}
This indicates that $(\omega,X)$ is a function of $\eta\sim \text{ENV}(V,\beta,\pi_0)$ and hence has a unique law. However, the RHS of the latter identity is not well-defined since testing $\eta$ against $f$ and $\nabla f$ is problematic. Therefore, we have to make an approximation argument, for which we need some preparatory results. This is also where we will make use of assumption \eqref{largescalebehaviourwellposed} since it allows us to conclude that certain error terms, which are not directly functions of $\eta$, vanish in the limit (cf. equation \eqref{approxidentityX} and Lemmas \ref{vanishingRm} and \ref{vanishingQm} below).

Let us make this idea more precise. Let $f\in \mathcal{S}$ be non-zero such that $\supp \hat{f}\subset [-1,1]$. We define \(f^m=mf(m^{-1}\cdot)\), so that $\supp( \fHat^m ) \subset [-m^{-1},m^{-1}]$. Testing \eqref{eq:etaFromX} against $f^m$, and then performing a Taylor expansion on the function $f^m(x - \cdot)$, we obtain

\begin{align*}
\eta_t(f^m) &= \eta_0 (f^m(\cdummy - X_t)) - \beta\int_0^t V(\nabla f^m(-X_t + X_s+\cdot)) \dd s \\
&= \eta_0 (f^m) - X_t \eta_0(\nabla f^m (x)) + \eta_0(R^m_t) - \beta\int_0^t V(\nabla f^m(-X_t + X_s+\cdot)) \dd s
\end{align*}
in which we have define $R^m_t$ to be the remainder in the integral form of Taylor's expansion
\begin{equation}\label{defRm}
R^m_t (x) := \int_x^{x - X_t} \Delta f^m (y) (x - X_t - y) \dd y.
\end{equation}
Rearranging, we obtain
\begin{align}
X_t &=  \frac{1}{\eta_0(\nabla f^m)} \bigg(\eta_0(f^m) - \eta_t(f^m) + \eta_0(R_t^m) - \beta\int_0^t V(\nabla f^m(-X_t + X_s + \cdot)) \dd s \bigg) \nonumber\\
&=  \frac{1}{\eta_0(\nabla f^m)} \bigg(\eta_0(f^m) - \eta_t(f^m) + \eta_0(R_t^m) - t \beta V (\nabla f^m) - \int_0^t V (Q_{s,t}^m) \dd s \bigg) \label{approxidentityX}
\end{align}
in which we have defined
\begin{equation}\label{defQm}
Q^m_{s,t} = \nabla f^m(-X_t + X_s + \cdummy) -\nabla f^m.
\end{equation}
Note that dividing by $\eta_0(\nabla f^m)$ in \eqref{approxidentityX} is justified because $\pi_0 \ll \pi$, and $\eta_0(\nabla f^m)$ is a Gaussian under $\pi$.

The only terms in \eqref{approxidentityX} that are not given purely in terms of $\eta$ are the terms $\eta_0(R_t^m)$ and $\int_0^t V(Q_{s,t}^m)\dd s$. In the case $f(x) = x$, for which $f^m(x) = x$, these terms are zero, and we recover \eqref{eq:22}. Instead, for our choice of $f^m$, , $\supp \hat{f}\subset [-1,1]$, we will show that these are error terms which vanish in the limit $m \rightarrow \infty$ upon being divided by $\eta_0(\nabla f^m)$. The crucial property being that $f^m$ has $\nabla f^m \gg \Delta f^m$, meaning that the numerator dominates the denominator in each case.

\subsubsection{Spatial regularity of $\eta_0$ and control of $R^m$}
We will need some information about the spatial regularity of realisations of $\omega \sim \pi_0$. The reason for this is that we will have to control terms of the form
\[ \omega (\phi (\cdummy + X_t)). \]
Assume for a moment that $\omega\sim\pi(V)$. If $X$ was replaced by some process $Y$ which is
independent of $\omega$, we could simply bound
\[ \bbE [\omega (\phi (\cdummy + Y_t))^2] = \bbE_{x_t \sim Y_t} [\bbE_{\omega \sim \pi} [\omega(\phi (\cdot + x_t))^2] ] \leqslant \| \phi \|^2_{L^2(\bbR)} . \]
Unfortunately, this independence is {\textit{not}} true for $X$ and thus we have to argue in a pathwise sense.

We work with the following weighted homogeneous Besov norms (see e.g.\ \cite{BahouriCheminDanchin11}). Let $\varphi \in C^{\infty}(\bbR)$ be supported in an annulus such that $\varphi_j = \varphi(2^{- j}\cdummy), j \in \bbZ$ defines a smooth partition of unity, i.e.
\[ \sum_{j \in \bbZ} \varphi_j (p) = 1, \quad p \neq 0 . \]
Unlike in the case of inhomogeneous Besov spaces, here we consider negative $j$. This allows us to measure regularity at large scales and the difficulty of our argument lies exactly in this regime since we essentially want to make sense of $\omega(g)$ for $g(x)=x$, as explained above. As usual, we write $\phi_j = \mathcal{F}^{- 1} (\varphi_j)=2^j\phi_0(2^j\cdot)$ and $\Delta_j \eta = \phi_j \ast \eta$.

For $p \in [1, \infty), s,\gamma \in \bbR$, we fix the weight $w_{\gamma}(x)=(1+|x|^2)^{\gamma/2}$, and define
\[\| \eta \|_{\dot{B}^{s}_{p, p, \gamma}}^p = \sum_{j \in\bbZ} 2^{sjp} \| w_{\gamma} \Delta_j \eta\|_{L^p}^p, \]
One checks that for conjugate indices $r, q \geqslant 1$ such that $r^{-1} + q^{-1} = 1$, we have the duality relation
\[ \langle \eta, \tilde{\eta} \rangle \leqslant \| \eta\|_{\dot{B}^{s}_{p, p, \gamma}} \| \etaTil\|_{\dot{B}^{-s}_{q, q, - \gamma}},\phantom{ooo} \eta,\tilde{\eta}\in\mathcal{S}(\bbR), \]

The proof of the following lemma is a simple computation, and it can be found in the appendix, Section \ref{sec:proof-regularity}.
\begin{lemma}\label{regularityeta} Let $V,\pi_0$ satisfy the assumptions of Theorem \ref{wellpsoednessenergysolutionsX} and let $\omega \sim \pi_0$, $r \in [1, \infty)$ and $\gamma > 1 / r$. Then it almost surely holds that $\lVert\Pi_{\leqslant 0} \omega\rVert_{\dot{B}^{-s}_{r, r, -\gamma}} < \infty$ for any $s<\frac{1+\alpha}{2}$, in which $\Pi_{\leqslant 0} \omega = \sum_{j\leqslant 0}\Delta_j\omega$.
\end{lemma}

With the pathwise regularity of $\pi_0$ established as above, we may control the $R^m$ term as in the following lemma.
\begin{lemma}\label{vanishingRm}
Let $V,\pi$ and $\alpha >-1$ be as in Theorem \ref{wellpsoednessenergysolutionsX}, let $(\omega,X)$ be a corresponding energy solution to \eqref{eq:3}, let $\eta=\mathcal{E}^V(\omega,X)$ according to \eqref{eq:etaFromX} and define $R_t^m$ according to \eqref{defRm}. For any $t\geqslant0$ and $\epsilon>0$ almost surely there is a constant $C > 0$ for which uniformly among $m \in \bbN$, it holds
\begin{equation}\label{eq:73}
|\eta_0 (R_t^m)| \lesssim m^{-(1+\alpha)/2+\epsilon}
\end{equation}
\end{lemma}

\begin{proof}
With Lemma \ref{regularityeta} we may bound for $r \in (1, \infty), \gamma > 1/ r, s<\frac{1+\alpha}{2}$ and with $\frac{1}{r} + \frac{1}{q} = 1$,
\[ |\eta_0(R^m_t)| \leqslant \| \Pi_{\leqslant 0}\eta_0 \|_{\dot{B}^{-s}_{r, r, -\gamma}} \| R^m_t \|_{\dot{B}^s_{q, q, \gamma}}. \]
It remains to show that we can choose such $(r, \gamma, s)$ for which $\| R^m_t \|_{\dot{B}^s_{q, q, \gamma}}$ satisfies the bound given in \eqref{eq:73}, uniformly in $m$.

Since $R^m_t=f^m(\cdot-X_t)-f^m(x)-X_t\nabla f^m$, we have for any $m$ that
\begin{equation}\label{supportmrate}
\text{supp} \, \RHat^m_t\subset [-m^{-1},m^{-1}] .
\end{equation}
With a change of variables, we write
$$ R^m_t(x)=m g_{1/m}(x/m), \quad \text{where} \quad g_{1/m}(x) := \int_x^{x-X_t/m}\Delta f(z)(x - X_t/m -z)\dd z .$$
Using \eqref{supportmrate} and the identity $\Delta_j g_\varepsilon (m^{-1}\cdot)(x)= \Delta_{j+\text{log}_2 m}g_\varepsilon(x/m)$, we have
\begin{align}
\| R^m_t \|_{\dot{B}^s_{q, q, \gamma}}^q &= \sum_{j\leqslant -\text{log}_2 m}2^{s j q}\|w_\gamma\Delta_j R^m_t\|_{L^q(\bbR)}^q \notag \\
&=\sum_{j\leqslant -\text{log}_2 m} 2^{s j q}\int w_\gamma^q(x)|\Delta_j m g_{1/m}(\cdot/m)|^q(x)\dd x\notag\\
&\lesssim m^{q(1+\gamma -s)+1} \sum_{j\leqslant 0} 2^{s j q}\int w_\gamma^q(x)|\Delta_j g_{1/m}(x)|^q\dd x .\label{largescaleRm}
\end{align}

Let us estimate the integral. We may assume that $m$ is large enough so that $|X_t/m|\leqslant 1$, in which case, upon setting $\psi(x)=\text{supp}_{z\in B_1(x)}|\Delta f(z)|$, it holds that $|g_{1/m}(x)| \lesssim_{X_t} m^{-2} \psi(x)$. Therefore, for $m$ large enough, it holds
\begin{align}
\int w_\gamma^q(x) \bigg|\int \Delta_j  g_{1/m}(x)\bigg|^q \dd x &= \int w_\gamma^q(x)\bigg|\int \phi_j(y-x)  g_{1/m}(y)\dd y\bigg|^q\dd x \nonumber\\
&\lesssim_{X_t} m^{-2q}\int w_\gamma^q(x) \bigg|\int |\phi_j(y-x)||\psi(y)|\dd y\bigg|^q\dd x \label{decompintegralest}
\end{align}
We consider small $\delta \in (0, 1)$ and split the integral based on $|y| \leqslant w_\delta(|x|)$ and $|y| > w_\delta(|x|)$.

For the case $|y| \leqslant w_\delta(|x|)$, using $\phi_j = 2^j\phi_0(2^j\cdot)$ we obtain
\begin{align*}
\int &w_\gamma^q(x)\bigg|\int_{|y| \leqslant w_\delta(|x|)} |\phi_j(y-x)||\psi(y)| \dd y\bigg|^q\dd x \\
&\lesssim \|\psi\|^q_\infty 2^{jq}\int w_{\gamma+\delta}^q(x) \sup \{ |\phi_0(z)|^q : |z - 2^{j}x| \leqslant 2^{j} w_\delta(|x|) \} \dd x
\end{align*}
Since $\phi_0$ is a Schwarz function, for any $K \geqslant 0$, there is a constant so that uniformly in $\delta \in (0, \half)$, $x \in \bbR$, $j \leqslant 0$, it holds
\[ \sup \{ |\phi_0(z)| : |z - 2^{j}x| \leqslant 2^{j}(1+|x|^2)^{\delta/2} \} \lesssim_{K} (1 + |2^jx|^2)^{-K/2} \leqslant 2^{-jK} w_K(x) \]
Continuing the estimate, we have
\begin{align*}
\int w_\gamma^q(x)\bigg|\int_{|y| \leqslant w_\delta(|x|)} |\phi_j(y-x)||\psi(y)| \dd y\bigg|^q\dd x \lesssim_{K} 2^{j(q-K)}\int w_{\gamma+\delta-K}^q(x) \dd x
\end{align*}
The integral is finite provided $(\gamma+\delta-K)q < -1$ and we choose $K = \gamma + \delta + (1+\delta)q^{-1}$. Therefore, uniformly in $j \leqslant 0$, we have
\begin{eqnarray}
\int w_\gamma^q(x)\bigg|\int_{|y| \leqslant w_\delta(|x|)} |\phi_j(y-x)||\psi(y)| \dd y\bigg|^q\dd x & \lesssim_{\delta, q, \gamma} & 2^{jq(1 - \gamma - \delta) - 1 - \delta}. \label{decompintegralest1}
\end{eqnarray}

For the case $|y| > w_\delta(|x|)$, we use that $\psi$ has super-polynomial decay. For any $K \geqslant 0$,
\begin{eqnarray}
\int w_\gamma^q(x)\bigg|\int |\phi_j(y-x)||\psi(y)||1_{|y|> L(|x|)}\dd y\bigg|^q\dd x & \lesssim_{K} & \int w_{\gamma- K\delta}^q(x)
\end{eqnarray}
which is finite provided that $(\gamma-K\delta)q < -1$, and we choose $K = K(\delta, q, \gamma)$ as such.

Plugging the two cases into \eqref{decompintegralest}, and then subsequently into \eqref{largescaleRm}, we obtain
\begin{eqnarray}
\| R^m_t \|_{\dot{B}^s_{q, q, \gamma}}^q &\lesssim_{\delta, q, \gamma}& m^{q(-1+\gamma -s)+1} \sum_{j\leqslant 0}2^{sjq} (2^{jq(1 - \gamma - \delta) - 1 - \delta} +1) .\notag
\end{eqnarray}
Given $\epsilon > 0$ we choose $s \in (0, \frac{1+\alpha}{2})$ close enough to $\frac{1+\alpha}{2}$, $r\geqslant 0$ large enough (corresponding to $q \downarrow 1$), and $\gamma > 1/p$ small enough so that
$$q(-1+\gamma -s)+1 < -(1+\alpha)/2+\epsilon$$
Finally we choose $\delta \in (0, \frac{1}{2}]$ so that the sum is finite. In fact, $\delta = \half$ will do.

\end{proof}

\subsubsection{Control of $Q^m$}
To control the $Q^m$ term, we use the following.

\begin{lemma}\label{vanishingQm}
Let $V,\pi$ and $\alpha >-1$ be as in Theorem \ref{wellpsoednessenergysolutionsX}, let $(\omega,X)$ be a corresponding energy solution to \eqref{eq:3} and define $Q_{s,t}^m$ according to \eqref{defQm}. Then, for any $0 \leqslant s \leqslant t$, almost surely there is a constant $C > 0$ for which uniformly among $m \in \bbN$, it holds
\[ \int_0^t V(Q_{s,t}^m)\dd s\lesssim m^{-\alpha-1}. \]
\end{lemma}
\begin{proof}
Let us fix $s,t\geqslant 0$. We write $Y=-X_t+X_s$. Using that $\hat{f}^m(p)=m^2\fHat(mp)$ and $\text{supp}\hat{f}^m\subset [-m^{-1},m^{-1}]$, we have
\begin{eqnarray}
|V(\nabla f^m(\cdot+Y)-\nabla f^m)|&= &m^2\bigg|\int_{-m^{-1}}^{m^{-1}}\hat{V}(p)p \hat{f}(mp)(1-e^{2\pi ipY})\dd p\bigg|\notag\\
&=& \bigg|\int_{-1}^1\hat{V}(p/m)p \hat{f}(p)(1-e^{2\pi ipY/m})\dd p\bigg| \lesssim |Y|m^{-\alpha-1}\notag
\end{eqnarray}
\end{proof}

\subsubsection{Proof of Theorem \ref{wellpsoednessenergysolutionsX}}
Now, we are in the position to prove Theorem \ref{wellpsoednessenergysolutionsX}. We first show existence of an energy solution and then uniqueness.

\begin{proof}[Proof of Theorem \ref{wellpsoednessenergysolutionsX}]
\textit{Existence of $X$.} Let $\eta\sim\text{ENV}(V,\beta,\pi_0)$ and define $X=\mathcal{X}^V(\eta)$. We claim that $(\eta_0, X)$ is an energy solution of the SDE. That is to say, upon defining $\etaTil:=\mathcal{E}^V(\eta_0,X)$, we claim that $\law(\etaTil) = \text{ENV}(V,\beta,\pi_0)$. Since in Theorem \ref{thm:env}, we characterised $\text{ENV}(V,\beta,\pi_0)$ as the limit of the natural approximation, $(\eta^\varepsilon : \varepsilon \in (0, 1))$, it suffices to show that $\eta^\varepsilon \convd \etaTil$.

Recall that $\eta^\varepsilon = \mathcal{E}^{V^\varepsilon}(\eta^\varepsilon_0,X^\varepsilon)$, and by definition $\etaTil:=\mathcal{E}^V(\eta_0,X)$. At this stage, we are almost done, because Theorem \ref{thm:Xconvergence} tells us that $(\eta^\varepsilon_0,X^\varepsilon) \convd (\eta_0,X)$. The only remaining issue is the $\varepsilon$-dependence in the function $\calE^{V^\varepsilon}$. To that end, for fixed $t \geqslant 0, \phi \in \calS(\bbR)$, we write
\begin{equation}
\label{eq:82}
\eta^\varepsilon_t(\phi) = F_{\phi, t}(\eta^\varepsilon_0, X^\varepsilon) + r^\varepsilon_{\phi, t}(X^\varepsilon)
\end{equation}
in which $F_{\phi, t} : \calS'(\bbR) \times C([0, T], \bbR) \rightarrow \bbR$ is the continuous function
$$F_{\phi, t}(\omega, X) := \omega(\tau_{-X_t} \phi) + \beta\int_0^{t} V(\tau_{-X_t + X_s} \phi_i) \dd s $$
and $r^\varepsilon_{\phi, t}$ is the remainder
$$ r^\varepsilon_{\phi, t}(X) := \beta\int_0^{t} (V^\varepsilon - V)(\tau_{-X_t + X_s} \phi) \dd s $$
In the limit, \eqref{eq:82} holds with no remainder, in that $\etaTil_t(\phi) = F_{\phi, t}(\eta_0, X)$.

To obtain the result $\eta^\varepsilon \convd \etaTil$, it is sufficient to show that for times $0 \leqslant t_1 \leqslant \ldots \leqslant t_n$, a collection of test functions $\phi_i\in\mathcal{S}(\bbR)$, and for $f \in \calS(\bbR^n)$, it holds that $\bbE[f((\eta^\varepsilon_{t_i}(\phi_i))_{i=1}^n)]$ converges as $\varepsilon \rightarrow 0$ to $\bbE[f((\etaTil_{t_i}(\phi_i))_{i=1}^n)]$. Using the functions defined above, that is to say
\begin{equation}
\label{eq:78}
\lim_{\varepsilon\rightarrow0} \bbE[f((F_{\phi_i, t_i}(\eta^\varepsilon_0, X^\varepsilon) + r^\varepsilon_{\phi_i, t_i}(X^\varepsilon))_{i=1}^n)] = \bbE[f((F_{\phi_i, t_i}(\eta_0, X))_{i=1}^n)]
\end{equation}
It follows from Theorem \ref{thm:Xconvergence} that
$$ \lim_{\varepsilon\rightarrow0} \bbE[f((F_{\phi_i, t_i}(\eta^\varepsilon_0, X^\varepsilon))_{i=1}^n)] = \bbE[f((F_{\phi_i, t_i}(\eta_0, X))_{i=1}^n)] $$
Therefore, \eqref{eq:78} follows once we can show that
\begin{equation}
\label{eq:83}
\lim_{\varepsilon\rightarrow0} \bbE[f((F_{\phi_i, t_i}(\eta^\varepsilon_0, X^\varepsilon))_{i=1}^n) - f((F_{\phi_i, t_i}(\eta^\varepsilon_0, X^\varepsilon) + r^\varepsilon_{\phi_i, t_i}(X^\varepsilon))_{i=1}^n)] = 0
\end{equation}

To prove \eqref{eq:83}, we use a Taylor expansion in $f$. This gives
$$ |f((F_{\phi_i, t_i}(\eta^\varepsilon_0, X^\varepsilon))_{i=1}^n) - f((F_{\phi_i, t_i}(\eta^\varepsilon_0, X^\varepsilon) + r^\varepsilon_{\phi_i, t_i}(X^\varepsilon))_{i=1}^n)| \lesssim \lVert\nabla f\rVert \max_{i \in 1{:}n}\{|r^\varepsilon_{\phi_i, t_i}(X^\varepsilon)|\} $$
Performing a Taylor expansion on $\rhoHat^\varepsilon(p) = \rhoHat(\varepsilon p)$, we obtain
\begin{align*}
|r^\varepsilon_{\phi_i, t_i}(X^\varepsilon)| \leqslant \beta t \int \VHat(p) |\phiHat(p)| |\rhoHat^\varepsilon(p)^2 - 1 | \dd p \lesssim_\rho \beta t \varepsilon \int \VHat(p) |p| |\phiHat(p)| \dd p
\end{align*}
Since the RHS converges to zero as $\varepsilon \rightarrow 0$, we obtain \eqref{eq:83}.

\textit{Uniqueness of $X$.} Let $(\omega,X)$ be an energy solution, put $\eta=\mathcal{E}^V(\eta_0,X)$. We rewrite the identity of \eqref{approxidentityX} according to
\begin{equation*}
X_t = F_t^m(\eta) + \frac{\eta_0(R_t^m) - \int_0^t V (Q_{s,t}^m) \dd s}{\eta_0(\nabla f^m)}, \quad F_t^m(\eta):=\frac{\eta_0(f^m)-\eta_t(f^m)-t V (\nabla f^m)}{\eta_0(\nabla f^m)}.
\end{equation*}
We claim that there is a subsequence along which almost surely
\begin{equation}
\label{eq:79}
\frac{\eta_0(R_t^m) - \int_0^t V(Q_{s,t}^m) \dd s }{\eta_0(\nabla f^m)} \rightarrow 0
\end{equation}
The uniqueness then follows from this claim because a bounded continuous observable of $(\omega, X)$ could be computed from the statistics of $\eta$ by taking a limit along the subsequence, as follows: for bounded continuous $g_i : \bbR \rightarrow \bbR$, $\psi : \calS'(\bbR) \rightarrow \bbR$,
\[ \bbE[\psi(\omega)g_1(X_{t_1})...g_1(X_{t_1})]=\lim_m \bbE[\psi(\eta_0)g_1(F_{t_1}^m(\eta))...g_n(F_{t_n}^m(\eta))] . \]
This determines the law of $X$ uniquely.

To obtain \eqref{eq:79}, we invoke Lemmas \ref{vanishingRm} and \ref{vanishingQm}, which imply that for any $\epsilon > 0$ it holds
\[ m^{(1+\alpha)/2-\epsilon}\bigg|\eta_0(R_t^m) - \int_0^t V (Q_{s,t}^m) \dd s\bigg|\lesssim 1 . \]
It remains to find a subsequence along which $(m^{(1+\alpha)/2-\epsilon} \eta_0(\nabla f^m))$ diverges.

For $\pi_0=\pi(V)$, we have that $\eta_0(\nabla f^m)$ is a centered Gaussian with standard deviation
\begin{eqnarray*}
\|\nabla f^m\|_{L^2(V)} &\simeq & m^2 \sqrt{\int_{\bbR}\hat{V}(p)|k|^2|\hat{f}(mp)|^2\dd p}\\
&= & m^{1/2} \sqrt{\int_{\bbR}\hat{V}(p/m)|k|^2|\hat{f}(p)|^2\dd p}\\
& \gtrsim & m^{(1-\alpha^\ast)/2} \sqrt{\int_\bbR|p|^2|\hat{f}(p)|^2\dd p} .
\end{eqnarray*}
In particular, the variance of $(m^{(1+\alpha)/2-\epsilon} \eta_0(\nabla f^m))_{m\in \bbN}$ is diverging, because $\frac{1+\alpha}{2}+\frac{1-\alpha^\ast}{2}>0$. Therefore we may take a subsequence along which $m^{(1+\alpha)/2-\delta}\eta_0(\nabla f^m)$ diverges almost surely. Since this is a property of the law $\pi(V)$, it is inherited for the general case in which $\pi_0 \ll \pi$.
\end{proof}

\subsection{Superdiffusivity}
\label{subsec:superdiffusivity}
Superdiffusive bounds are given in \cite{TarrTothValk12_DiffusivityBounds} for the SRBP in the smooth setting of Section \ref{sec:mollified-srbp}, and under certain assumptions on the infra-red behaviour of the interaction function $V$. The techniques considered there generalise {\it mutatis mutandis} to the present case, giving rise to Theorem \ref{thm:super}. In this subsection, we give a summary of the key ideas, the missing details can be found in \cite{TarrTothValk12_DiffusivityBounds}.

\begin{proof}[Proof sketch of Theorem \ref{thm:super}]
We begin with the observation that in the definition of $X$, the two terms on the RHS of \eqref{Xfrometa} are uncorrelated. This is due to a property known as Yaglom reversibility, which holds for the smooth SRBP, and the property persists in the limit. This implies that
$$ D(t) = t + \bbE\Big[\Big(\int_0^t \eta_s(0) \dd x \Big)^2\Big] $$
To show superdiffusivity, it is sufficient to consider the latter term.

For cylinder function $f \in \calC$, it holds, writing $\langle \cdot, \cdot\rangle$ in place of $\langle \cdot, \cdot\rangle_{\vfock}$
\begin{equation}\label{eq:41}
\int_0^\infty e^{-\lambda t} \bbE\Big[\Big(\int_0^t f(\eta_s) \dd x \Big)^2\Big] \dd t = 2 \lambda^{-2} \langle f, (\lambda - \calL)^{-1} f\rangle
\end{equation}
For the smooth case, this is a standard computation, and again it is a property that persists in the limit. By continuity, \eqref{eq:41} also holds in the case of $f = \delta_0$, and therefore it remains to control $\langle \delta_0, (\lambda - \calL)^{-1} \delta_0\rangle$.

For the upper bound, we use $\langle \delta_0, (\lambda - \calL)^{-1} \delta_0\rangle \leqslant \langle \delta_0, (\lambda - \calL_0)^{-1} \delta_0\rangle$. Write $M := \sup_{z \in \bbZ} \int_z^{z+1} \VHat(q) \dd q$, and let $R > 0$ be such that $\sup_{|p| \leqslant R} |p|^{-\alpha} \VHat(p) < C$ for some $C < \infty$. Then it holds
\begin{align*}
\langle \delta_0, (\lambda - \calL_0)^{-1} \delta_0\rangle &= \int_\bbR \frac{\VHat(q)}{\lambda + \thalf |q|^2} \dd q \\
&\lesssim_{C, M} \int_{|x| \leqslant R} \frac{|q|^\alpha}{\lambda + \thalf |q|^2} \dd q + \sum_{n = 0}^\infty \int_{R + n < |q| \leqslant R + n + 1} \frac{\VHat(q)}{|q|^2} \dd q \\
&\lesssim_{M, R} \lambda^{(\alpha-1)/2} + 1
\end{align*}
where we used $\sum_{n = 0}^\infty (n+R)^{-2} \lesssim_R 1$.

The lower bound is derived from the variational formula
$$ \langle \delta_0, (\lambda - \calL)^{-1} \delta_0\rangle = \sup_{f \in \vfock} \{ 2\langle\delta_0, f\rangle - \langle f, (\lambda - \calL_0)f\rangle - \langle \calA f, (\lambda - \calS)^{-1} \calA f\rangle \} $$
Considering test functions of the form $f = \ell_\phi$ for $\phi \in \calS(\bbR)$ yields a lower bound, and moreover $\calA_- \ell_\phi = 0$ in this case:
\begin{align}
\langle \delta_0, (\lambda - \calL)^{-1} \delta_0\rangle &\geqslant \sup_{\phi \in \calS(\bbR)} \{ 2\langle\delta_0, \ell_\phi\rangle - \langle \ell_\phi, (\lambda - \calL_0)\ell_\phi\rangle - \langle \calA_+ \ell_\phi, (\lambda - \calL_0)^{-1} \calA_+ \ell_\phi\rangle \} \nonumber \\
&\geqslant \tfrac{1}{4} \int_\bbR \frac{\VHat(p)} {\lambda + \thalf |p|^2 + |p|^2 J(\lambda, p)} \dd p \label{eq:45}
\end{align}
where the last line comes from explicit computations, a specific minimising choice of $\phi$, and features the inner-integral
$$ J(\lambda, p) := \beta^2 \int_\bbR \frac{\VHat(q)}{\lambda + \thalf |p+q|^2} \dd q $$

We claim that $\sup_{|p| \leqslant \lambda^{1/2}} J(\lambda, p) \lesssim \beta^2 \lambda^{(\alpha-1)/2}$ as $\lambda \rightarrow 0$. Indeed, continuing with the above definitions of $R, C, M > 0$, it holds
\begin{align*}
\int_\bbR \frac{\VHat(q)}{\lambda + \thalf |p+q|^2} \dd q &\lesssim_C \int_{|q-p| \leqslant R} \frac{|q-p|^\alpha}{\lambda + |q|^2} \dd q + \int_{|q-p| > R} \frac{\VHat(q)}{|q|^2} \dd q \\
&\leqslant \lambda^{(\alpha-1)/2} \int_{|q-p\lambda^{-1/2}| \leqslant R\lambda^{-1/2}} \frac{|q-p\lambda^{-1/2}|^\alpha}{1 + |q|^2} \dd q + \int_{|q-p| > R} \frac{\VHat(q)}{|q|^2} \dd q \\
&\leqslant \lambda^{(\alpha-1)/2} \sup_{|r| \leqslant 1} \int_{\bbR} \frac{|q-r|^\alpha}{1 + |q|^2} \dd q + \int_{|q| > R - \sqrt{\lambda}} \frac{\VHat(q)}{|q|^2} \dd q
\end{align*}
where in the last line we used the assumption $|p| \leqslant \lambda^{1/2}$ and the triangle inequality. We assume that $\lambda>0$ is small enough so that $R - \sqrt{\lambda} \geqslant \thalf R$, in which case, continuing from the above display, and using the finiteness of the first integral, it holds
\begin{align*}
\int_\bbR \frac{\VHat(q)}{\lambda + \thalf |p+q|^2} \dd q &\lesssim_{C} \lambda^{(\alpha-1)/2} + \int_{|q| > \thalf R} \frac{\VHat(q)}{|q|^2} \dd q \\
&\lesssim_{M, R} \lambda^{(\alpha-1)/2} + 1
\end{align*}
where again we used $\sum_{n = 0}^\infty (n+R)^{-2} \lesssim_R 1$.

Combining the claim with \eqref{eq:45} yields that as $\lambda \rightarrow 0$ it holds
$$ \langle \delta_0, (\lambda - \calL)^{-1} \delta_0\rangle \gtrsim_{R, C, M} \int_{|p| \leqslant \lambda^{1/2}} \frac{\VHat(p)} {\lambda + \lambda^{(\alpha-1)/2} |p|^2 } \dd p $$
Using the lower bound in \eqref{eq:42}, let $R' > 0$ be such that $\inf_{|p| \leqslant R'} |p|^{-\alpha} \VHat(p) > C'$ for some constant $C' > 0$. Plugging this, and assuming $\lambda > 0$ is small enough so that $\lambda^{1/2} < R'$, we obtain
\begin{align*}
\langle \delta_0, (\lambda - \calL)^{-1} \delta_0\rangle \gtrsim_{R, C, M, C'} \int_{|p| \leqslant \lambda^{1/2}} \frac{|p|^\alpha} {\lambda + \lambda^{(\alpha-1)/2} |p|^2 } \dd p \gtrsim \lambda^{-(1-\alpha)^2/4}
\end{align*}
This concludes the proof.

\end{proof}

\appendix
\section{Appendix}
\subsection{Proof of Proposition \ref{prp:iib}}\label{sec:proof-prop-iib}
\begin{proof}[Proof of Proposition \ref{prp:iib}]
Write $M := \sup_{z \in \bbZ} \int_z^{z+1} \VHat(q) \dd q$. Fix $R > 0$ and split the integral into regions $\{ |p+q| \leqslant R \}$ and $\{ |p+q| > R \}$. This yields
\begin{align*}
\int_\bbR \frac{\VHat(q)}{(\lambda + \thalf |p+q|^2)^s} \dd q &\lesssim \lambda^{-s} \int_{|q| \leqslant R} \VHat(q-p) \dd q + \sum_{n = 1}^\infty \int_{nR < |q| \leqslant (n+1)R} \frac{\VHat(q-p)}{|q|^{2s}} \dd q \\
&\lesssim_M \lambda^{-s} R + R^{1 - 2s} < \infty
\end{align*}
where we have used $\sum_{n = 1}^\infty n^{-2s} < \infty$. The claim $\lim_{\lambda\rightarrow\infty} J^s(\lambda) = 0$ follows from choosing $R = \sqrt{\lambda}$.
\end{proof}

\subsection{Proof of (\ref{approxabstractPDE})}\label{sec:proof-abstractpde}
We present the following standard argument that is required in the proof of Theorem \ref{existenceKBE}. 

\begin{proof}[Proof of \eqref{approxabstractPDE}]
For $\lambda > 0$, consider the Banach space $X=C_T \vsob{1}{2}$ with respect to the norm $\| F \|_X := \sup_{0 \leqslant t \leqslant T} e^{- \lambda t} {\| F (t) \|_{\vsob{1}{2}}}$. We begin by establishing the existence of a fixed point $F^m \in X $ of the map

$$\Phi : X \rightarrow X, \quad F \mapsto e^{\mathcal{L}_0 t}F_0 + \int_0^{\mathbf{t}} e^{\mathcal{L}_0 (t - s)} \mathcal{A}^m  F (s) \dd s$$
We have, for $\varepsilon > 0$,
\begin{align*}
{\| \Phi (F) - \Phi (\psi) \|_X}  &\leqslant \sup_{0 \leqslant t \leqslant T} \int_0^t e^{- \lambda s} \| e^{- (\lambda - \mathcal{L}_0) (t - s)} \mathcal{A}^m (F - \psi) (s) \|_{\vsob{1}{2}} \dd s\\
&= \sup_{0 \leqslant t \leqslant T} \int_0^t e^{- \lambda s} \| e^{-(\lambda - \mathcal{L}_0) (t - s)} ((\lambda - \mathcal{L}_0) (t - s))^{\varepsilon - \varepsilon} \mathcal{A}^m(F - \psi) (s) \|_{\vsob{1}{2}} \dd s\\
&\lesssim_{\varepsilon} \sup_{0 \leqslant t \leqslant T} \int_0^t e^{- \lambda s} \| ((\lambda - \mathcal{L}_0) (t - s))^{- \varepsilon} \mathcal{A}^m(F - \psi) (s) \|_{\vsob{1}{2}} \dd s\\
&\leqslant \sup_{0 \leqslant t \leqslant T} \lambda^{- \varepsilon} \int_0^t e^{- \lambda s} (t - s)^{- \varepsilon} \| \mathcal{A}^m(F - \psi) \|_{\vsob{1}{2}} (s) \dd s\\
&\lesssim_m \lambda^{- \varepsilon} \| F - \psi \|_X   \sup_{0 \leqslant t \leqslant T} \int_0^t (t - s)^{- \varepsilon} \dd s\\
&\lesssim_{\varepsilon, T} \lambda^{- \varepsilon} \| F - \psi \|_X .
\end{align*}
Hence, for $\lambda > 0$ large enough, we obtain a fixed point  of $\Phi$. Differentiating the respective fixed point equation in $t$ with convergence in $\vsob{-1}{2}$, we see that $F^m$ is indeed fulfils \eqref{approxabstractPDE}. Since $\mathcal{L}^m \in L(\vsob{1}{2}, \vsob{-1}{0})$, see \eqref{eq:34}, the latter also means that $F^m \in H^1_T \vsob{-1}{0}$.
\end{proof}

\subsection{Proof of Lemma \ref{regularityeta}}\label{sec:proof-regularity}

\begin{proof}[Proof of Lemma \ref{regularityeta}]
W.l.o.g. we may assume that $\pi_0=\pi(V)$. Note $\int w_{-\gamma}^r (x)\dd x<\infty$. Using that $\Delta_j\omega(x)$ is Gaussian for each $j$, we obtain
\begin{eqnarray}
\bbE \Big[ \sum_{j \leqslant 0} 2^{-sjr} \int w_{-\gamma}^r (x) | \Delta_j \omega (x) |^p \dd x \Big]  \notag &\simeq &\int w_{-\gamma}^r (x)\dd x  \sum_{j \leqslant 0} 2^{-sjr}  \Big( \int \hat{V} (p) | \varphi (2^{- j} p) |^2 \dd p \Big)^{r / 2}  \notag \\
&\lesssim & \sum_{j \leqslant 0}2^{-sjr}2^{rj(1+\alpha)/2} \notag ,
\end{eqnarray}
which is finite because $(1+\alpha)/2 - s > 0$.
\end{proof}

\endappendix
\section*{Acknowledgements}
The authors would like to thank G.~Cannizzaro for helpful feedback on this work. L.~G.\ thanks C.~Ling and A.~Martini for helpful discussions. H.~G. gratefully acknowledges financial support via the EPSRC Mathematical Sciences Doctoral Training Partnership EP/W523793/1. L.~G. gratefully acknowledges funding by DFG through EXC 2046, Berlin Mathematical School and through IRTG 2544 ``Stochastic Analysis in Interaction'' as well as the UKRI Future Leaders Fellowship, 2022 - ``Large-scale universal behaviour of Random Interfaces and Stochastic Operators'' MR/W008246/1 (G. ~Cannizzaro).

\bibliographystyle{Martin}
\bibliography{srbp_refs,srbp_refs_manual}

\end{document}